\documentclass[12pt,english]{article}
\usepackage{amsmath,pdfsync}
\usepackage{amsthm}
\usepackage{amsfonts}
\usepackage[T1]{fontenc}
\usepackage[utf8]{inputenc}
\usepackage{csquotes}
\usepackage{verbatim}
\usepackage{enumitem}
\usepackage{braket}
\usepackage{lmodern}
\PassOptionsToPackage{obeyFinal}{todonotes}
\usepackage{color}
\usepackage{textcomp}
\usepackage{mathtools}

\usepackage[capitalize]{cleveref}
\crefdefaultlabelformat{#2\textup{#1}#3}
\crefformat{equation}{\textup{(#2#1#3)}}

\numberwithin{equation}{section}
\numberwithin{figure}{section}
\numberwithin{table}{section}
\theoremstyle{plain}
\newtheorem{thm}{Theorem}[section]
\theoremstyle{definition}

\theoremstyle{plain}
\newtheorem{prop}[thm]{Proposition}
\crefname{prop}{Proposition}{Propositions}
\theoremstyle{plain}
\newtheorem{cor}[thm]{Corollary}
\theoremstyle{plain}
\newtheorem{lem}[thm]{Lemma}
\theoremstyle{remark}
\newtheorem{rem}[thm]{Remark}

\usepackage{latexsym}
\usepackage{amsmath,amsthm,esint}
\usepackage{amssymb,amsfonts}
\usepackage{cases}
\numberwithin{equation}{section}

\def\R{\mathbb{R}}
\def\Rm{\mathbb{R}}

\hsize \textwidth
\advance \hsize by -\marginparwidth
\evensidemargin \oddsidemargin
\usepackage{amssymb}
\advance\hoffset by 5mm

\usepackage{amsmath,amsthm,amsfonts,amsopn,verbatim,graphicx,amssymb}
\chardef\bslash=`\\ % p. 424, TeXbook
\theoremstyle{remark}
\newcommand{\farc}{\frac}
\def\R{\mathbb{R}}
\def\S{\mathbb{S}}

\usepackage{amscd,amsthm}
\usepackage{}

\newcommand{\eval}[2][\right]{\relax
  \ifx#1\right\relax \left.\fi#2#1\rvert}

\title{Chemotaxis and reactions in biology}

\author{Alexander Kiselev
\thanks{Department of
Mathematics, Duke University, Durham, NC 90320, USA; email: kiselev@math.duke.edu}
\and Fedor Nazarov
\thanks{Department of Mathematical Sciences, Kent State University, Kent OH 44242, USA; email: nazarov@math.kent.edu}
\and Lenya Ryzhik
\thanks{Department of Mathematics, Stanford University, Stanford CA 94305, USA; email: ryzhik@stanford.edu}
\and Yao Yao
\thanks{School of
Mathematics, Georgia Institute of Technology, Atlanta, GA 30332, USA; email: yaoyao@math.gatech.edu}
}
\begin{document}

\maketitle

%PRELIMINARY DRAFT, TITLE TENTATIVE

\begin{abstract}
Chemotaxis plays a crucial role in a variety of processes in biology and ecology. Quite often it acts to improve efficiency of biological reactions.
One example is the immune system signalling,
where infected tissues release chemokines attracting monocytes to fight invading bacteria.
Another example is reproduction, where eggs release pheromones that attract sperm.
A macro scale example is flower scent appealing to pollinators.
In this paper we consider a system of PDE designed to model such processes. Our interest is to quantify
the effect of chemotaxis on reaction rates compared to pure reaction-diffusion. We limit consideration to surface chemotaxis, which is well motivated from the point
of view of many applications. Our results provide the first insight into situations where chemotaxis can be crucial for reaction success,
and where its effect is likely to be limited.
The proofs are based on new analytical tools; a significant part of the paper is dedicated to
building up the linear machinery that can be useful in more general settings.
In particular we establish precise estimates on the rates of convergence to ground state
 for a class of Fokker-Planck operators with potentials that grow at a logarithmic rate at infinity.
 These estimates are made possible by a new sharp weak weighted Poincar\'e inequality improving in particular a
result of Bobkov and Ledoux \cite{bl}.
\end{abstract}

\maketitle

\section{Introduction}

Chemotaxis describes the motion of cells or species that sense and attempt to move
towards higher (or lower) concentration of some chemical.  Its first mathematical
studies go back to Patlak \cite{Patlak} and
Keller-Segel \cite{KS1}, \cite{KS2}.
The Keller-Segel system introduced in the latter work describes a population of bacteria
or mold secreting an attractive chemical substance, and remains
the most studied model of chemotaxis.
In the simplified parabolic-elliptic form, this equation can be written as
(see, e.g. \cite{Pert})
\begin{equation}\label{chemo1}
\partial_t \rho - \Delta \rho +
\chi \nabla\cdot(\rho \nabla (-\Delta)^{-1} \rho)=0, \,\,\,\rho(x,0)=\rho_0(x).
\end{equation}
The last term in the left side describes the attraction of $\rho$ by
a chemical with the concentration~$c(x,t) = (-\Delta)^{-1} \rho(x,t).$
This is an approximation to the diffusion equation
\[
\partial_tc=\kappa\Delta c+R\rho,
\]
under the assumption that $\kappa\sim R\gg 1$, so that
the chemical is both produced and diffuses on
faster time scales than those for the rest of dynamics of~\eqref{chemo1}. The literature on the
Keller-Segel equation is very extensive. In particular,
a number of different variants of \eqref{chemo1} have been derived from more basic kinetic
models (see, e.g. \cite{HO1,OH1,DS,JV1,PVW}).
It is known that in dimensions larger than one solutions to \eqref{chemo1} can
concentrate and become singular in a finite time. We refer to
\cite{h,Hor2,Pert}  for more details and further references.

In many settings in biology where chemotaxis is present, it facilitates and enhances
success rates of reaction-like processes.
One example is reproduction for many species, where eggs secrete chemicals that attracts sperm
and help improve fertilization
rates. This is especially well studied for marine life such as corals, sea urchins, mollusks,
etc (see \cite{HRZZ,RZ,ZR} for further
references), but the role of chemotaxis in fertilization extends to a great number of
species, including humans~\cite{Raltetal}.
In the same vein, many plants appeal primarily to the insects' sense of smell to attract pollinators.
Another process where chemotaxis plays an important role is mammal immune systems fighting bacterial infections.
Inflamed tissues release special proteins, called chemokines,
that serve to chemically attract monocytes, blood killer cells, to the source
of infection   \cite{Desh}, \cite{Taub}.
Chemotaxis can also be involved when things go awry, for instance,
playing a role in tumor growth \cite{VanCoil}.

In the mathematical literature, the studies of equations including both
chemotaxis and reactions focused mainly on existence and regularity
of solutions as well as general features of the long time dynamics (see  \cite{cpz,ESV,EW,MT,OTYM,TW,W1,W2,Winkler3} for further references).
To the best of our knowledge, there are very few works where the question of how chemotaxis affects
the reaction rates has been studied rigorously or even modeled computationally.
As far as we know, the first step in this direction has been taken in \cite{kr}, \cite{kr2}
where a generalization of \eqref{chemo1} including an
absorbing reaction and a fluid flow has been considered
\begin{equation}\label{chemo2}
\partial_t \rho +(u \cdot \nabla)\rho - \Delta \rho +
\chi \nabla\cdot(\rho \nabla (-\Delta)^{-1} \rho)
=-\epsilon \rho^q, \,\,\,\nabla \cdot u=0, \,\,\,\rho(x,0)=\rho_0(x) \geq 0.
\end{equation}
This work was motivated by modeling the life cycle of corals. Corals, and many other
marine species, reproduce by broadcast spawning. It is a fertilization strategy whereby males
and females release sperm and egg gametes  that rise to the surface of the ocean.
As they are initially separated by the ambient water, an effective surface mixing is necessary
for a successful fertilization. For coral spawning, field measurements of fertilization
rates are usually around~$50$\%, and are often as high as $90$\%~\cite{lasker,pennington}.
On the other hand, numerical simulations based on purely reaction-diffusion models \cite{ds}
predict fertilization rates of less than~$1$\% due to the strong dilution of gametes.
A more sophisticated model, taking into account the instantaneous details of the advective
transport was proposed in \cite{ccw,chw}.
Adding fluid flow to the model can account for part of the gap between simulations and field
measurements, but appears unlikely to completely explain it \cite{kr}.
However, as we already mentioned, there is also experimental evidence that chemotaxis plays a
role in coral and other marine animals fertilization: eggs release a chemical that attracts the
sperm~\cite{coll1,coll2,miller,miller2}.

The results of \cite{kr,kr2} show, in the framework of \eqref{chemo2}, that the role of
chemotaxis in reaction enhancement can be quite significant --
especially when reaction is weak, as is known to be the case in many biological
processes~\cite{VCCW}. The efficiency of the reaction can be measured by the decay of the
total mass of the remaining density
\[
m(t)=\int \rho(x,t)\,dx.
\]
If $\chi=0,$ then the decay of $m(t)$ is very slow if $\epsilon$ is small,
uniformly in the incompressible fluid velocity $u$ \cite{kr}.
On the other hand, if $\chi \ne 0$,
then in dimension two, relevant for the corals application,
the extent of decay and time scales of decay of $m(t)$ are independent
of $\epsilon$, and the decay can be very significant and fast if the
chemotactic coupling is sufficiently strong.
While the results of \cite{kr} and \cite{kr2} are suggestive, taking \eqref{chemo2} as a model
makes a strong simplifying assumption  that the densities of male and female species
are equal and are both chemotactic on each other. In reality, only the male density is
chemotactic, hence \eqref{chemo2} can be expected to overestimate the effect of chemotaxis
on the reaction rates.

Although there are certainly examples of mold and bacteria that are chemotactic on the chemicals
they themselves release, significantly more numerous situations in biology involve species
that are chemotactic on a chemical secreted by other agents. Most of the examples mentioned above
are of this kind.  In this paper, we would like to initiate qualitative
analysis of a more realistic system of equations modeling
chemotaxis enhanced reaction processes, of the form
%As a first step, we study the following model
%In this paper, we study the following system of equations, which arises from reproduction processes in biology:
\begin{equation}
\begin{aligned}
&\partial_t \rho_1 - \kappa \Delta \rho_1 + \chi \nabla\cdot(\rho_1 \nabla(-\Delta)^{-1} \rho_2) = -\epsilon \rho_1 \rho_2\\
&\partial_t \rho_2 = -\epsilon \rho_1 \rho_2.
\end{aligned} %\right.
\label{eq:before_normalization}
\end{equation}
There is no ambient fluid advection:
as the first step, we assume that the fluid flow is adequately modelled by effective diffusion.
The chemically attracted density is $\rho_1;$ the density $\rho_2$ that produces the attractant is assumed to be immobile, which is a realistic assumption
%While this is not a very good assumption for the coral modeling, where both densities are advected by the flow and diffuse,
%This is not an unreasonable first step even in the coral modeling setting, as eggs are larger and less mobile than sperm.
in many interesting problems: for example, the
inflamed tissue releasing chemokines and attracting monocytes,
plants attracting insects, or immobile eggs attracting sperm in the mammal reproduction tract are in this category.
We also maintain the parabolic-elliptic structure, with the assumption that the
signaling chemical diffusion time is much shorter than other relevant time scales.
The system \eqref{eq:before_normalization} is one of the most natural first step models in analyzing any situation where a fixed target aims
to attract, by using a fast diffusing chemical, a diffusing and mobile
species which is involved in some kind of reaction with the target.
Systems of this type have been certainly analyzed in the literature - for example, in \cite{cpz} a system of a very similar form but
with different chemotactic term has been considered as a model of angiogenesis.
However, the focus of most such studies has been on proving global regularity, asymptotic behavior and finding special classes of self-similar solutions.
Perhaps the closest to our aim here are the papers \cite{ES1,CKL} that yield some estimates on the effect of chemotaxis on reaction in a related setting.
%like a reproduction process or monocytes fighting infection.
%While systems involving chemotaxis have been extensively studied, the main
%questions addressed have been existence and regularity of solutions or general features of large time dynamics (see e.g. \cite{cpz}, \cite{ESV}, \cite{Winkler3} where
%further references may be found). %The corals story and the naive model \cite{KR1}, \cite{KR2} suggest that role of chemotaxis in a broad range of biological reactions can be very
%significant.
However, to the best of our knowledge, our paper is the first attempt
at sharp qualitative estimates for the scaling rules of the effect of chemotaxis on the reaction rates
in a setting of chemotaxis system involving two distinct densities.
%and even this first step is non-trivial.
%fairly sophisticated and rich new mathematics is involved even for this first step.
Here, we will limit the consideration to two spatial dimensions and to the classical form
of the Keller-Segel chemotaxis flux.
We make comments on some possible extensions and generalizations in Section~\ref{Disc}.

%In this paper, we will consider the system \eqref{eq:before_normalization}
%in two dimensions. We will regard $\epsilon$ as a small parameter, as it is known that in many cases biological reactions are weak. In the case of reproduction
%processes, this is dictated by the fact that sperm needs to hit a small part of the egg's surface for fertilization to be successful \cite{VCCW}. This is a mechanism
%preventing simultaneous fertilization by several sperm.

The purpose of this paper is twofold. First, we provide a careful analysis of the linear problem
corresponding to \eqref{eq:before_normalization}. This analysis is interesting in its own right, and focuses
on a class of Fokker-Planck operators with logarithmic potentials that is very natural especially in dimension two.
This linear problem models convergence of a density attracted by a fast diffusing chemical to a target that releases it.
Secondly, we present an initial nonlinear application of the techniques we develop which also involves the reaction term.
In the nonlinear case, this paper focuses on the radial setting and develops a general framework for applying the
linear techniques for analysis of reaction rates. Generalizations to more general settings will be addressed in future work;
the Section~\ref{Disc} outlines some of the avenues that we expect to pursue. An interesting by-product of our work is
a suggestion that the traditional Keller-Segel term may be ill-suited to accurately modeling reaction enhancement effects,
and a so-called flux-limited version may be more appropriate. This is also discussed in more detail below and in Section~\ref{Disc}.

To describe our main results, we begin from the nonlinear application that will motivate the linear problem.
%system \eqref{eq:before_normalization}
For the sake of simplicity, we assume that the initial condition for $\rho_2$
is compactly supported and smooth:
% We could take $\rho_2(\cdot,0)$ equal to the characteristic function
%of a disc, but lack of regularity introduces some unnecessary technical details. We will
%therefore take
$\rho_2(x,0) = \theta \eta(x),$ where $\theta$ is a coupling constant, and
$\eta(x) \in C_0^\infty(\R^2)$ is close to the characteristic function
of the disc $B_R$ centered at the origin in the $L^1$ norm -- obviously, we can make
it as close as we want. It is useful to re-scale \eqref{eq:before_normalization};
by a space-time rescaling we can normalize the parameters $\kappa$ and $R$,
so that (\ref{eq:before_normalization}) becomes
%Namely, by scaling time and space, we can arrange for diffusion coefficient and radius $l$ to be equal
%to one.
%We then apply a change of variable $\tilde \rho_2 := \chi\rho_2 $ to normalize the coefficient for the chemotaxis term, and as a result the new reaction coefficient becomes
%$\tilde \epsilon := \epsilon/\chi$. For notational simplicity, after such normalization we still denote the reaction coefficient and the second density by $\epsilon$ and $\rho_2$.
% The scaling gives different reaction couplings in the equations, since in one we take time derivative of \rho_1 (not nomralized), and in the other of \rho_2.
%From now on, we focus on the normalized system
\begin{equation}
%\left\{
\begin{aligned}
&\partial_t \rho_1 - \Delta \rho_1 +  \chi \nabla\cdot(\rho_1 \nabla(-\Delta)^{-1} \rho_2) = -\epsilon \rho_1 \rho_2\\
&\partial_t \rho_2 = -\epsilon \rho_1 \rho_2,
\end{aligned} %\right.
\label{eq:original_system}
\end{equation}
where for simplicity we keep the same notation for variables and parameters.
The connection between parameters before and after rescaling will be documented after Theorem~\ref{main1} below.
The initial condition for $\rho_2$ has the form $\rho_2(x, 0) := \theta \eta(x)$,
with some $\theta>0$ and radial
$\eta \in C_0^\infty(\Rm^2)$, such that  $\eta(x)$ is close
in $L^1$ to the characteristic function $\chi_{B_1}(x)$ of the unit disk, with
\[
0\le \chi_{B_1}(x)\le\eta(x) \le 1.
\]
It is straightforward to extend our results to more general radial initial data
$\rho_2(x,0) \in C_c^\infty(\Rm^2)$ or just rapidly decaying.
For the initial condition $\rho_1(x, 0) \geq 0$ for (\ref{eq:original_system}),
we assume that it is smooth and decaying quickly at infinity, and is located at a distance $\sim L$ from
the origin. Specifically, we will assume that its mass in a ball $B_L(0)$ is at least~$M_0$ while the mass inside
$B_1$ is much smaller than $M_0:$
\begin{equation}\label{oct714}
\int_{|x|\le L}\rho_1(x,0)dx\ge M_0, \,\,\,\int_{|x|\le 1}\rho_1(x,0)dx \ll M_0.
\end{equation}
 % -- in the already re-scaled system of coordinates.
Thus, $M_0,$ $L,$ $\theta,$ $\chi$ and $\epsilon$ are the parameters left in the problem,
and  it is convenient to combine the mass of $\rho_2$ that is $\sim \theta$ and $\chi$ into a single
parameter~$\gamma := \theta \chi.$
We are primarily interested in the situations where $M_0$ is large,
so that~$M_0\epsilon \gg \gamma \gg 1$ and~$M_0 \gg \theta$; the motivation for such relationship between the
parameters will be discussed below.
%with $M\gg \gamma \gg 1$.
Our goal is to compare the efficiency of reaction,
%In most applications $M$ is much larger than $\gamma.$
%The efficiency of reaction translates into
that is, the decay rate of the integral
%(or $L^1$ norm, which is the same since the density is non-negative)
\[
\int_{\R^2} \rho_2(x,t)\,dx,
\]
with and without chemotaxis.
A reasonable measure of the reaction rate is a typical "half-time" scale  during
which about half of the initial mass $\sim \theta$ of $\rho_2$ will react. More precisely,
we will say that half-time $\tau_C$ is the time by which the mass of $\rho_2$ decreases by the
amount~$\pi\theta/2$.
%We will call this time ``half-time".
%We are looking at $\rho_2$ and not $\rho_1$ since $M_0 \gg \theta$ and the mass $\int
%\rho_1\,dx$ is not going to become small.
%We will prove the following theorem.
Our main nonlinear application is
\begin{thm}\label{main1}
Assume that the initial conditions $\rho_1(\cdot,0)$ and $\rho_2(\cdot,0)$
are as above and, in addition, radially symmetric.
Let $\frac{\chi \gamma}{\epsilon} \geq c>0.$
There exists $B>0$ sufficiently
large that depends only on $c$, so that if
%Suppose that
\begin{equation}\label{oct702}
\farc{M_0\epsilon}{\gamma},\gamma,\farc{M_0}{\theta} \geq B,
%M_0 \epsilon \gg \gamma \gg 1 \gg \epsilon,\hbox{ and $M_0 \gg \theta.$ }
\end{equation}
then the half-time for the solution of the system \eqref{eq:original_system} satisfies
\begin{equation}\label{tauCfin}
\tau_C \lesssim \frac{L^2}{\gamma}+ \log \gamma.
\end{equation}
On the other hand, if $\chi=0$ and $\rho_1(x,0)$
is supported in $\{|x|\ge L/2\}$, then
the pure reaction-diffusion half-time
satisfies $\tau_D \gtrsim L^2/\log(\epsilon M_0).$
\end{thm}
\noindent\it Remarks. \rm
1. Note that time-space rescaling leading from \eqref{eq:before_normalization} to \eqref{eq:original_system} is given by
$x'=x/R,$ $t'=tR^2/\kappa.$ The new parameters are given by $\chi'=\chi R^2/\kappa,$ $\epsilon'=\epsilon R^2/\kappa,$
$M_0'=M_0/R^2,$ $L'=L/R,$ and $\gamma' = \theta \chi' = \theta R^2 \chi/\kappa.$ As we mentioned above, after the change of variables, we reverted to
denoting new parameters without primes. The conditions \eqref{oct702} in the original parameters take form
$M_0 \epsilon/(\theta R^2 \chi)\geq B,$ $\theta R^2 \chi /\kappa \geq B,$ $M_0/(\theta R^2) \geq B.$  Here $\theta R^2$ $\sim$
initial mass of $\rho_2.$ \\
%While we need the first three ratios to be large, the last one just indicates that the constant
%with the exception of the $\gamma \geq C$ condition, which reads $\gamma/\kappa \geq C.$ \\
2. The  assumption (\ref{oct702})
%We note that the stated relationship between parameters $M_0,$ $\gamma,$ and $\epsilon$
is reasonable in many applications.
For example, in coral spawning, a typical number of sperm is of the order $\sim 10^{10},$
the number of eggs $\sim 10^6$, and~$\epsilon \sim 10^{-2}.$
It is difficult to find data on the measurements of strength of chemotactic coupling in biological literature. \\
%2. We use the notation $\gg$ in (\ref{oct702})
%in the sense that there exists a universal constant $C$ such that
%if~$M_0\epsilon \geq C\gamma$, $\gamma \geq C$, $\epsilon\ll C^{-1}$, and
%$M_0 \geq C\theta,$ then the conclusion of the theorem is true. \\
3. The notation $\lesssim,\gtrsim$ and $\sim$ means, as usual, bounds with universal constants independent of the key parameters of the problem. \\
%In particular, the constant appearing in \eqref{tauCfin} will depend on the constants in \eqref{oct702}.
4. In Section~\ref{trancompsec}, we prove Theorem~\ref{mainalt1}, a variant of Theorem~\ref{main1}, that eliminates the $\log \gamma$ term
in \eqref{tauCfin} at the price of providing less precise information about the dynamics of the system. \\

We believe that, possibly up to a logarithmic in $\gamma$ correction, the result of Theorem~\ref{main1} is sharp. %The regime $M \epsilon >> 1$ is most natural from the point of view of applications.
%Thus Theorem~\ref{main1}
It provides an indication that %in what kind of situations
the presence
of chemotaxis can significantly improve reaction rates if
%in terms of the re-scaled parameters of \eqref{eq:original_system}
%this happens if
%$e^{\gamma} \gg M_0\epsilon.$ %For convenience, let us introduce notation $\gamma = \chi
%\kappa.$
$\gamma\gg \log(M_0\epsilon)$.
In particular, in the framework of  \eqref{eq:original_system}, one can expect chemotaxis
to provide significant improvement only if $\gamma$ is sufficiently large.

There are natural further questions discussed in some detail in Section~\ref{Disc}.
Here, let us just comment on the radial assumption on the initial data. The technical
reason behind this condition is an artifact of the Keller-Segel form of chemotaxis term.
%where the speed of directed motion is proportional to the gradient of
As the chemical concentration is $(-\Delta)^{-1}\rho_2$, the $\rho_1$ species concentrates
%In particular,
near the center of the support of $\rho_2$, and, in general, it may arrive there without
ever meeting $\rho_2$, so that reaction is not enhanced at all. This is prohibited in
the radial geometry where $\rho_1$-species will have to see $\rho_2$ as they move toward
the origin.
%this leads to a very large drift and high concentration of $\rho_1,$ depleting the value of the reaction by lack
%of coverage by $\rho_1$.
%More realistic models of chemotaxis incorporate saturation effect limiting the
%speed of the drift,
%leading to the so-called flux-limited chemotaxis (see e.g. \cite{BW,DS,HP,PVW} for further references).
We expect that
the techniques developed in this paper should apply to other chemotactic models and to a broader
class of initial data configurations, with
%In this paper, we decided to stay with the classical
%Keler-Segel form of chemotaxis term, and the extensions will be considered elsewhere; this will
%be discussed in more detail in Section~\ref{Disc}.
Theorem~\ref{main1} as an initial application.
%The main goal of this paper
%is to develop main elements of the techniques
%suitable for such application which will have broader applicability.
%On the other hand, as will be explained in more detail Section~\ref{zex}, one cannot expect that chemotaxis
%is helpful in arbitrary settings: there exist initial data geometries and values of parameters where presence of chemotaxis will actually slow down the reaction process.
%This appears to be true both for classical Keler-Segel model and for the models incorporating saturating effects.

The proof of Theorem~\ref{main1} relies on several ideas. We expect that the main positive
effect of chemotaxis is in speeding up transport of the species $\rho_1$ towards the origin
where the species~$\rho_2$
is concentrated. To capture this, we estimate the transport
stage by comparing the solutions of the coupled system to
the solutions of the linear Fokker-Planck equation
with a properly chosen time-independent potential %. The latter equation has form
\begin{equation}\label{fokker11}
\partial_t \rho - \Delta \rho + \nabla\cdot (\rho \nabla H) =0.
\end{equation}
One would wish to take $\rho(x,0) = \rho_1(x,0),$ and $H = (-\Delta)^{-1}\rho_2.$
However, the time dependence of $H$ would complicate the analysis. Instead, we use
a comparison to the solution
to \eqref{fokker11} with "the weakest" attractive potential $H(x)$ in an appropriate class.
The operator
\[
F_H\phi = -\Delta\phi + \nabla\cdot (\phi \nabla H),
\]
appearing in \eqref{fokker11} is
self-adjoint and non-negative on the weighted space $L^2(e^{-H},dx)$ and, if $\gamma$ is
sufficiently large, has a ground state $e^H.$ The rate of convergence of
the solution to the ground state for large times corresponds to
transport of the density $\rho$ from far field towards the region with higher values of $H(x).$
As we will see, the worst case potential is %given by
\begin{equation}\label{potential}
H(x) = \gamma (-\Delta)^{-1} \big( \chi_{B_1}(x) - \chi_{B_{1/\sqrt{2}}}(x) \big).
\end{equation}
%Here, $\chi_S(x)$ is the characteristic function of a set $S.$
It is not difficult
to see that in dimension two, $H(x)\approx - (\gamma \pi/{2}) \log |x|$ for $|x|\gg 1,$
and we need to deal with a Fokker-Planck equation with a logarithmic potential.
We stress that all estimates we prove for the linear problem \eqref{fokker11} apply
in full generality, without radial constraint on $f.$

Thus, our principal goal in this paper is to provide precise bounds on the rate of
convergence to the ground state for this class of Fokker-Planck operators, and to develop
a comparison scheme to use these estimates in the analysis of nonlinear problems.
The rate of convergence to an equilibrium for Fokker-Planck operators is a classical
subject, and
%as it describes a general parabolic-elliptic setting with chemical attraction
%of density to a stationary chemical-releasing target in dimension two.
the literature on this question %on rates of convergence to ground state
is vast.
The uniformly convex case $-D^2H(x)  \geq \lambda \hbox{Id}$ with $\lambda>0$,
can be viewed as a direct application
of Brascamp-Lieb ideas \cite{BrL}, and the operator $F_H$ has a spectral gap, so that
convergence to the ground state is exponential in time. There has been much work on generalizations of these results. An extension to, in particular,
$H(x) = |x|^\beta$ with $1<\beta<2,$ and further references can be found in \cite{amtu}.
%for more references and discussion.
%(see Bakry and Emery \cite{be}, \cite{BE2}, \cite{BE3}). In this case, the operator $F_H$ has a spectral gap and convergence to the ground state is exponential in time. There %has been much
%work on generalization of these results. See, for example, \cite{amtu} for an extension that works in particular for $H(x) = |x|^\beta$ with $1<\beta<2,$ and for more references %and discussion.
For slower growth potentials there can be no spectral gap.
R\"ockner and Wang \cite{RW} provide convergence to equilibrium estimates
for $H(x) = |x|^\beta$ with $0<\beta <1$ which are sub-exponential in time,
as well as algebraic in time convergence bounds for a logarithmic potential
-- which is precisely our case. However, the dependence of these bounds on the coupling constant
is not sufficiently sharp for the applications that motivate us.
There is also related work based on probabilistic techniques, in particular,
by Veretennikov~\cite{PVer,Ver}. These estimates are designed with different
applications in mind, and are also not sufficient for our purpose.

While weighted Poincar\'e inequalities can be used to prove exponential in time convergence
to equilibrium for the Fokker-Planck operators, the tools that can be deployed when
the rate of convergence is slower are called weak Poincar\'e or Poincar\'e-type inequalities.
An inequality of this kind involving power weights has been proved by Bobkov
and Ledoux~\cite{bl}.
% (of course $v$ depends on $\gamma$ but we suppress this in notation).
That paper contains, in particular, the following inequality for every $f \in C_0^\infty(\R^d)$
\begin{equation}\label{bobl1}
\int_{\R^d} |f-\overline{f}|^2 v(x)\,dx \leq \frac{C}{\gamma}
\int_{\R^d} |\nabla f|^2 (1+|x|^2)v(x)\,dx,
\end{equation}
with the weight $v(x) = (1+|x|^2)^{-\gamma/2}$ for some sufficiently large $\gamma$, and
\[
\overline{f}=\int_{\Rm^d}f(x)v(x)dx.
\]
%Here, $\overline{f}$ is the average of the function $f$ with respect to the weight $v(x).$
The proof of Bobkov and Ledoux is based on convexity techniques, and builds on generalizations
of the Brascamp-Lieb inequality \cite{BrL}. For our application, we need a version
of \eqref{bobl1} with the weight equal to $w(x) = e^H.$ While the behavior of $w(x)$
and $v(x)$ near infinity is virtually
identical, the weight $w(x)$ does not seem to satisfy the convexity assumptions needed for
the techniques of \cite{bl} to work. Moreover, the factor $C/\gamma$ in the right side
of \eqref{bobl1} would lead to sub-optimal estimates on the rate of convergence to the ground state.
One could verify that such estimate
could only yield $\tau_C \lesssim L^2$ in Theorem~\ref{main1}.
This is not very interesting, since pure reaction-diffusion is not outperformed in relevant regimes. We prove the following improved weighted Poincar\'e-type inequality
by differentiating the regions where behavior of the weight $w$ is different. % that improves on \cite{bobl1}.
\begin{thm}\label{main2}
Let $\gamma > 2,$ $f \in C_0^\infty(\R^2),$ and $w(x)=e^H,$ with
$H$ given by \eqref{potential}. Then the following weak weighted Poincar\'e inequality holds:
\begin{equation}\label{wP2}
\int_{\R^2} |f - \overline{f}|^2 w(x)\,dx \leq C \int_{B_1} |\nabla f|^2 w(x)\,dx + \frac{C}{\gamma^2}
\int_{B^c_1} |\nabla f|^2 (1+|x|^2) w(x)\,dx.
\end{equation}
\end{thm}
The bound \eqref{wP2} provides an improvement from $\gamma^{-1}$ to to $\gamma^{-2}$ factor
in the far field that is crucial for our application. It is not difficult to build examples
to show that such scaling is sharp. We will prove a further refinement of
Theorem~\ref{main2} which is a bit too technical to state in the introduction;
it works in any dimension and for a broader class of weights,
 %.The result of Theorem~\ref{main2} applies to a much broader class of
%weights,
including the straight power weight $v(x)$, for which
%In the case of $v(x),$ the inequality would
it takes the form
\begin{equation}\label{bobimp}
\int_{\R^d} |f-\overline{f}|^2 v(x)\,dx \leq
\frac{C(d)}{\gamma} \int_{B_{1}} |\nabla f|^2 v(x)\,dx + \frac{C(d)}{\gamma^2}
\int_{(B_1)^c} |\nabla f|^2 (1+|x|^2) v(x)\,dx
\end{equation}
for all sufficiently large $\gamma.$
%Hence the weaker factor $C/\gamma$ is only needed in front of an integral over a unit ball centered at the origin.
Our proof of Theorem~\ref{main2} is based on direct analytic estimates.

The paper is organized as follows. In Section~\ref{heur}, we provide an heuristic
motivation for the main application result. In Section~\ref{sec:Linfty},
we sketch the proof of the
global well-posedness for~\eqref{eq:original_system}, along with an $L^\infty$-bound
on the density $\rho_1.$
In Section~\ref{sec:comparison_principle}, we discuss the mass comparison principles,
which will allow the estimates for the linear Fokker-Planck equations with a time
independent potential to be useful for the nonlinear analysis.
In Section~\ref{sec:wwpti}, we derive new weak weighted Poincar\'e inequalities, in particular
proving Theorem~\ref{main2}, and in Section~\ref{sec:fp} use these inequalities to obtain
estimates on the rates of convergence to ground state for the Fokker-Planck operators
with logarithmic-type
potentials. %In particular, in Subsection~\ref{subsec:poincare} we establish a class of weighted Poincar\'e estimates that may be of independent interest.
In Section~\ref{trancompsec}, we provide a brief detour and show how to set up a version of Theorem~\ref{main1}, Theorem~\ref{mainalt1},
using only comparison principles and avoiding the analysis of Fokker-Planck equation. This argument is much simpler,
and generates result similar to our main application here. However, it  provides limited
information on distribution of $\rho_1$ near target support, that may be useful in other applications,
and does not yield intuition explaining limitations of the standard Keller-Segel chemotaxis term that are leading to our radial assumption.
In Section~\ref{system}, we apply the results proved in previous sections to finalize the proof Theorem~\ref{main1} and Theorem~\ref{mainalt1}.
%In Section~\ref{zex}, we present an example showing that in some initial data configurations, chemotaxis can be in fact slow down the reaction process.
In Section~\ref{Disc}, we provide a preview
of more advanced applications that we believe may be possible using the developed techniques.

Throughout the paper, we will by denote $\|f\|_{p}$ the $L^p(\Rm^d)$-norm
of the function $f$ with respect to Lebesgue measure.

\section{Heuristics}\label{heur}

%Let $\tau$ be the time it takes for $\|\rho_2(\cdot, t)\|_{L^1}$ to drop by a half,
%that is, the ``half-time'' of the reaction.
% The reason we do not say exactly by a half and instead use ``about a half" language is as follows.
%For simplicity, we could take the initial condition $\rho_2(x,0)$ to be
%the characteristic function of the unit disk $B_1(0).$ However, the lack of regularity
%of such initial condition together with the lack of smoothing in the second equation
%in \eqref{eq:original_system} creates unnecessary technical issues. For that reason, as
%mentioned in the introduction,
%we take $\rho_2(x,0) =  \theta \eta(x),$ where $\eta(x) \geq \chi_{B_1(0)}(x)$ is a
%$C_c^\infty(\Rm^2)$ function that is close to
%$\chi_{B_1(0)}(x)$ in $L^1$-norm.
 %On the other hand, when choosing the potential $H$ for the Fokker-Planck operator \eqref{potential}, it will be convenient for us to choose it corresponding
 %to the density that has the total mass exactly equal to $1/2.$ Of course, we could always change $H$ a little for the corresponding mass to be under $1/2$ and all
 %the results would remain true (up to adjustment in universal constants).
In order to tell whether the chemotaxis term can enhance reaction, it suffices to compare the
half-times $\tau_C,$ $\tau_D$ in the two systems, with and without chemotaxis, respectively.
In Section~\ref{subsec:diff}, we will derive a \emph{rigorous} lower bound for $\tau_D$
in the absence of chemotaxis. We then give a \emph{heuristic} argument for the full system
in Section \ref{sec:heuristics}, formally deriving an upper bound for $\tau_C$ in the
presence of the chemotaxis term. Comparing with the estimate without chemotaxis,
it suggests that in a certain parameter regime,
chemotaxis  should significantly shorten the half-time, thus
meaningfully enhancing the reaction between the two densities. Of course,
the upper-bound for $\tau_C$ in the system with chemotaxis is just formal at this moment,
but it will be made rigorous in the rest of this paper in the radially symmetric case.

\subsection{Estimates in the purely diffusive case}\label{subsec:diff}

Consider the system without  chemotaxis:
\begin{equation}
%\left\{
\begin{aligned}
 &\partial_t \rho_1 -\Delta \rho_1  = -\epsilon \rho_1 \rho_2\\
 &\partial_t \rho_2 = -\epsilon \rho_1 \rho_2,
\end{aligned} %\right.
\label{eq:diffusion_system}
\end{equation}
where the initial conditions are the same as for the original system \eqref{eq:original_system}.
The time $\tau_D$ it takes for~$\|\rho_2(\cdot, t)\|_{L^1}$ to drop by a half obeys a lower
bound
\begin{equation}\label{oct712}
\tau_D\ge\tau.
\end{equation}
Here, $\tau$ is the time it takes
for $\|g_2\|_{L^1}$ to drop in half, where $g_2$ is the solution to
%One can obtain a lower bound of $\tau_D$  by comparing~\eqref{eq:diffusion_system}
%to the following system
\begin{equation}
\left\{
\begin{array}{ll}
\partial_t g_1 & =\Delta g_1 \\
\partial_t g_2 &= -\epsilon g_1 g_2,
\end{array}\right.
\label{eq:diffusion_system_2}
\end{equation}
where $g_1$ and $g_2$ have the same initial data as $\rho_1$ and $\rho_2$ respectively.
Indeed, the comparison principle implies that $\rho_1(\cdot, t) \leq g_1(\cdot, t)$
for all $t\geq 0$,
so that $\rho_2(\cdot, t) \geq g_2(\cdot, t)$, and (\ref{oct712}) follows.
% Hence if we denote by $ \tau$ the time it takes
%for $\|g_2\|_{L^1}$ to drop in half, one would have $ \tau \leq \tau_D$.

Recall that $g_1(\cdot, 0) = \rho_1(\cdot, 0)$ is concentrated at a distance $L\gg 1$
away from the origin, in the sense of (\ref{oct714}) and $\rho_1(x,0)$ is supported
inside $|x|\ge L/2$. This gives an upper bound
\begin{equation*}
\begin{split}
g_1(x,t) &= \frac{1}{4\pi t} \int_{\mathbb{R}^2} e^{-\tfrac{|x-y|^2}{4t}} \rho_1(y,0) dy \leq \frac{M_0}{4\pi t} e^{-C{L^2}/{t}}  \quad\text{ for all }x\in B(0,1).
\end{split}
\end{equation*}
%where the inequality comes from the fact that $y$ is supported near $L$ and $x \in B(0,1)$.
One can plug this estimate in the equation for $g_2$ and obtain
$$
\partial_t \log g_2 \geq -\frac{M_0\epsilon}{4\pi t} e^{-C{L^2}/{t}}.$$
Hence, $ \tau_D$ satisfies
\begin{equation*}
M_0\epsilon \int_0^{\tau_D}  \frac{1}{4\pi t} e^{-C{L^2}/{t}} dt \geq \log 2,
\end{equation*}
which, after a change of variable $y =CL^2/t$,
is equivalent to
\begin{equation}
\int_{CL^2/\tau_D}^\infty  \frac{e^{-y}}{y} dy \geq \frac{4\pi \log 2}{M_0\epsilon}.
\label{eq:tau_ineq_2}
\end{equation}
To estimate $\tau_D$, we consider  two cases.

{\textbf Case 1.} $M_0\epsilon \ll 1$, which is the
very weak reaction regime, or fairly small $M_0$ regime.
%We can also call this regime ``risky reaction".
%Since $\frac{1}{M_0\epsilon} \gg 1$ and $\int_1^\infty  \frac{e^{-y}}{y} dy \leq 1$,
Then, \eqref{eq:tau_ineq_2} is equivalent to
\[
\int_{CL^2/\tau_D}^1 \frac{1}{y} dy\gtrsim \frac{1}{M_0\epsilon},
\]
or, $-\log(CL^2/\tau_D) \gtrsim {1}/{(M_0\epsilon)}$. Thus $\tau_D$ has to satisfy
\begin{equation}\label{tauDrisky}
\tau_D \gtrsim L^2 e^{\frac{C'}{M_0\epsilon}}, \end{equation}
which is a very long time due to the large exponent.

\textbf{Case 2.} $M_0\epsilon \gg 1$, the reaction regime that appears more relevant
to the applications we have in mind.
In this case we have $CL^2/\tau_D\gg 1$, hence for a crude lower bound for $ \tau_D$,
one can find $\tau$ such that
\[
\int_{CL^2/\tau_D}^\infty e^{-y} dy \gtrsim \frac{1}{M_0\epsilon},
\]
which reduces to
$CL^2/\tau_D \lesssim \log(M_0\epsilon),$
and gives a bound
\begin{equation}\label{tauD}
\tau_D \gtrsim \frac{L^2}{\log(M_0\epsilon)}.
\end{equation}

\subsection{Formal heuristics with the chemotaxis term}\label{sec:heuristics}

Now we come back to the full system \eqref{eq:original_system}, including
the chemotaxis term. Again, let $\tau_C$ denote the half-time of $\rho_2$.  The following
formal argument suggests that adding this term may significantly reduce the half time
%in some regime of parameters. Namely,
in the regime $M_0\epsilon\gg 1$, where we will formally argue
that $\tau_C \sim L^2/\gamma\ll \tau_D\sim %$, which is much smaller compared to the half-time
%estimate $
L^2/\log(M_0\epsilon)$, as long as $\log(M_0 \epsilon) \ll \gamma$.

%The heuristics for the estimate $\tau_C \sim L^2/\gamma$ is as follows.
To this end, note that due to chemotaxis, $\rho_1$ is advected
by the velocity field
\begin{equation*}
%\begin{split}
v(x,t) = \chi \nabla ((-\Delta)^{-1} \rho_2)(x,t)
= -\frac{\chi}{2\pi} \int_{\mathbb{R}^2} \frac{x-y}{|x-y|^2} \rho_2(y,t) dy.
%\end{split}
\end{equation*}
Since $\tau_C$ is the half-time for $\rho_2$, for any $t\leq \tau_C$
we have $\|\rho_2(\cdot, t)\|_{L^1} \sim \theta$, and $\rho_2(\cdot, t)$
is supported near the origin. Therefore, for all $|x|\geq 2$ and $t\leq \tau_C$,
we have the following lower bound for the inward drift:
\begin{equation*}
\begin{split}
v(x,t) \cdot \frac{(-x)}{|x|} \sim \chi \int_{\mathbb{R}^2} \frac{\rho_2(y,t)}{|x-y|}dy
\sim \frac{\gamma}{|x|}.
\end{split}
\end{equation*}
Recall that initially all of $\rho_1$ starts at distance $L$ from the origin.
Hence, in the time~$t\sim L^2/\gamma$, the chemotactic transport should bring
a significant portion (say, a half) of $\rho_1$ into $B_1(0)$, and then
$\rho_1 \sim M_0$ in this ball.
%Once a half of $\rho_1$
%is in this ball, since $\rho_1$ has total mass $M_0$, ideally one may expect $\rho_1 \sim M_0$
%in this ball.
This enables the mass of $\rho_2$ to decrease exponentially at the
rate~$M_0 \epsilon\gg 1$, and the half-time is quickly reached;
thus one formally expects $\tau_C \lesssim L^2/\gamma.$

In the "risky" regime $M_0 \epsilon \ll 1,$ we need to add
non-trivial reaction time, which is now of the order
$\sim {1}/({M_0 \epsilon})$.
Then, one expects
\[
\tau_C \sim \frac{L^2}{\gamma} + \frac{1}{M_0 \epsilon},
\]
which can be quite a dramatic improvement compared to \eqref{tauDrisky}.

Note that this heuristic argument ignores many essential points, such as effect of diffusion,
or close field dynamics. There are indications that for the Keller-Segel chemotaxis term, reaction
time may be longer due to "over concentration" of $\rho_1.$ We discuss this point further in Section~\ref{Disc}.

\section{Global regularity and an $L^\infty$-bound}\label{sec:Linfty}

In order to get a uniform bound for the solutions to (\ref{eq:original_system}),
let us first consider an equation with a prescribed drift:
\begin{equation}
 \rho_t - \Delta \rho + \nabla\cdot(\rho \nabla \Phi(x,t)) = -h(x,t)\rho,
\label{eq:fp}
\end{equation}
where $h \in L^\infty(\mathbb{R}^d \times [0,\infty))$ is non-negative,
$\Phi$ is $H^2_{loc}$ in space for all time, and
such that $\nabla \Phi \in L^\infty(L^{\infty}(\R^d);[0,\infty))$. % and $\Delta \Phi \geq -\gamma.$
The proof of the following a priori $L^1-L^\infty$ bound for (\ref{eq:fp})
%which does not depend on its initial $L^\infty$ norm.
is very close to that of~\cite[Theorem 5]{carlenloss}. We recall it
in the appendix for the sake of completeness.
\begin{thm}\label{thm:L_infty_bound}
Let the initial condition $\rho_0$ for (\ref{eq:fp}) satisfy $\rho_0\in L^1(\mathbb{R}^d)\cap L^\infty(\mathbb{R}^d).$
Assume that $h \in L^\infty(\mathbb{R}^d \times [0,\infty))$
is non-negative, and $\Phi$ is $H^2_{loc}$ in space for all time and $\nabla \Phi \in L^\infty(L^{\infty}(\R^d);[0,\infty)).$
If there exists $\gamma>0$ such that $\Delta \Phi(\cdot, t) \geq -\gamma$ for all $t\geq 0$,
then
\begin{equation}\label{oct802}
\|\rho(\cdot, t)\|_{\infty }
\leq C(d) \max\left\{t^{-d/2}, \gamma^{d/2}\right\} \|\rho_0\|_{1},
~~\hbox{for all $t\geq 0$.}
\end{equation}
\end{thm}
The assumption that $\rho_0\in L^\infty(\Rm^d)$ in Theorem~\ref{thm:L_infty_bound}
is not necessary, and is made simply because we always consider solutions with bounded
initial conditions.

Note that the $\rho_1$-equation in \eqref{eq:original_system} is of the
form \eqref{eq:fp} with $h = \epsilon\rho_2 \geq 0$ and the
potential~$\Phi(\cdot,t) = \chi (-\Delta)^{-1}\rho_2(\cdot,t)$.
The potential $\Phi$ grows at a logarithmic rate at infinity, and minimal beyond $L^\infty$
regularity of $\rho_2$ would insure that $\Phi \in H^2_{loc}.$
This extra regularity is established below in Theorem~\ref{globreg1120}.
Also, from the explicit formula for the inverse Laplacian
it is not hard to see that $\nabla \Phi \in L^\infty(L^\infty(\R^d);[0,\infty)).$
 We will therefore be able to apply
Theorem \ref{thm:L_infty_bound} to obtain an a priori
bound for~$\|\rho_1(\cdot,t)\|_{L^\infty}$.

%\color{orange}
The global regularity of solutions to \eqref{eq:original_system} in all dimensions
$d\ge 1$ follows from a standard argument, which we briefly sketch below. The following lemma
contains the key estimates.

\begin{lem}\label{legelemma}
Suppose that $f,g \in L^1(\R^d) \cap H^m(\mathbb{R}^d)$, with an integer $m>d/2.$
Then we have
\begin{equation}\label{l1contle}
\|f \nabla (-\Delta)^{-1} g\|_{1} \leq C\|f\|_{L^1}(\|g\|_{1}+\|g\|_{\infty})
\end{equation}
and
\begin{equation}\label{hmcontle}
\|f \nabla (-\Delta)^{-1} g\|_{H^m} \leq C(\|g\|_1+\|g\|_\infty)\|f\|_{H^m}+C\|f\|_\infty\|g\|_{H^m}.
\end{equation}
\end{lem}
\begin{proof}
The inequality \eqref{l1contle} follows from the estimate
\[ \|f \nabla (-\Delta)^{-1} g\|_{1} \leq \|f\|_{1}\|\nabla (-\Delta)^{-1}g\|_{\infty}, \]
and
\begin{equation}\label{linfgrad19}
\|\nabla(-\Delta)^{-1} g\|_{\infty} \leq C \sup_{x\in\Rm^d}
\int_{\R^d} |x-y|^{-d+1} |g(y)|\,dy \leq C(\|g\|_1+\|g\|_{\infty}).
\end{equation}
To estimate the $H^m$ norm in \eqref{hmcontle}, let us start with the $L^2$ norm which is controlled similarly to
\eqref{linfgrad19}:
\[ \|f \nabla (-\Delta)^{-1}g\|_{2} \leq \|f\|_{2}\|\nabla (-\Delta)^{-1}g\|_{\infty} \leq
C\|f\|_{H^m}(\|g\|_{1}+\|g\|_{\infty}). \]
Any other term that we need to estimate to control the $H^m$ norm squared from \eqref{hmcontle} is of the form
\[ \int_{\R^d} D^j (f \nabla (-\Delta)^{-1}g ) \cdot D^j (f \nabla (-\Delta)^{-1}g )\,dx, \]
where $D^j$ is some partial derivative of the order $j \leq m.$
It suffices to control any term of the form
\begin{equation}\label{auxterm1920} \int_{\R^d} |D^{j-s} f|^2 |D^s \nabla (-\Delta)^{-1}g|^2 \,dx, \end{equation}
where integer $s$ satisfies $0 \leq s \leq j.$ If $s=0,$ \eqref{auxterm1920} is bounded by
$\|\nabla (-\Delta)^{-1} g\|_{\infty}^2 \|f\|^2_{H^j},$ and using \eqref{linfgrad19} and $j \leq m$
leads to the estimate we seek. If $s=j,$ \eqref{auxterm1920} is bounded by $\|f\|_\infty^2 \|g\|_{H^{m-1}}^2.$
If $j>s>1,$ we can estimate \eqref{auxterm1920} by
\begin{equation}\label{auxterm1920a}  C \|D^{j-s} f\|^2_{p} \|D^{s-1} g\|^2_{q}, \end{equation}
where $p^{-1}+q^{-1} =1/2,$ and $1<p,q<\infty.$
Specifically, let us choose $p = \frac{2(j-1)}{j-s}$ and $q = \frac{2(j-1)}{s-1}.$
In this step we used only that the Riesz transforms are bounded in $L^r$ if $1<r<\infty.$
Recall a Gagliardo-Nirenberg inequality
\begin{equation}\label{gn220} \|D^k f \|_{\frac{2n}{k}} \leq C \|f\|_\infty^{1-\frac{k}{n}}\|f\|_{H^n}^{\frac{k}{n}} \end{equation}
valid in any dimension for integer $k,n$ such that $0<k<n$ (see e.g. \cite{Mazya}). Applying it to the norms in \eqref{auxterm1920a} with $n=j-1$ and $k=j-s$ and $k=s-1$
respectively, we get the bound from above by
\[ C\|f\|_\infty^{2-2\frac{j-s}{j-1}}\|f\|_{H^{j-1}}^{2\frac{j-s}{j-1}}\|g\|_\infty^{2-2\frac{s-1}{j-1}}\|g\|_{H^{j-1}}^{2\frac{s-1}{j-1}}
\leq C\left(\|f\|_\infty^2 \|g\|_{H^{j-1}}^2 + \|g\|_\infty^2 \|f\|_{H^{j-1}}^2 \right). \]
Here we used the inequality $a^\beta b^{1-\beta} \leq a +b$ if $a,b \geq 0$ and $0 \leq \beta \leq 1.$
Finally, if $s=1,$ note that we can assume $j>1$ since otherwise $s=j$ and this is covered above. In this case,
we estimate
\[ \int_{\R^d} |D^{j-1}f|^2 |D\nabla (-\Delta)^{-1}g|^2\,dx \leq \|D^{j-1}f\|^2_{\frac{2j}{j-1}}\|D \nabla (-\Delta)^{-1}g\|^2_{2j} \leq
C\|D^{j-1}f\|^2_{\frac{2j}{j-1}}\|g\|^2_{2j}. \]
Due to \eqref{gn220},
\[ \|D^{j-1}f\|_{\frac{2j}{j-1}} \leq C \|f\|_\infty^{\frac{1}{j}}\|f\|_{H^j}^{\frac{j-1}{j}}, \]
while
\[ \|g\|_{2j} \leq \|g\|_\infty^{\frac{2j-1}{2j}}\|g\|_1^{\frac{1}{2j}}. \]
By Young's inequality,
\begin{eqnarray*} \|f\|_\infty^{\frac{1}{j}}\|f\|_{H^j}^{\frac{j-1}{j}}\|g\|_\infty^{\frac{2j-1}{2j}}\|g\|_1^{\frac{1}{2j}} \leq
C(\|f\|_{H^j}\|g\|_\infty^{\frac{2j-3}{2j-2}}\|g\|_1^{\frac{1}{2j-2}}+\|f\|_\infty\|g\|_\infty) \leq  \\
C\|f\|_{H^j} (\|g\|_1+\|g\|_\infty) +C\|f\|_\infty \|g\|_{H^m}. \end{eqnarray*}
Here in the last step we used $m>d/2.$ Since also $m \geq j,$ the lemma follows.
\end{proof}

%\textcolor{purple}{(In your paper with Lenya you took the norm to be $H^m$ plus the $M_n$ norm (which controls the $n$-th moment). We could add the $M_n$ norm here as well, but %do we need to do so?)}
\begin{thm}\label{globreg1120}
If the initial conditions $\rho_1(\cdot,0)$, $\rho_2(\cdot,0)$ for
(\ref{eq:original_system}) are non-negative, lie in $L^1(\R^d) \cap H^m(\mathbb{R}^d)$ with an integer $m>d/2$,
%$\rho_2(\cdot,0) \in H_0^m(\mathbb{R}^d)$,
%with $d \geq 1$ and $m>d/2$ is an integer,
then there is
a global in time solution~$(\rho_1(\cdot, t), \rho_2(\cdot,t)) \in C(L^1(\R^d) \cap H^m(\mathbb{R}^d),[0,\infty))$
% L^1(\R^d) \cap H^m(\mathbb{R}^d))$
to \eqref{eq:original_system}.
\end{thm}
\begin{proof}
We assume that $\rho_1(\cdot,0)$, $\rho_2(\cdot,0)$ are non-negative purely for simplicity since in all our applications this is the case.
This assumption is not hard to remove. We note that standard comparison principle implies that non-negativity is conserved in time
for all sufficiently regular solutions.

The local in time well-posedness in $C(L^1(\R^d) \cap H^m(\Rm^d),[0,T])$ can be shown by a standard argument,
using Duhamel formula and the contraction mapping theorem, similarly to \cite[Appendix I]{kr}.
By integrating the equations, we find that the $L^1$ norms of $\rho_1(\cdot,t)$ and $\rho_2(\cdot,t)$ (which are equal to their integrals due
to non-negativity) are non-increasing in time.
Hence to improve the local well-posedness result to a global-in-time one, it suffices to obtain an a priori bound on
\[
I(t):= \|\rho_1(t)\|_{H^m}^2 + \|\rho_2(t)\|_{H^m}^2,
\]
on any given
finite time interval $[0,T]$.
Fix any multi-index $\alpha$ with $0 \leq |\alpha| \leq m$  and write
%$\alpha_1, \alpha_2$ are non-negative integers with $\alpha_1 + \alpha_2 \leq m$, we have
\begin{equation}\label{hm2}
\begin{split}
\frac{1}{2}\frac{d}{dt} \|\partial^\alpha \rho_2\|_{2}^2
&=-\epsilon \int_{\mathbb{R}^2} (\partial^\alpha \rho_2)\partial^\alpha(\rho_1 \rho_2)\, dx
\leq \epsilon \|\rho_2\|_{H^m} \|\rho_1 \rho_2\|_{H^m}\\&
\leq C \|\rho_2\|_{H^m} ( \|\rho_1\|_{\infty}  \|\rho_2\|_{H^m}
+ \|\rho_2\|_{\infty}\|\rho_1\|_{H^m})\\
&\leq C(\|\rho_1(\cdot,t)\|_{\infty} +1) (\|\rho_1(\cdot,t)\|_{H^m}^2
+ \|\rho_2(\cdot,t)\|_{H^m}^2).
\end{split}
\end{equation}
Here, the second line is obtained by the inequality  (see, e.g., \cite[Lemma 3.4]{mb})
\begin{equation}\label{ineq_sobolev}
\|uv\|_{H^m} \leq C(\|u\|_{\infty}\|v\|_{H^m} + \|v\|_{\infty} \|u\|_{H^m})
\quad\text{ for } m>{d}/{2},
\end{equation}
and in the last line we use the fact
that $\|\rho_2(t)\|_{\infty} \leq \|\rho_2(0)\|_{\infty} \leq C$.
As for $\|\rho_1\|_{H^m}$, for any multi-index $\alpha$ as above, integration by parts gives
\begin{equation}\label{gr529a}
%\begin{split}
\frac{1}{2}\frac{d}{dt} \|\partial^\alpha \rho_1\|_{2}^2 =
-\|\nabla\partial^\alpha  \rho_1\|_{2}^2 +
\chi \int_{\mathbb{R}^2} \nabla(\partial^\alpha  \rho_1)
\cdot  \partial^\alpha( \rho_1 \nabla (-\Delta)^{-1}\rho_2)dx
-\epsilon \int_{\mathbb{R}^2} (\partial^\alpha \rho_1) \partial^\alpha(\rho_1 \rho_2) dx.
\end{equation}
The last integral on the right side can be bounded by the right side of (\ref{hm2}), while
%\[
%C(\|\rho_1(\cdot,t)\|_{\infty} +1) (\|\rho_1(\cdot,t)\|_{H^m}^2 + \|\rho_2(\cdot,t)\|_{H^m}^2).
%\]
%&\leq -\|\nabla\partial^\alpha  \rho_1\|_{L^2}^2 + \chi \|\nabla\partial^\alpha  \rho_1\|_{L^2}
%\|\partial^\alpha( \rho_1 \nabla (-\Delta)^{-1}\rho_2)\|_{L^2}\\
the first one %integral on the right hand side of \eqref{gr529a}
can be estimated by
\begin{equation}\label{oct804}
\begin{split}
&\frac{1}{\chi}\|\nabla\partial^\alpha  \rho_1\|_{2}^2+ C \|\rho_1  \nabla (-\Delta)^{-1}\rho_2\|_{H^m}^2 \\
&\leq \frac{1}{\chi}\|\nabla\partial^\alpha  \rho_1\|_{2}^2+  C(\|\rho_1\|_{\infty}^2
\|\rho_2\|_{H^m}^2 + \|\rho_1\|_{H^m}^2(\|\rho_2\|_1^2+\|\rho_2\|_\infty^2)).
%\|\nabla (-\Delta)^{-1}\rho_2\|_{\infty})^2 \\
%&\leq \frac{1}{\chi}\|\nabla\partial^\alpha  \rho_1\|_{2}^2+ C(\|\rho_1\|_{\infty} \|\rho_2\|_{H^m} + \|\rho_1\|_{H^m})^2\\
%&\leq \frac{1}{\chi}\|\nabla\partial^\alpha  \rho_1\|_{2}^2+ C (\|\rho_1(\cdot,t)\|_{\infty} + 1)^2 (\|\rho_1(\cdot,t)\|_{H^m}^2 + %\|\rho_2(\cdot,t)\|_{H^m}^2).
\end{split}
\end{equation}
We used the Cauchy-Schwarz inequality and the Young's inequality in the first line
and Lemma~\ref{legelemma} in the second line; the constants $C$ depend on $\chi$ and may change from line to line.
%In the third line in (\ref{oct804}) we used the estimate \eqref{linfgrad19}:
%\[
%\begin{split}
%\|\nabla(-\Delta)^{-1} \rho_2\|_{\infty}
%\leq C \sup_{x\in\Rm^d}
%\int_{\R^d} |x-y|^{-d+1} |\rho_2(y)|\,dy
%\leq C(\|\rho_2\|_1+\|\rho_2\|_{\infty})
% \leq C(\|\rho_2(\cdot,0)\|_{1}+ \|\rho_2(\cdot,0)\|_{\infty})) \leq C.
%\end{split}
%\]
%where the second step is obtained by splitting the domain of integration into $B_1(x)$ and
%$B_1(x)^c,$ and the third step is
%based on the fact that $\|\rho_2(t)\|_{1}$ and $\|\rho_2(t)\|_{\infty}$ are both non-increasing.
Combining the above estimates and taking into account that $\|\rho_2\|_1$ and $\|\rho_2\|_\infty$ are non-increasing gives
%for $\frac{d}{dt} \|\partial^\alpha \rho_1\|_{2}^2$
%and $\frac{d}{dt} \|\partial^\alpha \rho_2\|_{2}^2$ yields
\begin{equation}\label{gronwallge}
\frac{d}{dt} (\|\rho_1(\cdot,t)\|_{H^m}^2 + \|\rho_2(\cdot,t)\|_{H^m}^2) \leq C (\|\rho_1(\cdot,t)\|_{\infty} + 1)^2 (\|\rho_1(\cdot,t)\|_{H^m}^2 + \|\rho_2(\cdot,t)\|_{H^m}^2).
\end{equation}
The first equation in \eqref{eq:original_system} gives the bound
\begin{equation}\label{rho1119}
\|\rho_1(\cdot,t)\|_{\infty} \leq \|\rho_1(\cdot,0)\|_{\infty}
\exp\big\{\chi \|\rho_2(\cdot,t)\|_{\infty}t\big\}
\leq  \|\rho_1(\cdot,0)\|_{\infty} \exp\big\{\chi \|\rho_2(\cdot,0)\|_{\infty}t\big\}.
\end{equation}
Thus $\|\rho_1(\cdot,t)\|_\infty$ remains finite for all times, and then \eqref{gronwallge} leads to global regularity.
To get a more precise bound, we may use \eqref{rho1119}
%so that $\|\rho_1(\cdot,t)\|_{L^\infty}$ is uniformly bounded
for $0 \leq t \leq 1$,
while for $t\ge 1$ we may deploy the uniform bound from Theorem~\ref{thm:L_infty_bound}.
%bounded on $[0,T]$ for any $T<\infty.$
%In fact, combining this exponential bound with Theorem~\ref{thm:L_infty_bound}
%gives a
Therefore, there exists $C>0$ so that
$\|\rho_1(\cdot,t)\|_{\infty} \leq C$ for all $t\ge 0$.
Then \eqref{gronwallge} gives exponential in time control of the $H^m$ norms of the solution, for all times.
%It follows that $I(t)$ is bounded in $[0,T]$ for any $T>0$,
%finishing the proof.
\end{proof}

%\color{black}

\section{The mass comparison principle}
\label{sec:comparison_principle}

We now obtain a comparison principle that allows us to compare $\rho_1$ to
the solution $\rho$ of the Fokker-Planck equation
\begin{equation}
\partial_t \rho - \Delta \rho + \nabla \cdot (\rho \nabla H) = 0,
\label{eq:fokker_planck_2}
\end{equation}
with a certain prescribed $H$.
The comparison will be in a mass concentration sense that
will be clarified in Proposition \ref{prop:mass_comparison}.
%The general part of the result will be stated for arbitrary dimension $d \geq 2$ for future reference.
%The one dimensional version is not hard to derive but it will be done elsewhere.
Let us assume that $H=(-\Delta)^{-1} g$, with a radially symmetric function $g=g(|x|)$ supported in a ball $B_{R_0}(0).$
%decaying sufficiently fast so that the integral defining $H$ converges (for example, $g(x) \leq C|x|^{-d-\epsilon}$
%for large $|x|$)
%of compact support, in a ball $B_{R_0}(0)$.
% is radially symmetric, i.e. $g(x)=g(|x|)$.
The explicit form of $g$ and $H$, that we will use, is  given in \eqref{def:g} and \eqref{H:explicit}.
The function $H$ is radially symmetric as well, and the divergence theorem gives
\begin{equation}
%\begin{split}
\partial_r H(r)= \frac{1}{|\partial B_r|} \int_{B_r} \Delta H(x) dx
= \frac{1}{2\pi r} \int_{B_r} (-g(x))dx= -\frac{1}{r} \int_0^r g(s) s ds.
%\end{split}
\label{eq:div_thm}
\end{equation}
Integrating in $r$ gives an expression
\begin{equation}
H(r) = -\log r \int_0^r g(s)sds - \int_r^\infty (\log s) g(s) s ds+\hbox{const}.
\label{eq:explicit_H}
\end{equation}
%Clearly $H$ is defined up to a constant and
Since $g$ is compactly supported, taking the arbitrary constant in (\ref{eq:explicit_H}) to be zero gives
\[
H(r) = - \farc{1}{2\pi}\|g\|_{L^1} \log r, \,\,r \geq R_0.
\]
for large $r.$
%if $d=2$ and
%\[ H(r) = \frac{d-2}{r^{d-2}} \int_0^R g(s)s^{d-1}\,ds + (d-2)\int_r^\infty g(s)s\,ds \]
%for $d>2.$
%In the next lemma and proposition,
%We now make some observations on the $(-\Delta)^{-1}$
%operator and equation \eqref{eq:fokker_planck_2}.
As a direct consequence of \eqref{eq:div_thm}, we have the following.
\begin{lem}
Assume that $g_1$ and $g_2$ are both radially symmetric and compactly supported.
Suppose that $g_1$ is more concentrated than $g_2$, in the sense that
\[
\int_0^r g_1(s) s ds \geq \int_0^r g_2(s)sds\hbox{ for all $r\geq 0$.}
\]
Then, the functions $H_i := (-\Delta)^{-1} g_i$, $i=1,2$, satisfy
$\partial_r H_1 \leq \partial_r H_2 \leq 0$ for all $r> 0$.
In addition, if $g_i\in {L^\infty(\Rm^2)}$, then $\partial_r H_i(0) = 0$ for $i=1,2$.
\label{lemma:inverse_laplacian}
\end{lem}
%
%\begin{proof}
%The proof is
%\end{proof}
We now compare the mass concentration of solutions to the Fokker-Planck equations.
\begin{prop}\label{lemma:FP_comparison}
Suppose that $u_1$ and $u_2$ are non-negative solutions to
\[
\partial_t u_i - \Delta u_i + \nabla\cdot(u_i \nabla H_i) = 0,
\]
for $i=1,2$, and $u_1$ is more concentrated than $u_2$ at $t=0$, so that
\begin{equation}\label{oct808}
\int_{B_r} u_1(x, 0) dx \geq \int_{B_r} u_2(x,0) dx\hbox{ for all $r\geq 0$.}
\end{equation}
If, in addition, $H_2(\cdot,t)$ is radially symmetric, and
\begin{equation}\label{oct806}
\partial_r H_2(r,t) \geq \max_\phi \partial_r H_1(r,\phi,t),
\hbox{ for all $t \geq 0$ and $r>0$,}
\end{equation}
% Suppose also that $H_i(\cdot, t)$ are radially symmetric for $i=1,2$ and all $t$, and $\partial H_1(r,t) \leq \partial_r H_2(r,t) \leq 0$.
then $u_1(\cdot, t)$ is more concentrated than $u_2(\cdot, t)$ for all $t\geq 0$.
\end{prop}
Note that $u_{1,2}$ are not necessarily radially symmetric.
%\noindent \it Remark. \rm
These results are valid in arbitrary space dimension -- the proof below is given for $d=2$ for notational convenience but the argument
can be generalized in a straightforward manner.
\begin{proof}
The masses
\[
M_i(r,t) :=  \int_{B_r} u_i(x,t) dx
\]
satisfy
\begin{equation}
\begin{split}
\partial_t M_i(r,t) &=
\int_{B_r} \Delta u_i \,dx- \int_{B_r} \nabla\cdot(u_i \nabla H_i) \, dx
= \int_{\partial B_r} \partial_r u_i \,d\sigma -
\int_{\partial B_r} u_i \partial_r H_i \,d\sigma\\
&= r\int_0^{2\pi} \partial_r u_i (r,\phi,t) d\phi
-r \int_0^{2\pi} u_i(r,\phi,t) \partial_r H_i(r,\phi,t)\,d\phi. \,
\label{eq:M_1}
\end{split}
\end{equation}
Here, $d\sigma = r d\phi$ is the surface  measure on the boundary.
Note that
\[
\partial_r M_i = \int_{\partial B_r} u_i d\sigma = r\int_0^{2\pi} u_i (r,\phi, t) d\phi,
\]
so that
% Hence dividing by $r$ and then taking derivative in $r$ yields
\[
\int_0^{2\pi} \partial_r u_i d\phi = \partial_r \Big(\frac{\partial_r M_i}{r}\Big)
= \frac{\partial_{rr} M_i}{r} - \frac{\partial_r M_i}{r^2}.
\]
Substituting the above two equations into \eqref{eq:M_1} gives
\begin{equation}
\partial_t M_i(r,t) = \partial_{rr} M_i - \frac{1}{r} \partial_r M_i -
%\frac{\partial M_i }{\partial r}  ~\frac{\partial H_i}{\partial r}.
r\int_0^{2\pi} u_i(r,\phi,t) \partial_r H_i (r,\phi,t)\,d\phi.
\label{eq:M}
\end{equation}
Subtracting the two equations and using the radial symmetry of $H_2$, we obtain
\begin{equation}
\begin{aligned}
\partial_t (M_1-M_2) &- \partial_{rr} (M_1-M_2) + \frac1r \partial_r (M_1-M_2)
\geq \partial_r H_2 \partial_r M_2 - (\partial_r M_1)
\max_\phi \partial_r H_1(r,\phi,t)   \\
&= -(\partial_r H_2) \partial_r (M_1-M_2) +
\big(\partial_r H_2 - \max_\phi \partial_r H_1(r,\phi,t)\big) \partial_r M_1\\
&\ge -(\partial_r H_2) \partial_r (M_1-M_2). \label{eq:M1}
\end{aligned}
\end{equation}
We used (\ref{oct806}) as well as $\partial_r M_1 \geq 0$ in the last inequality above.
%Notice that the last term on the right hand side is non-negative by
%assumption on $H_i$ and since $\partial_r M_1 \geq 0.$
%Recall that $\partial H_1(r,t) \leq \partial_r H_2(r,t) \leq 0$, and $\partial_r M_{1,2}(x,t) \geq 0$ for all $x$ and $t$ due to the definition of $M_{1,2}$.
Now, the standard parabolic comparison principle (see e.g. \cite{lieberman,PW}) and~(\ref{oct808}) imply that
%Now one can apply the parabolic comparison principle to \eqref{eq:M1}, and get
\[
M_1(r,t) \geq M_2(r,t)\hbox{ for all $r,t \geq 0$.}
\]
To make the application completely routine one can consider $M_1^\epsilon(r,t) = M_1(r,t)+\epsilon$ with $\epsilon>0$ (note that $M_1^\epsilon$ satisfies the same equation
as $M_1$). Then in view of the definition of $M_i$
and the upper bound of Theorem~\ref{thm:L_infty_bound}, we have $M_1^\epsilon(r,t) -M_2(r,t)>0$ in some small neighborhood of $r=0$ uniformly in $t.$ Larger values of $r$
are controlled by standard comparison principle. Taking $\epsilon$ to zero yields the result.
\end{proof}

Let us now go back to \eqref{eq:original_system}.
%We are interested in an upper bound
%on the half-time time $\tau_C$ before the $L^1$ norm of $\rho_2$ falls in half.
Let us recall the notation
\begin{equation}\label{oct810}
\theta = \frac{1}{\pi} \|\rho_2(\cdot,0)\|_{L^1},~~M_0 = \|\rho_1(\cdot,0)\|_{L^1},
~~\gamma=\chi\theta,
\end{equation}
and that we are interested in the regime
%The regime we are particularly interested in due to its relevance in applications is
%$M_0\epsilon \gg \gamma \gg 1 \gg \epsilon$,
$M_0 \gg \theta$. We assume to simplify the technicalities
that $\rho_2(\cdot, 0)$ is smooth
but very close to $\chi_{B_1}(x)$ in $L^1$ norm, and $\rho_2(\cdot,0) \geq \chi_{B_1}(x)$,
%and is concentrated on a ball of radius slightly larger than $1$,
%and assume  .
%As above,
%We prefer that $\rho_2$ is smooth to reduce the technicalities,
but in the argument below we think of $\rho_2(\cdot, 0)$ as equal to~$\theta\chi_{B_1}(x)$.
To make this argument completely rigorous, while  still using exactly the
function $g(x)$ in (\ref{def:g}), and keeping $\rho_2(\cdot,0)$ smooth,
one may work with a time $\tau_\alpha$
by which the mass of $\rho_2$ drops by a factor of $\alpha$ with $\alpha<1/2$,
rather than~$\tau_C$, as the discrepancy between $\rho_2(\cdot,0)$ and $\chi{B_1}$
can be made arbitrarily small in $L^1(\Rm^2)$.
%In what follows, we will ignore the small
%discrepancy between this characteristic function and its smooth approximation
%as it can be made arbitrarily small.

Observe %, using \eqref{eq:div_thm},
that any radial function $f(x)\ge 0$
supported on $B_1,$ and such that
\[
0 \leq f(x) \leq \rho_2(x,0)\hbox{ and }
\|f\|_{L^1} \geq \frac{1}{2} \|\rho_2(\cdot, 0)\|_{L^1},
\]
is more concentrated than %$g(x)$ given by
%the least concentrated function $g$ supported on $B_1$ satisfying $0 \leq g(x) \leq \rho_2(x,0)$ and $\|g\|_{L^1} \leq \frac{1}{2} %\|\rho_2(\cdot, 0)\|_{L^1}$ is given by
\begin{equation}
g(x):= \theta (\chi_{B_1}(x) - \chi_{B_{1/\sqrt{2}}}(x)).
\label{def:g}
\end{equation}
In particular, $g$ is less concentrated than $\rho_2(\cdot, t)$ for all $t\leq \tau_C$.
One may use \eqref{eq:explicit_H} to obtain
%$H(x) := \chi (-\Delta)^{-1} g$; one can apply \eqref{eq:explicit_H} to compute $H$ explicitly:
\begin{equation}\label{H:explicit}
H(x):= \chi (-\Delta)^{-1} g = \begin{cases}
({\gamma}/{8}) (1-\log 2) & \text{ for } 0\leq r < {1}/{\sqrt{2}}. \\
({\gamma}/{4})(\log r + 1 - r^2) & \text{ for } {1}/{\sqrt{2}} \leq r < 1,\\
-({\gamma}/{4}) \log r & \text{ for } r\geq 1.\\
\end{cases}
\end{equation}
We can now compare $\rho_1$ to the solution to the Fokker-Planck equation with
the drift potential~$H$, and conclude the following:
\begin{prop}\label{compmassprop521}
Let $\rho_1(x,t), \rho_2(x,t)$ solve \eqref{eq:original_system}
with radially symmetric initial conditions,
where $\rho_2(\cdot, 0) = \theta \eta,$ $\eta$ smooth, radial and $\eta(x) \geq \chi_{B_1}(x)$,
and $\rho(x,t)$ solve the Fokker-Planck equation \eqref{eq:fokker_planck_2} with the drift potential $H$ given by \eqref{H:explicit} and the same initial condition as $\rho_1$.
Let~$\tau_C$ be the time it takes for the $L^1$ norm of $\rho_2$ to decrease
by $\theta\pi/2$.  Then we have
\begin{equation}
\int_{B_r} \rho_1 (x, t) dx \geq \int_{B_r} \rho(x, t) dx - \frac{1}{2}\int_{\R^2} \rho_2(x,0)\,dx \quad\text{ for all }t\leq \tau_C \text{ and }r\geq 0.
\label{eq:mass_rho1}
\end{equation}
\label{prop:mass_comparison}
\end{prop}
\begin{proof}
%To begin, given that $\rho_1(x,t)$ and $\rho_2(x,t)$ solve \eqref{eq:original_system},
Let $\tilde \rho$ solve the equation for $\rho_1$ without the reaction term:
\begin{equation}
\partial_t \tilde \rho - \Delta \tilde \rho + \chi \nabla \cdot (\tilde \rho \nabla (-\Delta)^{-1} \rho_2) = 0,
\label{eq:modified_rho1}
\end{equation}
with the same initial condition as $\rho_1$.
%Given the solution $\rho(x,t)$ to the Fokker-Planck equation \eqref{eq:fokker_planck_2}
%with the initial condition $\rho_1(x,0),$
Note %compare $\tilde \rho(\cdot, t)$ to $\rho(\cdot, t)$, and show
that~$\tilde \rho(\cdot, t)$ is more concentrated than $\rho(\cdot, t)$
for all~$t\leq \tau_C$. Indeed, the function $g$ defined in \eqref{def:g}
is less concentrated than $\rho_2(\cdot, t)$ for all~$t \leq \tau_C$,
% This enables us to apply
hence Lemma \ref{lemma:inverse_laplacian} implies that
\[
\chi \partial_r (-\Delta)^{-1} \rho_2(\cdot, t) \leq \partial_r H\leq 0 \text{ for all }t\leq \tau_C,
\]
where $H$ as in (\ref{H:explicit}).
%Since $\rho$ and $\tilde \rho$ both solve the Fokker-Planck
%equation \eqref{eq:fokker_planck_2} with drift potential $H$
%and $\chi (-\Delta)^{-1} \rho_2(\cdot, t)$ respectively, the comparison result in
Thus, Proposition \ref{lemma:FP_comparison} gives
\begin{equation}
\int_{B_r} \tilde \rho(x,t) dx \geq \int_{B_r} \rho(x, t)dx \text{ for all }t\leq \tau_C \text{ and } r\geq 0.
\label{eq:mass_temp}
\end{equation}
To prove \eqref{eq:mass_rho1}, it suffices then
to compare $\rho_1$ and $\tilde \rho$ and show that
\begin{equation}
\int_{B_r}  \rho_1(x,t) dx \geq \int_{B_r} \tilde \rho(x, t)dx - \frac{1}{2} \int_{\R^2} \rho_2(x,0)\,dx \text{ for all }t\leq \tau_C \text{ and } r\geq 0.
\label{eq:mass_rho_1_rho_tilde_1}
\end{equation}
Note that
\begin{equation}\label{41120}
\int_{\R^2} \rho_1(x, 0)\,dx -\int_{\R^2} \rho_1(x, t)\,dx
= \int_{\R^2} \rho_2(x, 0)\,dx -\int_{\R^2} \rho_2(x, t)\,dx
\leq \frac{1}{2}\int_{\R^2}\rho_2(x,0)\,dx \text{ for all }t\leq \tau_C,
\end{equation}
%where in the last step we used that $\tau_C$ is the half-time for $\rho_2$.
%The inequality \eqref{eq:mass_rho_1_rho_tilde_1} is due to the following simple observation.
and the comparison principle implies that
%On one hand, since the only difference between the equations for $\tilde \rho$ and $\rho_1$
%is the non-positive reaction term, applying comparison principle on these two densities gives %that
\begin{equation}\label{rhotildedom}
\tilde\rho(x,t) \geq \rho_1 (x,t) \text{ for all }x,t.
\end{equation}
Hence, we may write
\begin{equation*}
\begin{split}
&\int_{B_r} \rho_1(x, t) dx = \int_{\mathbb{R}^2} \rho_1(x, t)dx
- \int_{\mathbb{R}^2\backslash B_r} \rho_1(x, t)dx\\
&\ge
\int_{\mathbb{R}^2} \rho_1(x,0)dx-\farc{1}{2}\int_{\mathbb{R}^2} \rho_2(x,0)dx
- \int_{\mathbb{R}^2\backslash B_r} \tilde\rho(x, t)dx
\\
&= -\frac{1}{2}\int_{\R^2} \rho_2(x,0)\,dx +
\int_{\mathbb{R}^2} \tilde \rho(x, 0) dx -
\int_{\mathbb{R}^2\backslash B_r} \tilde \rho(x, t)dx
= \int_{B_r} \tilde \rho(\cdot, t)dx -\frac{1}{2}\int_{\R^2} \rho_2(x,0)\,dx,
\end{split}
\end{equation*}
which is \eqref{eq:mass_rho_1_rho_tilde_1}. Here we used \eqref{41120} and \eqref{rhotildedom}
in the first step, and conservation of mass for $\tilde \rho$ in the last step.
% is obtained, and  the proof is finished once we combine it with \eqref{eq:mass_temp}.
%
%
%In addition, we claim that for all $t\leq \tau_C$, the masses of $\rho_1$ and $\tilde \rho$ inside any ball $B_r$ do not differ by more than $\frac12 \int_{\R^2} \rho_2(x,0)\,dx$. Note that the total mass of $\tilde\rho$ is preserved over time, while the total mass of $\rho_1$ is decreasing due to the reaction term. %To prove the claim, it suffices to show that
%However, the total mass of $\rho_1$ can at most drop by $\frac{1}{2}\int_{\R^2} \rho_2(x,0)\,dx $ over time $\tau_C$.
%This is true since the equations for $\rho_1$ and $\rho_2$ have the same reaction term, hence
%$$\| \rho_1(\cdot, 0)\|_{L^1} -\| \rho_1(\cdot, t)\|_{L^1} = \| \rho_2(\cdot, 0)\|_{L^1} -\| \rho_2(\cdot, t)\|_{L^1} \leq \frac{1}{2}\|\rho_2(x,0)\|_{L^1} \text{ for all }t\leq \tau_C,$$
%where in the last step we used that $\tau_C$ is the half-time for $\rho_2$. This inequality, \eqref{rhotildedom}, and conservation of
%mass for $\tilde \rho$ imply
%\begin{equation*}
%\begin{split}
%\int_{B_r} \rho_1(\cdot, t) dx &= \int_{\mathbb{R}^2} \rho_1(\cdot, t)dx - \int_{\mathbb{R}^2\backslash B_r} \rho_1(\cdot, t)dx\\
%&\geq -\frac{1}{2}\int_{\R^2} \rho_2(x,0)\,dx +  \int_{\mathbb{R}^2} \tilde \rho(\cdot, 0) dx - \int_{\mathbb{R}^2\backslash B_r} \tilde \rho(\cdot, t)dx\\
%&= \int_{B_r} \tilde \rho(\cdot, t)dx -\frac{1}{2}\int_{\R^2} \rho_2(x,0)\,dx.
%\end{split}
%\end{equation*}
%Thus \eqref{eq:mass_rho_1_rho_tilde_1} is obtained, and  the proof is finished once we combine it with \eqref{eq:mass_temp}.
\end{proof}

%\section{Convergence rate for Fokker-Planck equation}\label{sec:fp}
\section{Weak weighted Poincar\'e-type inequalities}\label{sec:wwpti}

In this section, we develop some analytical tools that we will need
to derive sufficiently sharp estimates on the convergence
to equilibrium rates for the solutions to the Fokker-Planck equation with
a logarithmic potential. To motivate these results, consider the Fokker-Planck equation
\begin{equation}
\partial_t \rho - \Delta \rho + \nabla\cdot(\rho \nabla H) = 0
\text{ in } \mathbb{R}^2 \times [0,+\infty),
\label{eq:fp_const}
\end{equation}
where $H = \chi(-\Delta)^{-1} g$ is time independent, and $g(x)$ is
the radially symmetric function supported in $B(0,1)$ defined in \eqref{def:g}.
%\[ %\begin{equation}
%g(x) = \theta \Big(\chi_{B_1(x)} - \chi_{B_{1/\sqrt{2}}(x)}\Big).
%\] %\end{equation}
As outlined in the previous section, we plan to use the solution $\rho$ as a comparison tool to control the behavior of $\rho_1.$

The operator $L$
\[
L \rho = -\Delta\rho+ \nabla \cdot (\rho \nabla H)
\]
%It is straightforward to check that the operator $L$
is self-adjoint in the weighted space $L^2(e^{-H}dx)$ (when defined on a natural
weighted Sobolev space),
and is non-negative. Its unique ground state corresponding to the zero eigenvalue is
a multiple of $e^H$, provided that
\[
\int e^H\,dx < \infty,
\]
otherwise there is no ground state.
%One can check that for $\|\rho(\cdot, 0)\|_{L^1(\mathbb{R}^2)} = 1$, the unique steady state solution is $e^{H(x)}/\|e^H\|_{L^1(\mathbb{R}^2)}$.
In our situation, $H$ is given by \eqref{H:explicit}, so that
%and thus $e^H$ is given by
\begin{equation}
\label{eH}
e^{H(x)} = \begin{cases}
({e}/{2})^{\gamma/8}, & |x|< {1}/{\sqrt{2}},\\
e^{\gamma/4} |x|^{\gamma/4} e^{-\gamma|x|^2/4}, & {1}/{\sqrt{2}}\leq |x| \leq 1,\\
|x|^{-\gamma/4}, &|x|\geq 1.
\end{cases}
\end{equation}
As the evolution \eqref{eq:fp_const} conserves the integral of $\rho$,
%\int \rho\,dx,$ so
we expect theat
\[
\rho(t,x)\to e^{H(x)} \Big(\int \rho\,dx\Big)\Big( \int e^H\,dx\Big)^{-1},
\hbox{ as $t\to+\infty$.}
\]
%Our goal in this section is to obtain some explicit convergence rate of $\rho$ towards the stationary solution.
%Let us define the operator $L$ by $L\rho = \Delta \rho + \nabla \cdot (\rho \nabla H)$,
The dual operator $L^*$ with respect to
the standard~$L^2(dx)$ inner product,
given by
\[
L^*f = -\Delta f - \nabla H \cdot \nabla f,
\]
is self-adjoint in $L^2(e^H \,dx),$ with the ground state equal to a constant.
The corresponding dual evolution is
\begin{equation}
\label{eq:f}
\partial_t f - \Delta f - \nabla H \cdot \nabla f = 0.
\end{equation}
Note that $\rho(\cdot, t)$ solving \eqref{eq:fp_const} is equivalent to
\[
f(x, t) := \rho(x, t)e^{-H(x)}
\]
solving \eqref{eq:f}.
The evolution \eqref{eq:f} conserves the integral of $f(x) \exp(H(x))$
so we expect that
\[
f(x,t)\to\bar f := \Big(\int f_0e^H dx\Big)\Big(\int e^H dx\Big)^{-1}, \hbox{ as $t\to+\infty$,}
\]
where $f_0(x)=f(x,0)$.
%which is comparable to $\frac{M}{(\frac{e}{2})^{\gamma/4}}$.
Note that we have
\begin{equation}\label{dzdt}
\frac{d}{d t} \underbrace{\int_{\mathbb{R}^2} (f(x,t) - \bar f)^2 e^{H(x)} dx}_{=:Z(t)} = -2\underbrace{\int_{\mathbb{R}^2} |\nabla f(x,t)|^2  e^{H(x)} dx}_{=:W(t)}.
\end{equation}
If we can bound $Z(t)$ from above as
\[
Z(t) \leq g(W(t), \|f_0\|_\infty),
\]
with some function $g$ that increases in $W$, that would allow us to
bound $W(t)$ from below in terms of $Z(t)$ and $\|f_0\|_\infty$. Then,
\eqref{dzdt} would give us a differential inequality for $Z(t)$
leading to an explicit decay estimate on $Z(t)$. %This is our goal in the next two subsections.
In the simplest case, the bound $Z \leq CW$ applies, which is a standard Poincar\'e inequality.
Then there is a spectral gap for $L^*$, and exponential in time convergence
to the ground state in $L^2(e^H\,dx).$ This is true for uniformly concave potentials,
as in for the Brascamp-Lieb inequality. However, it is not difficult to verify that
in the case of logarithmic (or even $|x|^\alpha$ with $\alpha<1$) potential there is
no spectral gap, that is, the ground state zero is not an isolated point of
the spectrum. Then, the usual Poincar\'e inequality cannot hold, and one needs what
is called a weak Poincar\'e version that manifests itself in a different,
stronger weight deployed for the gradient norm.

We will prove the weak weighted Poincar\'e inequality for a more
general family of radial weights~$w(r)\geq 0$, which depend on a parameter $\gamma>0$,
than the specific choice
(\ref{eH}),
since the argument is essentially the same.
%Since we will work in polar coordinates, it will be convenient to include the Jacobian $r$ into the weight $w(r)$.
We will assume that the weights have the following properties:
there exist $0 < r_1 < r_2 < \infty$ and constants $C_0,$ $C_1$
independent of $\gamma$ such that
\begin{eqnarray}
&& C_0^{-1} w(s) \leq w(r) \leq C_0 w(s), \,\,\, {\rm for \,\,\,all}\,\,\,s,r \in [0,r_1]; \label{case0a} \\
&&w'(r)  \leq -C_1 \gamma (r-r_1) w(r), \,\,\, r \in [r_1,r_2];  \label{derivative_w} \\
&&w'(r) \leq -C_2 \gamma r^{-1} w(r), \,\,\, r \in [r_2,\infty). \label{case3a}
\end{eqnarray}
An elementary computation shows that
for the weight $w(r) = \exp(H(r))$ given by~(\ref{eH}),
assumptions (\ref{case0a})-(\ref{case3a}) hold with
\begin{equation}\label{oct2302}
r_1=\farc{1}{\sqrt{2}},~~r_2={3}/{4},
\end{equation}
where the choice of $r_2$ is rather arbitrary; any number larger than $r_1$ would work.
% in what follows (with adjusted constants).
%Indeed, then \eqref{case0a} is satisfied because $e^H$ is constant for $r \leq r_1.$ To check \eqref{derivative_w}, we compute for $r \in [r_1,1]$ using \eqref{eH}:
%\begin{equation}\label{int517}
% w'(r) =  \left(\frac{\gamma}{4r} - \frac{\gamma r}{2} \right) w(r)= -\gamma (r-r_1) \frac{r+r_1}{2r} w(r).
% \end{equation}
%Finally, \eqref{case3a} follows from \eqref{int517} if $r_2 \leq r \leq 1$ and from \eqref{eH} if $r>1.$
The power weight~$v(r)=(1+r^2)^{-\gamma/2}$ analyzed by
Bobkov and Ledoux in~\cite{bl} does not directly fit the above assumptions;
as we will see below, the natural choice of $r_1$ in this case does depend on $\gamma$,
the difference with our case being the lack of a plateau near zero.
We will indicate changes necessary to accommodate the power weight in Theorem~\ref{thmpower}.
%After the proof of our main weak weighted Poincar\'e inequality, we will indicate changes
%necessary to accommodate the power weight $v.$
% From now on, let $r_1 := \frac{1}{\sqrt{2}}, r_2 := \frac{3}{4}$. The choice $r_2 = \frac{3}{4}$ is arbitrary, and the following results work for any $r_2 \in %(\frac{1}{\sqrt{2}}, 1)$ with adjusted constants and range of validity for $\gamma$ (we are interested in the large $\gamma$ regime in any case).

It will be convenient for us to derive a slightly stronger version of the standard Poincar\'e estimate.
Given any $f(x)$, let
\[
\tilde f(r) := \frac{1}{2\pi} \int_0^{2\pi} f(r,\phi) d\phi.
\]
Instead of directly looking for an upper bound for
\[
Z=\int_{\mathbb{R}^2} (f(x) - \bar f)^2w(x) dx,
~~\bar f=\Big(\int_{\mathbb{R}^2} w(x)dx\Big)^{-1}\int_{\mathbb{R}^2} f(x)w(x)dx.
\]
it turns out to be easier
to control the following integral  that is closely related to $Z(t)$:
\[
I:= \int_{\mathbb{R}^2} (f(x) - \tilde f(r_1))^2 w(x) dx =: I_1 + I_2 + I_3.
\]
Here, $I_1, I_2, I_3$ denote the integrals over the three
sets $B_{r_1}$, $B_{r_2}\setminus B_{r_1}$, and $(B_{r_2})^c$, respectively.
Note that
\[
Z=\int_{\mathbb{R}^2} (f(x) - \bar f)^2w(x) dx=
\inf_{a}\int (f- a)^2w(x) dx\leq I.
\]
%with
%\[
%\bar f=\Big(\int w(x)dx\Big)^{-1}\int f(x)w(x)dx.
%\]
%since the integral  $\int (f- a)^2 e^{H} dx$ is minimized when $a = \bar f$.
Let us also define
\[
J_1 :=  \int_{B_{r_1}} |\nabla f|^2 w(x) dx , \quad
J_2 := \int_{B_{r_2} \setminus B_{r_1}} |\nabla f|^2 w(x) dx,\quad J_3 :=  \int_{(B_{r_2})^c} |\nabla f|^2 |x|^2 w(x) dx.
\]
Note that $J_1$ and $J_2$ are directly related to
\[
W=\int |\nabla f|^2 w(x)\,dx,
\]
but $J_3$ has an extra factor $|x|^2$ in the integrand.
%\subsection{a weighted Poincar\'e inequality}
%\label{subsec:poincare}
%Using the above notations for $I_i$ and $J_i$ for $i = 1, 2, 3$, we prove the following theorem.
\begin{thm}\label{prop:poincare} %\textcolor{blue}{
%Assume that $\gamma$ is sufficiently large ($\gamma \geq 1$ would do).
Suppose that the weight $w(x)\ge 0$ is radial and satisfies
\eqref{case0a}-\eqref{case3a}.
Then there exists a universal constant $C$ such that for all sufficiently large $\gamma$ and every $f$ in the weighted
Sobolev class  $W^{1,2}(w\,dx)$
the  following inequalities hold:
\begin{eqnarray}
%\begin{equation}
\label{final_I1}
&&I_1 \leq C J_1,\\
%\end{equation}
%\begin{equation}
\label{final_I2}
&&I_2 \leq \frac{C}{\gamma} J_2 + \frac{C}{\gamma}  J_1 + \frac14 I_1,\\
%\end{equation}
%\begin{equation}
\label{final_I3}
&&I_3 \leq \frac{C}{\gamma^2}J_3 + \frac{C}{\gamma^2} J_2 + \frac14 I_2.
\end{eqnarray}
\end{thm}
%
%\begin{remark}
%The second terms on the right hand side of \eqref{final_I2} and \eqref{final_I3} are needed only for the non-radial case. Also, the choice of $r_2 := \frac{3}{4}$ is arbitrary, and the proof works for any $r_2 \in (\frac{1}{\sqrt{2}}, 1)$.
%\end{remark}
\noindent\it Remarks. \rm %1. In the proof below we will freely assume that $\gamma$ is large enough, as this is the case we are interested in. \\
%It is not hard to check that the validity of all bounds remains true for all $\gamma \geq 1;$ it is just that for smaller $\gamma,$ the $\gamma$ factors
%in the estimate offer no advantage.    \\
1. As usual, it suffices to prove the inequalities for $f \in C_0^\infty(\R^2).$ \\
2. The factors $1/4$ in estimates \eqref{final_I2} and \eqref{final_I3}
are needed (any factor less than one would work) to derive the sharpest version
of the convergence to equilibrium estimate. \\
3. Here and in the estimates that follow, $C$ and $c$ stand for universal constants (in particular independent of $\gamma$) that may change from expression to expression.
These constants may depend on $r_1,r_2,C_0,C_1$ and $C_2.$ \\
4. The proof extends to all dimensions with a minor adjustment of the constants.
While in dimensions $d \ne 2$ the logarithmic behavior of $H$  does not correspond
%as natural as in dimension two since
to the Green's function of the Laplacian, the behavior of a particle in such
slowly growing potential is of an independent interest.\\
%has a different form. \\
5. All arguments, after minor adjustments, can be made to work for $\gamma > d.$

\begin{proof}
Since we will be working in polar coordinates, it is convenient to incorporate the
Jacobian into the weight, setting $u(r) =r w(r)$.
Let us restate our assumptions on $w$ in terms of $u.$ %\textcolor{blue}{
On the interval $[r_1, r_2]$ we have
\begin{equation}\label{derivative_u}
\begin{split}
u'(r) = rw'(r)+w(r) \leq \left(-C_1 \gamma (r-r_1)+\frac1r \right)u(r).
\end{split}
\end{equation}
Thus if $\gamma$ is large, $u$ is increasing at most for only a small distance past $r_1$,
and reaches its maximum no further than $r_{max} = r_1 + O(\gamma^{-1})$.
In particular, there is $\gamma_0$
large enough so that %for all $r \in [r_1 + \frac{1}{\sqrt{\gamma}},r_2]$, then we have
\[ u'(r) \leq -c\sqrt{\gamma} u(r)~~
\hbox{for all $r \in [r_1 + \dfrac{1}{\sqrt{\gamma}},r_2]$, for $\gamma>\gamma_0$,},
\]
with some $c>0$. %where in the last inequality we assume that $\gamma$ is sufficiently large.
For $r \in [r_2,+\infty)$,
we have
\[
u'(r) \leq (-C_2 \gamma+1) w(r) \leq -C_3 \gamma r^{-1} u(r),
\]
for large enough $\gamma.$

Putting together the observations above, $u$ satisfies the following differential
inequalities, with some $C, c>0$, and $\tilde r_1 := r_1+ {1}/{\sqrt{\gamma}}$:
\begin{subnumcases}{u'(r)\leq\label{ineq_w}}
   Cu(r) & for $r \in [r_1, \tilde r_1),$ \label{case1}\\
   -c\sqrt{\gamma} u(r) & for $ r \in [\tilde r_1, r_2),$ \label{case2}\\
   -c\gamma r^{-1} u(r) & for $r\in[ r_2, +\infty)$.\label{case3}
\end{subnumcases}
These are the inequalities that we will use in the analysis below, along with \eqref{derivative_u}.
%Indeed, the bound \eqref{derivative_w} implies

We first note that (\ref{final_I1})
%$I_1 \leq C J_1
is a direct consequence of a slight variation of the standard proof
of the Poincar\'e's inequality (see, e.g. \cite{Evans}), so we only need to
estimate $I_2$ and $I_3$. We will first show the inequalities for the
radially symmetric $f$, where we only need the first term in the right side
of \eqref{final_I2} and first two terms in the right side of \eqref{final_I3}, respectively, and then consider a general $f$.

\noindent\textbf{$\bullet$ Control of $I_2$: the radial estimates.}
Let $f$ be radial, and $h> 0$ an arbitrary function of single variable, then
\begin{equation}\label{ineq_radial_1}
\begin{split}
\int_{r_1}^{r_2} (f(r) - f(r_1))^2 u(r) dr
&= \int_{r_1}^{r_2} \left( \int_{r_1}^r f'(s) ds\right)^2 u(r) dr\\
&\leq \int_{r_1}^{r_2} \left( \int_{r_1}^r f'(s)^2 h(s) ds\right)  \left( \int_{r_1}^r h(t)^{-1}  dt\right) u(r) dr\\
&= \int_{r_1}^{r_2} f'(s)^2 h(s) \int_s^{r_2} u(r) \left( \int_{r_1}^r h(t)^{-1}  dt\right)dr ds.
\end{split}
\end{equation}
%where $g$ can be any positive function.
We choose $h = u^{1/2}$, and claim that
\begin{equation}
\label{claim1}
\int_s^{r_2} u(r) \left( \int_{r_1}^r u(t)^{-1/2}  dt\right)dr \leq \frac{C}{\gamma} u(s)^{1/2} \quad \text{ for all }s\in[r_1, r_2].
\end{equation}
Once this claim is proved, plugging it into \eqref{ineq_radial_1} yields
\begin{equation}\label{eq_radial1}
I_2=\int_{r_1}^{r_2} ( f - f(r_1))^2 u dr
\leq \frac{C}{\gamma} \int_{r_1}^{r_2} f'(s)^2 u(s) ds=\frac{C}{\gamma}J_2.
\end{equation}
%which gives $I_2 \leq \frac{C}{\gamma} J_2$ for all radially symmetric $f(r)$.

Now, let us prove   \eqref{claim1}. To this end, we will show that
\begin{equation}\label{inner_int}
\int_{r_1}^r u(t)^{-1/2}  dt \leq \frac{C}{\sqrt{\gamma}}u(r)^{-1/2}
\quad\text{ for all }r \in [r_1, r_2],
\end{equation}
and
\begin{equation}
\label{claim2}
\int_s^{r_2} u(r)^{1/2} dr \leq \frac{C}{\sqrt{\gamma}} u(s)^{1/2} \quad \text{ for all }s\in[r_1, r_2],
\end{equation}
which together imply (\ref{claim1}) immediately.
To prove (\ref{inner_int}), we note that, if $r \in [r_1, \tilde r_1],$ \eqref{derivative_u}
implies
\[
\hbox{$u(t)^{-1/2} \leq Cu(r)^{-1/2}$ for any $t \in [r_1, r]$,}
\]
hence \eqref{inner_int} holds for $r \leq \tilde r_1$.
If $r >\tilde r_1$, we split the integration domain in the left side of~(\ref{inner_int}) as
\begin{equation}\label{oct1020}
\int_{r_1}^r u(t)^{-1/2}  dt=\int_{r_1}^{\tilde r_1} u(t)^{-1/2}  dt
+\int_{\tilde r_1}^r u(t)^{-1/2}  dt=A+B.
\end{equation}
%  into $[r_1, \tilde r_1]$ and $[\tilde r_1, r]$.
Again by \eqref{derivative_u}, we have
\[
A\le \frac{C}{\sqrt{\gamma}} u(\tilde r_1)^{-1/2}
\le \frac{C}{\sqrt{\gamma}} u(r)^{-1/2},
\]
as $u(r)$ is decreasing for $r>\tilde r_1$.
For the second integral in (\ref{oct1020}), note that \eqref{case2} gives
\[
\hbox{$u(t)^{-1/2} \leq e^{-c\sqrt{\gamma} (r-t)/2} u(r)^{-1/2}$ for $t \in [\tilde r_1, r]$,}
\]
thus
%This bound implies that the integral over $[\tilde r_1, r]$ is bounded by
\begin{equation}\label{oct1022}
B\le \frac{C}{\sqrt{\gamma}} u(r)^{-1/2}.
\end{equation}
To prove (\ref{claim2}), note that if $s > \tilde r_1$, then (\ref{claim2}) follows
directly from \eqref{case2}, as in (\ref{oct1022}).
If~$s <\tilde r_1$, we again split the integration domain
\begin{equation}\label{oct1024}
\int_{s}^{r_2} u(t)^{1/2}  dt=\int_{s}^{\tilde r_1} u(t)^{1/2}  dt
+\int_{\tilde r_1}^{r_2} u(t)^{1/2}  dt=A+B.
\end{equation}
%into $[s, \tilde r_1)$ and $[\tilde r_1, r_2]$:
The first integral in the right side can be controlled by
\[
A\le \frac{C}{\sqrt{\gamma}} u(s)^{1/2} ,
\]
because $u(t)\le Cu(s)$ on this interval due to \eqref{derivative_u},
and $|\tilde r_1-s|\le 1/\sqrt{\gamma}$.
%$ \frac{C}{\sqrt{\gamma}} u^{1/2}(s) $ due to \eqref{derivative_u}, and
The second integral can be controlled by
\[
B\le \frac{C}{\sqrt{\gamma}} u(\tilde r_1)^{1/2}\le \frac{C}{\sqrt{\gamma}} u(s)^{1/2},
\]
by \eqref{case2} and \eqref{derivative_u}.

\textbf{$\bullet$ Control of $I_2$ for a non-radial function.}
For the general non-radial case, we  decompose a function $f = f(r,\phi)$
into the Fourier series
\begin{equation}\label{f_decomposition}
f(r,\phi) = \tilde f(r) + \sum_{n=1}^\infty (\psi_n(r) \cos(n\phi) + \xi_n(r) \sin(n\phi)).
\end{equation}
Using this decomposition, $I_2$ becomes
\begin{equation}\label{I2_nonradial}
I_2% = \int_{B_{r_2}\setminus B_{r_1}} |f - \tilde f(r_1)|^2 e^H dx
= \int_{r_1}^{r_2} (\tilde f(r) - \tilde f(r_1))^2 u(r)dr + \pi
\sum_{n=1}^\infty   \int_{r_1}^{r_2}   (\psi_n(r)^2 + \xi_n(r)^2) u(r) dr,
\end{equation}
whereas $J_2$ becomes
\begin{equation}\label{J2_nonradial}
J_2 %= \int_{B_{r_2}\setminus B_{r_1}} |\nabla f|^2 e^H dx
=\int_{r_1}^{r_2} \tilde f'(r)^2 u(r)dr
+ \pi\sum_{n=1}^\infty   \int_{r_1}^{r_2}\left(\frac{n^2}{r^2} \psi_n(r)^2
+ \frac{n^2}{r^2} \xi_n(r)^2 +\psi_n'(r)^2 +  \xi_n'(r)^2\right) u(r)dr   .
\end{equation}
Note that $I_1$ and $J_1$ can be written in the same form as $I_2$ and $J_2$ with
the domain of integration replaced by $[0,r_1] $.
To bound $I_2$, we will prove the following estimate for each $n\geq 1$
\begin{equation}\label{eq:phi_1}
\int_{r_1}^{r_2} \psi_n(r)^2 u(r) dr \leq
\frac{C}{\gamma}  \int_{r_1}^{r_2}  \psi_n'(r)^2 u(r) dr +
\frac{C}{\gamma} \int_{0}^{r_1} \psi_n'(r)^2 u(r) dr  +
\farc{1}{4}  \int_{0}^{r_1} \psi_n(r)^2 u(r) dr,
\end{equation}
with an identical estimate holding for $\xi_n$.
With (\ref{eq:phi_1}) in hand, adding~\eqref{eq_radial1}  for $\tilde f$
and \eqref{eq:phi_1} both for $\psi_n$ and $\xi_n$, we arrive at \eqref{final_I2}.

To prove \eqref{eq:phi_1}, first note that
\begin{equation}\label{oct1102}
\begin{aligned}
\int_{r_1}^{r_2} \psi_n(r)^2 u(r) dr
&\leq 2 \int_{r_1}^{r_2} (\psi_n(r) - \psi_n(r_1))^2 u(r) dr +
2\psi_n(r_1)^2  \int_{r_1}^{r_2} u(r) dr\\
&\le
\frac{C}{\gamma}  \int_{r_1}^{r_2}  \psi_n'(r)^2 u(r) dr+
2\psi_n(r_1)^2  \int_{r_1}^{r_2} u(r) dr.
\end{aligned}
\end{equation}
We have used \eqref{eq_radial1} applied to $\psi_n(r)$ in the last inequality above.
%as
%to the first integral in the right side,
%replacing $f$ by $\psi_n$,
%enables us to bound it by $\frac{C}{\gamma}  \int_{r_1}^{r_2} ( \psi_n')^2 u dr$.
%It remains to bound the second term.
To bound the last integral in the right side, we use \eqref{derivative_u} to observe that
$u(\tilde r_1)\le 2 u(r_1)$ for $\gamma$ sufficiently large, and then
also \eqref{case2} to get
\begin{equation}\label{aux1719}
\begin{split}
\int_{r_1}^{r_2} u(r) dr &= \int_{r_1}^{\tilde r_1} u(r) dr + \int_{\tilde r_1}^{r_2} u(r) dr
\leq \frac{2}{\sqrt{\gamma}}u(r_1) +
\int_{\tilde r_1}^{r_2} 2 e^{-c\sqrt{\gamma}(r-\tilde r_1)} u(r_1) dr \\
&\leq \frac{2}{\sqrt{\gamma}}\left(1+\frac1c\right)u(r_1).
\end{split}
\end{equation}
Note that if
\begin{equation}\label{eq:temp1}
4\left(1+\frac1c\right)\psi_n(r_1)^2 \frac{1}{\sqrt{\gamma}} u(r_1)
\leq \frac1{500} \int_0^{r_1} \psi_n(r)^2 u(r) dr,
\end{equation}
then \eqref{eq:phi_1} follows from (\ref{oct1102}), (\ref{aux1719}) and
(\ref{eq:temp1}). If \eqref{eq:temp1} does not hold,
and $\gamma$ is sufficiently large, there exist $C$ independent of $\gamma$, and
$r_3 \in[ r_1 - C/\sqrt{\gamma}, r_1]$ such that $|\psi_n(r_3)| < |\psi_n(r_1)|/2$.
Here, % we assume again that $\gamma$ is large enough so
%that $\tilde C/\sqrt{\gamma} \leq r_1/2,$ and
we used the fact that $u(r)\geq {C_0^{-1}u(r_1)}/{2}$ for all
$r \in [{r_1}/{2}, r_1]$ due to \eqref{case0a}.
%In fact, a computation shows that $\tilde C = 128 (1 +\frac1c)C_0$ would work.
Thus, we have
\begin{equation}\label{aux5171}
\begin{aligned}
\int_{0}^{r_1} \psi_n'(r)^2 u(r) dr &\geq \int_{r_3}^{r_1} \psi_n'(r)^2 u(r) dr
\geq \left(\int_{r_3}^{r_1} \psi_n'(r) dr\right)^2
\Big(\int_{r_3}^{r_1} \frac{1}{u(r)} dr\Big)^{-1}\\
&\geq \frac{|\psi_n(r_1)|^2}{4} \frac{u(r_1)}{2C_0(r_1-r_3)}
\geq \frac{|\psi_n(r_1)|^2}{4} \frac{\sqrt{\gamma}u(r_1)}{2C_0 C},
\end{aligned}
\end{equation}
which, using \eqref{aux1719}, gives that if (\ref{eq:temp1}) fails, then
\[
\psi_n(r_1)^2  \int_{r_1}^{r_2} u(r) dr \leq \frac{C}{\sqrt{\gamma}} \psi_n(r_1)^2  u(r_1)
\leq \frac{C}{\gamma}\int_{0}^{r_1} \psi_n'(r)^2 u(r) dr.
\]
This finishes the proof of \eqref{eq:phi_1}, and hence also of \eqref{final_I2}.

\textbf{$\bullet$ Control of $I_3$: radial estimates.}
To control $I_3$ for a radial function $f$, first note that
\begin{equation}\label{eq_radial_I3}
\int_{r_2}^\infty (f(r)- f(r_1))^2 u(r) \, dr
\leq 2\int_{r_2}^\infty (f(r)- f(r_2))^2 u(r) \, dr
+ 2 (f(r_2) - f(r_1))^2\int_{r_2}^\infty u(r) \, dr.
\end{equation}
We start with the second term in the right side, and claim that
\begin{equation}\label{eq_f_temp}
 (f(r_2) - f(r_1))^2\int_{r_2}^\infty u(r)dr \leq \frac{C}{\gamma^2} J_2.
\end{equation}
To this end,
%control $\int_{r_2}^\infty u(r)dr$, we
note that \eqref{case3}
implies that for all $r > s \geq r_2$ we have
\begin{equation}\label{ff515}
u(r) \leq u(s) \left( \frac{s}{r} \right)^{c\gamma}.
\end{equation}
%and a simple estimate to obtain
%the explicit formula for $u$ in $(1,+\infty)$, and obtain that
Applying with $s=r_2$ we get
\begin{equation}\label{int_w_temp}
\int_{r_2}^\infty u(r)dr \leq \frac{C}{\gamma} u(r_2).
\end{equation}
%\textcolor{blue}{
Also note that
\begin{equation}\label{f_r1r2}
\begin{split}
(f(r_2) - f(r_1))^2 = \left(\int_{r_1}^{r_2} f'(r) dr\,\right)^2
\leq \left(\int_{r_1}^{r_2} f'(r)^2 u(r)\, dr\right)
\left(\int_{r_1}^{r_2} \frac{1}{u(r)} \,dr\right).
\end{split}
\end{equation}
%
% $(f(r_2) - f(r_1))^2 \leq C\int_{r_1}^{r_2} f'(r)^2 dr$. Combining these two inequalities gives
%\begin{equation}\label{eq_f_temp}
%(f(r_2) - f(r_1))^2\int_{r_2}^\infty u(r)dr \leq \frac{C}{\gamma} \int_{r_1}^{r_2} f'(r)^2 u(r_2)dr \leq \frac{C}{\gamma} \int_{r_1}^{r_2} f'(r)^2 u(r)dr\leq \frac{C}{\gamma} J_2,
%\end{equation}
%where in the second inequality we used that $u(r_2)\leq Cu(r)$ for all $r\in [r_1, r_2]$, which follows from \eqref{derivative_u} and \eqref{case2}.
%
Next, we will show that
\[
\int_{r_1}^{r_2} \frac{1}{u(r)}\, dr \leq\frac{C}{\gamma u(r_2)}.
\]
By \eqref{derivative_u}, we have
\begin{equation}\label{oct1602}
\left(\frac{1}{u(r)}\right)' = -\frac{u'(r)}{u(r)^2} \geq
\left(C_1 \gamma (r-r_1)-\frac1r \right)\frac{1}{u(r)}
\quad\text{for }r\in [r_1,r_2].
\end{equation}
Hence, provided that $\gamma$ is sufficiently large,
we have
\[
\left(\frac{1}{u(r)}\right)' \geq  \frac{c \gamma}{u(r)},
\hbox{ for } r \in \Big[\frac{r_1+r_2}{2}, r_2\Big],
\]
with some $c>0$, implying that
\begin{equation}\label{oct1604}
\frac{1}{u(r)} \leq e^{c\gamma(r-r_2)} \frac{1}{u(r_2)},
\hbox{ for } r \in \Big[\frac{r_1+r_2}{2}, r_2\Big].
\end{equation}
By (\ref{oct1602}), we also have
\[
\left(\frac{1}{u(r)}\right)' \geq
- \frac{C}{u(r)}, \,\,\,{\rm thus}\,\,\,\frac{1}{u(r)}
\leq \frac{C}{u(\frac{r_1+r_2}{2})}\leq \frac{Ce^{-c\gamma(r_2-r_1)/2}}{u(r_2)},
\hbox{ for } r\in\Big[r_1, \frac{r_1+r_2}{2}\Big].
\]
We used (\ref{oct1604}) with $r=(r_1+r_2)/2$ in the last inequality above.
Putting these estimates together yields
\[
\int_{r_1}^{r_2} \frac{1}{u(r)}\,dr \leq C\int_{r_1}^{\frac{r_1+r_2}{2}} e^{-c\gamma(\frac{r_2-r_1}{2})} \frac{1}{u(r_2)}\, dr  + \int_{\frac{r_1+r_2}{2}}^{r_2}  e^{c\gamma(r-r_2)} \frac{1}{u(r_2)}\, dr  \leq   \frac{C}{\gamma u(r_2)}.
\]
Combining this bound with \eqref{f_r1r2} and \eqref{int_w_temp} gives us \eqref{eq_f_temp}.

For the first integral in the right side of \eqref{eq_radial_I3},
a computation identical to \eqref{ineq_radial_1}, but
with~$r_1$ replaced by $r_2$, and $r_2$ replaced by $\infty$, yields
\begin{equation}\label{ineq_radial_2}
\begin{split}
\int_{r_2}^{\infty} (f(r) - f(r_2))^2 u(r) dr \leq  \int_{r_2}^{\infty} f'(s)^2 h(s) \int_s^{\infty} u(r) \left( \int_{r_2}^r h(t)^{-1}  dt\right)dr ds,
\end{split}
\end{equation}
for any $h>0$. We again choose $h = u^{1/2}$, and claim that
\begin{equation}
\label{claim32}
\int_s^{\infty} u(r) \left( \int_{r_2}^r u(t)^{-1/2}  dt\right)dr
\leq \frac{C}{\gamma^2 }s^2 u^{1/2}(s) \quad \text{ for all }s \geq  r_2,
\end{equation}
with some $C>0$ (to be shown below). Substituting this into \eqref{ineq_radial_2} gives
\begin{equation}\label{eq_radial2}
\int_{r_2}^{\infty} ( f(r) - f(r_2))^2 u(r) dr
\leq \frac{C}{\gamma^2} \int_{r_2}^{\infty} (f'(s))^2 s^2 u(s) ds,
\end{equation}
and combining it with \eqref{eq_f_temp} and \eqref{eq_radial_I3} yields
\begin{equation}\label{eq_radial_I3_final}
\int_{r_2}^\infty (f(r)- f(r_1))^2 u(r)dr \leq   \frac{C}{\gamma^2}
\int_{r_2}^{\infty} (f'(s))^2 s^2 u(s) ds +
\frac{C}{\gamma^2} \int_{r_1}^{r_2} f'(r)^2 u(r)dr.
\end{equation}
That is, we have
\[
I_3 \leq \frac{C}{\gamma^2} J_3 + \frac{C}{\gamma^2} J_2,
\]
for all radially symmetric $f$.

To show \eqref{claim32}, we consider the inner integral first. Using \eqref{ff515}, we get that if $r_2 \leq t < r$ then
\[
u(r)^{-1/2} \left( \frac{t}{r} \right)^{c\gamma} \geq u(t)^{-1/2},
\]
so that
\[
\int_{r_2}^r u(t)^{-1/2}\,dt \leq u(r)^{-1/2} r^{-c\gamma}
\int_{r_2}^r t^{c\gamma}\,dt \leq \frac{C}{\gamma} u(r)^{-1/2}r.
\]
Thus the left hand side of \eqref{claim32} is bounded from above by
\[
 \frac{C}{\gamma} \int_s^{\infty} u^{1/2}(r) rdr \leq  \frac{C}{\gamma^2 }s^2 u^{1/2}(s) \text{ for all }s> r_2.
\]
The last inequality follows from \eqref{ff515} with $r_2$ replaced by $s$ and a direct computation.

%To show \eqref{claim32}, we consider the inner integral first. If $r \in [r_2, 1]$, then a simple estimate utilizing \eqref{case3} yields
%\[
%\int_{r_2}^r u^{-1/2}(t)  dt \leq \frac{C}{\gamma} u^{-1/2}(r) \leq \frac{C}{\gamma} u^{-1/2}(r) r.
%\] If $r> 1$, then the explicit formula $u(r) = r^{-\gamma/4+1}$ for $r>1$ gives that
%\[\int_{r_2}^r u^{-1/2}(t)  dt \leq \int_{r_2}^1 u^{-1/2}(t)  dt  + \frac{C}{\gamma} u^{-1/2}(r)r \leq \frac{C}{\gamma} u^{-1/2}(r)r.\]
%Thus the left hand side of \eqref{claim32} is bounded from above by
%\[
% \frac{C}{\gamma} \int_s^{\infty} u^{1/2}(r) rdr \leq  \frac{C}{\gamma^2 }s^2 u^{1/2}(s) \text{ for all }s> r_2.
%\]
%The last inequality follows from the explicit formula for $u$ if $s \geq 1$; if $s \in [r_2, 1)$, it can be derived from the estimate \eqref{case3} in
%the integral domain $[s,1]$ and the explicit formula in $[1,+\infty)$.

\textbf{$\bullet$ Control of $I_3$ for a nonradial function.}

For a general function $f$, using the decomposition \eqref{f_decomposition},
we can write $I_3$ and $J_3$ as
\begin{eqnarray}\label{I3_nonradial}
&&I_3% = \int_{B_{r_2}\setminus B_{r_1}} |f - \tilde f(r_1)|^2 e^H dx
= \int_{r_2}^{\infty} (\tilde f(r) - \tilde f(r_1))^2 u(r)dr +
\sum_{n=1}^\infty   \int_{r_2}^{\infty}\pi  (\psi_n(r)^2 + \xi_n(r)^2) u(r) dr,\\
\label{J3_nonradial}
&&J_3 %= \int_{B_{r_2}\setminus B_{r_1}} |\nabla f|^2 e^H dx
=\int_{r_2}^{\infty} \tilde f'(r)^2 r^2 u(r)dr \\
&&~~~+\pi
\sum_{n=1}^\infty   \int_{r_2}^{\infty}\Big(\frac{n^2}{r^2} \psi_n(r)^2
+ \frac{n^2}{r^2} \xi_n(r)^2 + \psi_n'(r)^2 + \xi_n'(r)^2\Big) r^2 u(r)dr.
\nonumber
\end{eqnarray}

We now aim to show the following estimate for each $\psi_n$,
$n\geq 1$:
\begin{equation}\label{eq:phi_2}
\int_{r_2}^{\infty} \psi_n^2(r) u(r) dr
\leq \frac{C}{\gamma^2}  \int_{r_2}^{\infty} \psi_n'(r)^2 r^2 u(r) dr
+ \frac{C}{\gamma^2} \int_{r_1}^{r_2} \psi_n'(r)^2 u(r) dr
+    \frac14 \int_{r_1}^{r_2} \psi_n(r)^2 u(r) dr.
\end{equation}
Combining (\ref{eq:phi_2}) with the analogous estimate for $\xi_n$
and the radial estimate \eqref{eq_radial_I3_final}, we will have \eqref{final_I3}.

First, we write
\begin{equation}\label{oct1608}
\int_{r_2}^\infty \psi_n(r)^2 u(r)dr \leq
2 \int_{r_2}^\infty (\psi_n(r) - \psi_n(r_2))^2 u(r) dr
+ 2 \psi_n(r_2)^2\int_{r_2}^\infty  u(r) dr.
\end{equation}
Applying \eqref{eq_radial2} to the first integral in the right side
%(and replacing $f$ by $\psi_n$) enables us to bound it by
gives
\[
\int_{r_2}^\infty (\psi_n(r) - \psi_n(r_2))^2 u(r) dr\le
\frac{C}{\gamma^2}  \int_{r_2}^{\infty} \psi_n'(r)^2 r^2 u(r) dr.
\]
For the second term in the right side of (\ref{oct1608}),
by \eqref{int_w_temp} we have
\[
2 \psi_n(r_2)^2\int_{r_2}^\infty  u(r) dr \leq \frac{C}{\gamma}\psi_n(r_2)^2 u(r_2).
\]
Thus, if
\begin{equation}\label{oct1612}
\frac{C}{\gamma}\psi_n(r_2)^2 u(r_2) \leq  \frac14  \int_{r_1}^{r_2} \psi_n(r)^2 u(r) dr,
\end{equation}
we are done. If not, since $u$ is decreasing in $(\tilde r_1, r_2)$,
there exists $r_3 \in [r_2-\frac{16 C}{\gamma}, r_2)$, such that~$\psi_n(r_3) \leq \psi_n(r_2)/2$. Then, we have
\begin{equation}\label{oct1614}
\begin{aligned}
\int_{r_1}^{r_2}  \psi_n'(r)^2 u(r) dr &\geq
\int_{r_3}^{r_2}  \psi_n'(r)^2 u(r) dr \geq \left( \int_{r_3}^{r_2} |\psi_n'(r)| dr\right)^2
\Big({\int_{r_3}^{r_2} \frac{dr}{u(r)}}\Big)^{-1} \\
& \geq C\gamma \psi_n(r_2)^2 u(r_2).
\end{aligned}
\end{equation}
In the last step we used that
\[
\int_{r_3}^{r_2} \frac{1}{u(r)} dr \leq \frac{C}{\gamma u(r_2)}.
\]
The latter inequality follows from the decay
of $u$ on $[r_3,r_2]$ and $r_2-r_3 \leq  16 C/\gamma.$ Thus, if~(\ref{oct1612})
fails, then
\[
2 \psi_n(r_2)^2\int_{r_2}^\infty  u(r) dr \leq \frac{C}{\gamma}\psi_n(r_2)^2 u(r_2) \leq \frac{C}{\gamma^2} \int_{r_1}^{r_2}  (\psi_n')^2 u dr ,
\]
which finishes the proof of \eqref{eq:phi_2}.
\end{proof}

Theorem~\ref{prop:poincare} leads to the following two corollaries. Note that
adding the inequalities in the theorem together, we get
\begin{equation}\label{eq:goal_poincare}
I = I_1 + I_2 + I_3 \leq C\left(J_1 + \frac{1}{\gamma} J_2 + \frac{1}{\gamma^2} J_3\right).
\end{equation}
We also recall, as already noted in the remarks  to Theorem~\ref{prop:poincare}, that the arguments above generalize to an arbitrary dimension $d >2$ in a straightforward manner.
This implies
\begin{cor}\label{cor:poincare1} Suppose that $d \geq 2.$
For the weight $w$ satisfying \eqref{case0a},  \eqref{derivative_w} and \eqref{case3a}, we have
\begin{equation}\label{eq_cor}
\begin{aligned}
\int_{\mathbb{R}^d} |f-\bar f|^2 w dx &\leq
\int_{\mathbb{R}^d} |f-\tilde f(r_1)|^2 w dx \\
&
\leq  C\left( \int_{B_{r_1}} |\nabla f|^2 w dx + \frac{1}{\gamma} \int_{B_{r_2} \setminus B_{r_1}} |\nabla f|^2 w dx +  \frac{1}{\gamma^2} \int_{B_{r_2}^c} |\nabla f|^2 |x|^2 w dx  \right) %\nonumber
\end{aligned}
\end{equation}
for all sufficiently large $\gamma.$
\end{cor}
The one dimensional results are in fact stronger and will be considered elsewhere.

Tracing through the proof of Theorem~\ref{prop:poincare}, it is straightforward to check that the result remains true for truncated integrals.
\begin{cor}\label{cor:truncation}
For any $R>r_2$, let us define $I_3^R$ and $J_3^R$ as the truncation of $I_3$ and $J_3$ to the integration domain $B_R \setminus B_{r_2}$. Likewise, let $I_R$ denote the truncation of $I$ to $B_R$. Then we have
\begin{equation}\label{ineqR1}
I_3^R \leq \frac{C}{\gamma^2}J_3^R + \frac{C}{\gamma^2} J_2 + \frac14 I_2,
\end{equation}
and
\begin{equation}\label{ineqR2}
I_R \leq  C\left(J_1 + \frac{1}{\gamma} J_2 + \frac{1}{\gamma^2} J_3^R\right).
\end{equation}
\end{cor}

We now pause to indicate a result that can be obtained with similar techniques
for the power weight $v(x)=(1+|x|^2)^{-\gamma/2}$ with a sufficiently large $\gamma.$
\begin{thm}\label{thmpower}
Let $v(x)=(1+|x|^2)^{-\gamma/2}.$ Then the following weak weighted Poincar\'e inequality holds for all dimensions $d \geq 2$ for sufficiently large $\gamma:$
\begin{equation}\label{powerpoin} \int_{\R^d} |f-\overline{f}|^2 v(x)\,dx \leq
\frac{C(d)}{\gamma} \int_{B_{1}} |\nabla f|^2 v(x)\,dx + \frac{C(d)}{\gamma^2}
\int_{(B_1)^c} |\nabla f|^2 (1+|x|^2) v(x)\,dx. \end{equation}
\end{thm}
\begin{proof}
In two dimensions, the only essential difference is that for the weight $v(x)$
the condition \eqref{case0a} holds if we choose $r_1 \lesssim \gamma^{-1/2}$
that depends on $\gamma$. Specifically, we could take $r_1 = {2}/{\sqrt{\gamma}}.$
With such choice, direct computations show that for the weight $u(r)=rv(r)$ the inequality \eqref{derivative_u} remains valid, while \eqref{case2} and \eqref{case3}
hold with $\tilde r_1 = r_1$ and $r_2=1.$ The standard Poincar\'e inequality
%$I_1 \leq CJ_1$
becomes
\begin{equation}\label{blIcon}
I_1 \leq Cr_1^2 J_1 = \frac{C}{\gamma} J_1,
\end{equation}
as it is important to keep track of the $r_1^2$ factor which now depends on $\gamma.$
The rest of the of the proof goes through.
One place that requires attention and minor adjustment is the control of~$I_2$ for
a nonradial function, namely the estimates \eqref{eq:temp1} and \eqref{aux5171}.
since we need to "step back" a distance $\tilde C/\sqrt{\gamma}$ into the $[0,r_1]$ region, and we may not have that much space.
However, the factor $1/4$ in \eqref{final_I2} is not crucial for establishing \eqref{powerpoin} given that we can control
$I_1$ via \eqref{blIcon}; any constant would do.
Then we can choose sufficiently large constant $C$ instead of $1/4$ in \eqref{eq:temp1} so that $r_3$ with needed properties can be found in $[r_1-\frac{1}{\sqrt{\gamma}},r_1.]$
With this modification, the rest of the argument goes through. We leave details to the interested reader.
Finally, as we already noted above, the proof generalizes to an arbitrary dimension $d$ with minor adjustments.
\end{proof}

\section{Convergence to equilibrium estimates for Fokker-Planck operators}\label{sec:fp}

\subsection{Weighted $L^2$ norm decay}

With the weak weighted Poincar\'e inequalities in hand, we may now go back to the dual
evolution (\ref{eq:f}) and  the dissipation inequality (\ref{dzdt}):
\begin{equation}\label{oct1618}
\farc{dZ}{dt}=-2W(t),
\end{equation}
with
\[
Z(t)=\int_{\mathbb{R}^2} (f(x,t) - \bar f)^2 e^{H(x)} dx,~~
W(t)=\int_{\mathbb{R}^2} |\nabla f(x,t)|^2  e^{H(x)} dx.
\]
%The corresponding dual evolution is
%\begin{equation}
%\label{eq:f}
%\partial_t f - \Delta f - \nabla H \cdot \nabla f = 0.
%\end{equation}
%Note that $\rho(\cdot, t)$ solving \eqref{eq:fp_const} is equivalent to
%\[
%f(x, t) := \rho(x, t)e^{-H(x)}
%\]
%solving \eqref{eq:f}.
%The evolution \eqref{eq:f} conserves the integral of $f(x) \exp(H(x))$
%so we expect that
%\[
%f(x,t)\to\bar f := \Big(\int f_0e^H dx\Big)\Big(\int e^H dx\Big)^{-1}, \hbox{ as $t\to+\infty$,}
%\]
%where $f_0(x)=f(x,0)$.
%%which is comparable to $\frac{M}{(\frac{e}{2})^{\gamma/4}}$.
%Note that we have
%\begin{equation}\label{dzdt}
%\frac{d}{d t} \underbrace{\int_{\mathbb{R}^2} (f(x,t) - \bar f)^2 e^{H(x)} dx}_{=:Z(t)} = -2\underbrace{\int_{\mathbb{R}^2} |\nabla f(x,t)|^2  e^{H(x)} dx}_{=:W(t)}.
%\end{equation}
%We now obtain explicit estimates on the decay rate of $Z(t)$, using \eqref{dzdt} and the weak
%weighted Poincar\'e inequality above.
We are going to focus on the specific weight in \eqref{eH}; % that we need in the application we have in mind:
we will need fairly sharp estimates to get close to the heuristic bounds.
Our analysis in this section will be driven by the
nonlinear application we have in mind:  to derive
sharp bounds on the time required
to transport a significant part of density towards the center of the attracting potential, the ball $B_{r_1}.$
We stress that in this section, we do not need the initial data $f_0(x)$ or $\rho_0(x)$ to be radial: the bounds on convergence
to equilibrium in linear setting with a fixed potential apply in full generality.
%Following the argument below, it is not hard to derive a bit weaker but more general convergence to equilibrium estimates for a class of Fokker-Planck operators
%with logarithmic-type potentials, but we do not focus on this here.

Although Corollary \ref{cor:poincare1} with $w(x)=e^{H(x)}$
already gives us an upper bound for $Z(t)$, we cannot directly control the right side
of (\ref{eq_cor}) by $W(t)$, due to the extra factor $|x|^2$ in the integrand of $J_3$. To overcome this issue, let us take a truncation at radius $R \geq r_2$ in $I(t)$ and apply Corollary \ref{cor:truncation}: %This leads to the estimate
\begin{equation}\label{eq:I_temp}
\begin{split}
I(t) &\leq I_R(t) + \frac{C}{\gamma}R^{-\frac{\gamma}{4} + 2} \|f(\cdot, t)\|_\infty^2
\leq C(J_1 + \frac{1}{\gamma} J_2 + \frac{1}{\gamma^2} J_3^R) + \frac{C}{\gamma}R^{-\frac{\gamma}{4} + 2} \|f_0\|_\infty^2\\
&\leq CW(t) + C\frac{R^2}{\gamma^2}W(t) + \frac{C}{\gamma}R^{-\frac{\gamma}{4} + 2} \|f_0\|_\infty^2.
\end{split}
\end{equation}
In the first inequality above we used (\ref{eH}), and in the second
the fact that $\|f(\cdot,t)\|_\infty$ is non-increasing in time, as well as
(\ref{ineqR2}).
%Since the above estimate holds for any $R \geq r_2$, we can optimize $R$ to minimize the right
%hand side. The optimal $R$ is
To optimize the right side of (\ref{eq:I_temp}) over $R\ge r_2$, we take
\[
R =  \gamma^{4/\gamma} W(t)^{-\frac{4}{\gamma}}\|f_0\|_\infty^{\frac{8}{\gamma}}.
\]
Note that if the radius $R$ defined this way satisfies $R<1,$ then $I(t) \leq CW(t)$ fitting the scheme below.
As
\begin{equation}\label{oct1620}
\gamma^{8/\gamma} \lesssim 1,
\end{equation}
this leads to
\begin{equation}\label{eq_I_temp2}
I(t) \leq CW(t) + C \gamma^{-2} W(t)^{\frac{\gamma-8}{\gamma}}
\|f_0\|_\infty^{\frac{16}{\gamma}}
\leq 2C \max\{W(t),\gamma^{-2} W(t)^{\frac{\gamma-8}{\gamma}}
\|f_0\|_\infty^{\frac{16}{\gamma}}\}.
\end{equation}
%(note that $\gamma^{8/\gamma} \lesssim 1$).
%Therefore, $I(t)\leq 2C \max\{W(t),\gamma^{-2}
%W(t)^{\frac{\gamma-8}{\gamma}} \|f_0\|_\infty^{\frac{16}{\gamma}}\}$.
Since $Z(t) \leq I(t)$, it follows that (\ref{oct1618}) and (\ref{eq_I_temp2}),
together with (\ref{oct1620}) imply
%the above inequality implies that
\begin{equation}\label{diff_ineq1}
Z'(t) \leq -c \min\left\{ Z(t), \gamma^2 Z(t)^{\frac{\gamma}{\gamma-8}} \|f_0\|_\infty^{-\frac{16}{\gamma-8}} \right\}.
\end{equation}

Let us now discuss how this differential inequality relates to the heuristic
bound~$\tau_C\sim L^2/\gamma$ for the reaction time
we have informally derived in Section~\ref{heur},
to give context and outline the main ideas behind the technical estimates that follow.
Let us think for now of  the linear
Fokker-Planck operator~\eqref{eq:fp_const} with the potential $H$ given by \eqref{H:explicit}.
Consider an initial condition $\rho_0$ that has total mass~$M_0$
and is concentrated at a distance $L$ from the origin.
%The heuristic argument
%of Section~\ref{heur} estimates the transport time $\tau$ of the
%process as $\tau \sim L^2/\gamma.$ We can set $f = \rho e^{-H},$
Then~$f=\rho e^{-H}$ solves the dual Fokker-Planck equation \eqref{eq:f}, and \eqref{diff_ineq1} is applicable.
If we drop the term $Z(t)$ from the minimum in~\eqref{diff_ineq4} (which of course strengthens the differential inequality compared to what we really have), then a direct computation,
with yet another use of (\ref{oct1620}), gives
\begin{equation}\label{Zdec519}
Z(t) \leq  \left( Z(0)^{-\frac{8}{\gamma-8}}
+ c\gamma t \|f_0\|_\infty^{-\frac{16}{\gamma-8}}\right)^{-{(\gamma-8)}/{8}}
\leq (c\gamma t)^{-{(\gamma-8)}/{8}} \|f_0\|_\infty^2.
\end{equation}
In our situation, we have $\|f_0\|_\infty \sim M_0 L^{\gamma/4}$ according to the
assumptions on $\rho_0$ and \eqref{eH}. Also, using the relationship between $\rho$ and $f$, we see that
\begin{equation}\label{def:Y}
%Y(t):
Z(t)= \int_{\R^2} |\rho(x,t) - \rho_s(x)|^2 e^{-H}\,dx. %=Z(t).
\end{equation}
%equals $Z(t),$
Here,
\[
\rho_s(x) = e^H \frac{\int \rho_0\,dx}{\int e^H \,dx}
\]
is the stationary state of the same mass as $\rho_0$ to which the solution $\rho$ converges.
From \eqref{def:Y} it is clear that transport of $\rho$ to the origin corresponds to decay of
%$Y(t)$, or, as it is the same thing,
$Z(t).$
Intuitively, from \eqref{eH} it looks likely that we need %$Y(t)= Z(t) \ll M^2$
$Z(t) \ll M^2_0$  in order to be sure that a significant portion of $\rho$ is inside~$B_{r_1}$
(we will make these
arguments precise later). Going back to \eqref{Zdec519} and the estimate on $\|f_0\|_{\infty},$ we find that to ensure the needed bound on $Z(t)$, we need
$t \gtrsim L^4/\gamma,$ which is quite a bit off the heuristic estimate. The situation is similar to the usual heat equation, where the~$L^1$ to~$L^2$ estimate
decays only  as $t^{-d/4},$ while the faster decay rate $t^{-d/2}$ is realized
for the~$L^1$ to~$L^\infty$ estimate. A standard way to attain the
%For the pure heat equation the
latter estimate if explicit heat kernel is not available (like for diffusions with incompressible drift,
see e.g. \cite{FKR}) is
%immediate; but for more complex settings where, for example, a non-trivial drift is also
%involved, it can be obtained by combining
to combine the $L^1$ to $L^2$
bound with its dual $L^2$ to $L^\infty$ bound. %(see e.g. \cite{FKR}).
We will need to follow a similar route in what follows.
The~$L^\infty$ to $L^2(e^Hdx)$ bound \eqref{Zdec519}
provides a decay estimate for $Z(t)$, which via (\ref{def:Y})
%by direct substitution $\rho=fe^{-H},$
leads to the $L^\infty(e^{-H}dx)$ to $L^2(e^{-H}dx)$ bound for $\rho.$
We will also derive a dual
to \eqref{Zdec519} bound, which is an $L^2(e^{-H}dx)$ to $L^1$ estimate for $\rho.$
Combining them leads to $L^{\infty}(e^{-H}dx)$ to $L^1$ bound for~$\rho$ which will
have the needed decay and also will provide control in the $L^1$ space
most convenient for measuring mass transport.

Before we go to the duality estimates, however, there is one more issue to take care of. The
presence of the term $Z(t)$ under minimum in \eqref{diff_ineq1}
affects the bound \eqref{Zdec519}. The balance of the two terms depends on
the initial data; the second term is smaller if $Z(t)$ is sufficiently small,
namely if
\[ Z(t) \lesssim \|f_0\|_{L^\infty}^2 \gamma^{-\frac{\gamma-8}{4}}.
\]
Our assumptions on $\rho_0$ give $Z(0)\sim M_0^2 L^{\gamma/4}$, and the above condition at $t=0$
translates into an additional constraint~$L \gtrsim \gamma.$
%which is an additional constraint.
For some configurations of parameters, say
when $1 \ll L \ll \gamma,$ the time delay before the second term in \eqref{diff_ineq1} becomes
smaller can be up to order $\gamma \log \gamma.$
We would like to avoid these additional constraints and significant losses in the estimate of
the transport time, as they appear to be of technical nature.
The idea is to use the $L^\infty$ norm time decay estimate proved in Theorem~\ref{thm:L_infty_bound}.
This gives an outline for the rest of this section.
First, we deploy the $L^\infty$ norm decay bound to improve the weighted $L^2$
control on $f$ and $\rho,$ and then use duality argument to obtain optimal convergence to equilibrium bounds for $\rho$ in~$L^1.$

%Although this differential inequality gives us some explicit decay rate of $Z(t)$, it is not satisfactory for our application purposes for the following reason. Recall that our goal is to show that by time $\frac{L^2}{\gamma}$, a large fraction of mass is attracted into the unit ball. Comparing the two terms in the bracket, we get that the first $Z(t)$ term is active as long as $Z(t) \geq \|f_0\|_\infty^2 \gamma^{\frac{\gamma}{2}-2}$.

The differential inequality \eqref{diff_ineq1} can be improved in the following way for $t \gtrsim 1$.
In the second inequality of \eqref{eq:I_temp}, instead of using \eqref{ineqR2} to bound the whole $I_R$, we can instead split
\[
I_R=I_1 + I_2 + I_3^R,
\]
and  directly control $I_1$ and $I_2$ as follows.
%(We still use \eqref{ineqR1} to bound $I_3^R$).
The bound in Theorem~\ref{thm:L_infty_bound} implies that
\[
\|f(\cdot, t)e^H\|_{\infty} = \|\rho(\cdot, t)\|_{\infty}
\leq C\gamma \|\rho_0\|_1 = C\gamma \|f_0 e^H\|_1
\]
for all $t\geq 1$. Two immediate consequences are
\begin{eqnarray}\label{eq:def_Z1}
&&\!\!\!\!\!\!\!\!\!
I_1(t) = \int_{B_{r_1}} |f - \tilde f(r_1)|^2 e^H dx
\leq C \gamma^2 \|f_0 e^H\|_1^2 e^{-H(0)}   =: Q_1
\quad \text{ for all }t\geq 1,%\frac{2M^2}{\int e^H dx} + C\gamma^2 M^2 (\frac{e}{2})^{-\gamma/4} \leq C\gamma^2 M^2 (\frac{e}{2})^{-\gamma/4} =: Z_1
\\
\label{eq:def_Z2}
&&\!\!\!\!\!\!\!\!\!
I_1(t) + I_2(t) = \int_{B_{r_2}} |f - \tilde f(r_1)|^2 e^H dx \leq C \gamma^2 \|f_0 e^H\|_1^2 e^{-H(r_2)} =: Q_2 \quad \text{ for all }t\geq 1.  %\int_{B_{r_2}\setminus B_{r_1}} |f-\tilde f(r_1)|^2 e^H dx \leq C \int_{B_{r_2}\setminus B_{r_1}} (\gamma M)^2 e^{-H} dx \leq C(\gamma M)^2 e^{-\gamma/32} =: Z_2,
\end{eqnarray}
Recall that $r_1=1/\sqrt{2}$ and $r_2=3/4$, as defined in (\ref{oct2302}).
Note that $Q_2 \gg Q_1$ due to $\gamma \gg 1$ and \eqref{eH}; hence,
for $t\geq 1$, if $I(t) \geq 4 Q_2$,
then we can bound $I_R$ by \eqref{eq:def_Z2} and \eqref{ineqR1} as follows:
\[
I_R(t) \leq Q_2 + I_3^R \leq Q_2 + \frac{C}{\gamma^2} J_2 + \frac{C}{\gamma^2}J_3^R + \frac14 I_2.
\]
Substituting this into the second inequality of \eqref{eq:I_temp},
and then absorbing $Q_2$ and $\frac{1}{4} I_2$ into the left side, we obtain
\[
I(t) \leq \frac{C}{\gamma^2} W(t) + C\frac{R^2}{\gamma^2}W(t) + \frac{C}{\gamma}R^{-\frac{\gamma}{4} + 2} \|f_0\|_\infty^2 \leq \frac{C}{\gamma^2}W(t) + \frac{C}{\gamma^2} W(t)^{\frac{\gamma-8}{\gamma}} \|f_0\|_\infty^{\frac{16}{\gamma}},
\]
where the last inequality comes from choosing the same optimal $R$ as before,
since the terms containing $R$ are the same as in \eqref{eq:I_temp} (and again, if we get $R \leq r_2$
then $I(t) \leq \frac{C}{\gamma^2}W$ fitting the scheme below).
The $\gamma^{-2}$ factor in the first term then leads to a stronger differential inequality:
\begin{equation}\label{diff_ineq2}
Z'(t) \leq -c\gamma^2 \min\left\{Z(t), Z(t)^{\frac{\gamma}{\gamma-8}} \|f_0\|_\infty^{-\frac{16}{\gamma-8}} \right\}.
\end{equation}

Likewise, for $t\geq 1$, if $I(t) \in [4Q_1, 4Q_2]$, then we control $I_R$ using \eqref{eq:def_Z1}, \eqref{final_I2} and \eqref{ineqR1}:
\[
I_R(t) \leq Q_1 + I_2 + I_3^R \leq Q_1  + \frac{C}{\gamma^2}J_3^R + \frac{C}{\gamma} J_2 + \frac{1}{4} I_1 + \frac{1}{4}I_2.
\]
and a similar to the above argument leads to the differential inequality
\begin{equation}\label{diff_ineq3}
Z'(t) \leq -c \min\left\{\gamma Z(t), \gamma^2 Z(t)^{\frac{\gamma}{\gamma-8}} \|f_0\|_\infty^{-\frac{16}{\gamma-8}} \right\}.
\end{equation}

%Now let us take $t=1$ as the starting time. The inequalities \eqref{diff_ineq1}, \eqref{diff_ineq2}, \eqref{diff_ineq3} control convergence of $Z(t)$ to zero.  Choosing the best differential inequalities from the three, we will use \eqref{diff_ineq2} for $Z\geq 4Z_2$, \eqref{diff_ineq3} for $Z(t) \in [4Z_1, 4Z_2]$, and \eqref{diff_ineq1} for $Z(t) \leq 4Z_1$. Note that each differential inequality has a min function on the right hand side.
%Generally speaking, which part on the right hand side will dominate on which time interval may depend on the initial data (more precisely, the relationship between $\|f_0\|_\infty$ and $Z_1(0), Z_2(0)$).

For all $t\geq 1$, the inequalities \eqref{diff_ineq1}, \eqref{diff_ineq2}, \eqref{diff_ineq3} control convergence of $Z(t)$ to zero.
%Choosing the best differential inequality from the three,
The above results are summarized in the following proposition.

\begin{prop} For all $t\geq 1$, $Z(t)$ satisfies the following differential inequality:
\begin{equation}\label{diff_ineq4}
Z'(t) \leq -c \min\left\{ \eta(Z(t)) Z(t), \gamma^2 Z(t)^{\frac{\gamma}{\gamma-8}} \|f_0\|_\infty^{-\frac{16}{\gamma-8}} \right\},
\end{equation}
where
\begin{equation}\label{etaeq}
\eta(Z) := \begin{cases}
1 &  \text{ for } Z(t) \leq 4Q_1\\
\gamma &  \text{ for }  Z(t) \in (4Q_1, 4Q_2)\\
\gamma^2 &  \text{ for }   Z(t) \geq 4Q_2.
\end{cases}
\end{equation}
\end{prop}

%Note that in all the three regimes, the minimum function in \eqref{diff_ineq4} will take the first argument for $Z$ large, and take the second argument for $Z$ small, if
Due to the minimum taken in \eqref{diff_ineq4}, which part will
dominate depends on the initial data, or, more precisely, on the
relationship between $\|f_0\|_\infty$, $Q_1,$ and $Q_2$.
A careful accounting needs to take care of several cases; however,
%Instead of deriving the most general bound for the decay of $Z(t)$ for all $t>0$,
it turns out that for the sake of the application at hand,
we only need to track the decay of $Z(t)$ until it drops to $Z^\sigma$, defined as
\begin{equation}\label{def:Zeps}
Z^\sigma :=\sigma e^{-H(0)} \|f_0 e^H\|_1^2,
\end{equation}
where $\sigma< 1$ is sufficiently small. The definition of $Z^\sigma$ is
motivated by Proposition~\ref{prop2} below. Basically, we will see that by the time $Z(t)$ reaches $Z^\sigma$, a significant portion of the mass of $\rho = fe^H$ has already moved into $B_{r_1}$, which will be sufficient to prove that significant reaction took place.

The following theorem says that even with the first item in the minimum
in \eqref{diff_ineq4}, the decay of $Z(t)$ is not too much worse than in \eqref{Zdec519} -- namely, as long as $Z(t)$ is above $Z^\sigma$, the presence of the first item in the
min function introduces at most  an extra time delay $t_1$
which is estimated below (and is much better than $\gamma$).

\begin{thm}\label{thm1}Let $f(x,t)$ be the solution to \eqref{eq:f} with initial
condition  $f_0 \in L^\infty(\mathbb{R}^2) \cap L^2(e^H)$, and let $Z(t)$ and  $Z^\sigma$ be given as in \eqref{dzdt} and \eqref{def:Zeps} respectively. Let
$t_1 := C(1 + \log \sigma^{-1} + \log \gamma),$ where $C$ is a sufficiently large universal constant. Then for all $t\geq t_1$, we have
\begin{equation}\label{eq9}
%Z(t) \leq \min\left( \tilde Z, \frac{Z(0)}{\left(1+C\gamma^{1+\frac{4}{\gamma-4}}(t-t_1)\|f\|_\infty^{-\frac{8}{\gamma-4}} Z(0)^{\frac{4}{\gamma-4}} \right)^{\frac{\gamma-4}{4}}}\right).
\begin{split}
Z(t)& \leq \max\left\{ Z^\sigma, (c\gamma (t-t_1))^{-\frac{\gamma-8}{8}} \|f_0\|_\infty^2\right\}.
%Z(t)& \leq \max\Big\{ Z^\epsilon, \left( Z(0)^{-\frac{4}{\gamma-4}} + c\gamma (t-t_1) \|f_0\|_\infty^{-\frac{8}{\gamma-4}}\Big)^{-\frac{\gamma-4}{4}}\right\}.\\
%&\leq \max\left\{ \tilde Z, \left(c\gamma (t-t_1)\right)^{-\frac{\gamma-4}{4}}\|f_0\|_\infty^2\right\}
\end{split}
\end{equation}
\end{thm}

\begin{proof}
Note that $Z(t)$ is decreasing in time, and in every regime where the form of $\eta$ in \eqref{etaeq} stays fixed,
once the second term becomes the smaller one, this continues for all subsequent times.
%At what value of $Z(t)$ this happens depends on the initial data; for some initial data the second term will be smaller from the beginning.
Let us first estimate the total time in the interval $t\geq 1$ where the first term under minimum in \eqref{diff_ineq4} is smaller, while
$Z(t)\geq 4Q_2$.
Comparing the two terms in the min function of \eqref{diff_ineq4},
we see that the minimum is achieved by the first term as long as $Z(t) \geq \|f_0\|_\infty^2$.
%and the second item otherwise.
Thus, $Z(t)$ decays exponentially not slower than $\exp(-c\gamma^2t)$.
%has an exponential decay with the rate $-c\gamma^2$.
Note that at $t=1$, we have
\[
Z(1) \leq Z(0) = \int_{\mathbb{R}^2} (f_0 - \bar f)^2 e^H dx \leq \int_{\mathbb{R}^2} f_0^2 e^H dx \leq \|f_0\|_\infty^2 \int_{\mathbb{R}^2} e^H dx,
\]
hence the total time $t\geq 1$ when $Z(t)\geq 4Q_2$ and the first term in \eqref{diff_ineq4} is the smaller one is bounded by
\[
t_{11} := \frac{1}{c\gamma^2}\log\left(\frac{Z(1)}{\|f_0\|_\infty^2}\right)
\leq \frac{1}{c\gamma^2} \log \left( \int_{\mathbb{R}^2} e^H dx \right)
\leq \frac{C}{\gamma}.
\]
Hence, in the $Z(t)\geq 4Q_2$ regime, the presence of the first
term at most introduces a time delay of the order $\gamma^{-1} \leq 1.$

Likewise, when the first term \eqref{diff_ineq4} is smaller and $Z(t) \in [4Q_1, 4Q_2]$,
$Z(t)$ has an exponential decay not slower than $\exp(-c\gamma t)$.
Hence, in this case the time with the first term active is bounded by
\[
t_{12} := \frac{1}{c\gamma}\log\left(\frac{Q_2}{Q_1}\right)\leq \frac{1}{c\gamma} (H(0)-H(r_2)) \leq C.
\]
Thus  the presence of the first term also at most introduces a time delay of order one in this regime.

Finally, in the $Z(t) \in [Z^\sigma, 4Q_1]$ regime and when the first term in \eqref{diff_ineq4} is smaller,
$Z(t)$ has an exponential decay not slower than $\exp(-ct)$. So the time with the first term active is bounded by
\[
t_{13} := \frac{1}{c}\log\left(\frac{Q_1}{Z^\sigma}\right)\leq
\frac{1}{c} \log\left(\frac{\gamma^2}{\sigma}\right) \leq C(1+\log \sigma^{-1} + \log\gamma).
\]
%thus  the presence of the first term in \eqref{diff_ineq4} at most introduces a time delay
%of the order
%\[
%1+\log \sigma^{-1} + \log\gamma
%\]
%in this regime.
Combining these estimate together, we see
that the total time delay caused by the first term in the mininmum function is bounded by
\[
t_1 := t_{11} + t_{12} + t_{13} = C(1 + \log \sigma^{-1} + \log \gamma).
\]
\end{proof}

%\begin{comment}
\begin{rem}
The appearance of $\log \gamma$ in the definition of $t_1$ is likely not optimal.
In fact, as far as pure transport of the density goes (without estimate on the rate of convergence to equilibrium),
in Section \ref{trancompsec} we outline a different method that yields a bound on transport without extra delay terms.
In the context of convergence to equilibrium estimates, this extra correction
 comes from the $\gamma^2$ factor in $Z_1$ and $Z_2$, which is due to
the $\gamma$ factor in our $L^\infty$ estimate of $\rho$ in Theorem~\ref{thm:L_infty_bound}
%Section~\ref{sec:Linfty}:
\[
\|\rho(\cdot,t)\|_\infty \leq C\gamma \|\rho_0\|_1\hbox{ for all $t\geq 1$.}
\]
Such a bound would be optimal if we had $H=(-\Delta)^{-1}\chi_{B(0,1)}$, and
our argument can also be adapted to this case. But for the weight $e^H$ in (\ref{eH}),
the top is flat and $\|e^H\|_1 \sim \|e^H\|_\infty$, which suggests that there
should not be a $\gamma$ factor, and we should have
\[
\|f(\cdot,t) e^H\|_{L^\infty} \leq C \|f_0 e^H\|_1\hbox{ for $t\geq 1$.}
\]
We can not show this and settle here for the $\log \gamma$ correction that in most situations
is not very significant.

%and to reduce technicalities improvements will be addressed elsewhere.
%It would be nice to prove it -- or show that $\gamma$ is indeed needed even in this case.
\end{rem}
%\end{comment}

We now translate the above weighted $L^2$ bounds  to $\rho$.
Let $\rho(x,t)$ be a solution to~\eqref{eq:fp_const} with initial
condition $\rho_0 \in L^\infty(e^{-H}) \cap L^1(\mathbb{R}^2) $. Recall
that
\[
\rho_s := e^H \frac{\int\rho_0 dx}{\int e^H dx},
\]
is a stationary solution to \eqref{eq:fp_const} with the same mass as $\rho$. Also recall
that $Z(t)$ can be written as in
%the definition
\eqref{def:Y}:
%\begin{equation}\label{def:Y}
\[
%Y(t)
Z(t)= \int_{\mathbb{R}^2} (\rho(x,t) - \rho_s(x))^2 e^{-H(x)} dx,
\]
%\end{equation}
and that $f(x,t) := \rho(x,t) e^{-H(x)}$ satisfies \eqref{eq:f} with  initial
condition
\[
f_0 = \rho_0 e^{-H} \in L^\infty(\mathbb{R}^2) \cap L^1(e^{H}).
\]
%Recall that the functional $Y(t)$ equals $Z(t)$ given the above link between $\rho$ and $f$. %above for $\rho(\cdot,t)$ is the same as $Z(t)$ for $f(\cdot,t)$.
Applying Theorem \ref{thm1} to $f = \rho e^{-H}$, we get an analog of \eqref{eq9}:
\begin{thm}\label{prop:Y}
Let $\rho(x,t)$ be the solution to \eqref{eq:fp_const} with initial
condition $\rho_0 \in L^\infty(e^{-H}) \cap L^1(\mathbb{R}^2)$, $t_1$ be as in
Theorem \ref{thm1}, %and $Y(t)$ as in \eqref{def:Y}.
and $Z^\sigma := \sigma e^{-H(0)}\|\rho_0\|_1^2$. Then for all $t\geq t_1$, we have
\begin{equation}\label{eq10}%Y(t) \leq \max\left( \tilde Z, \frac{Y(0)}{\left(1+C\gamma^{1+\frac{4}{\gamma-4}}(t-t_1)\|\rho e^{-H}\|_\infty^{-\frac{8}{\gamma-4}} Y(0)^{\frac{4}{\gamma-4}} \right)^{\frac{\gamma-4}{4}}}\right).
%Y(t)
Z(t)\leq \max\left\{ Z^\sigma, \left(c\gamma (t-t_1)\right)^{-\frac{\gamma-8}{8}}\|\rho_0 e^{-H}\|_\infty^2 \right\}.
\end{equation}
In particular, this implies that $%Y(t)
Z(t)\leq Z^\sigma$ for all
\[
t\geq t_2:= t_1 + \frac{C}{\gamma} \left(\frac{\|\rho_0 e^{-H}\|_\infty}{\sqrt{\sigma}\|\rho_0\|_1} \right)^{\frac{16}{\gamma-8}}.
\]
\end{thm}

On the other hand, once $Z(t)$ %$Y(t)$
drops below $Z^\sigma$, the following proposition shows that $\rho(\cdot,t)$ is sufficiently close to $\rho_s$ in $B_{r_1}$.

\begin{prop}\label{prop2}Let $Z^\sigma =\sigma e^{-H(0)} \|\rho_0\|_1^2$, with $\sigma < 1$,
and let $r \leq r_1 = 1/\sqrt{2}.$
If $Z(t) %Y(t)
\leq AZ^\sigma$, then
\begin{equation}\label{conveq521}
\int_{B_{r}} |\rho(x,t) - \rho_s(x)|dx \leq \sqrt{\pi \sigma A} r \|\rho_0\|_1. % \equiv \sqrt{\pi \sigma A} r M_0.
\end{equation}
Moreover, if we assume in addition that $\rho_0 \geq 0,$ then
\begin{equation}\label{massconc521}
\int_{B_r} \rho(x,t)\,dx \geq \left( 2r^2 - \frac{C}{\sqrt{\gamma}} - r \sqrt{\pi \sigma A} \right)\|\rho_0\|_1.
\end{equation}
\end{prop}

\begin{proof}
If $Z(t) %Y(t)
\leq AZ^\sigma$, the definitions of $Z(t)$ %$Y(t)$
and $Z^\sigma$ give
\[
\int_{B_{r}} |\rho - \rho_s|^2 e^{-H} dx \leq \int_{\mathbb{R}^2} \left| \rho -\rho_s \right|^2 e^{-H} dx \leq \sigma A e^{-H(0)} \|\rho_0\|_1^2.
\]
Using the fact that $e^{-H}\equiv e^{-H(0)}$ in $B_{r}$, the above inequality becomes
\[
\int_{B_{r}} |\rho - \rho_s|^2 dx \leq \sigma A\|\rho_0\|_1^2.
\]
Then a direct application of the Cauchy-Schwarz inequality gives \eqref{conveq521}.

A direct computation using \eqref{eH} shows that
\[ \frac{\int_{(B_{r_1})^c} e^{H(x)}\,dx}{\int_{B_{r_1}} e^{H(x)}\,dx} \leq \frac{C}{\sqrt{\gamma}}. \]
Then, if $\rho_0 \geq 0,$ we have, since $\rho_s$ is constant on $B_{r_1}$ and $r\le r_1$:
\[
\int_{B_r} \rho_s(x)\,dx \geq \frac{r^2}{r_1^2}\|\rho_0\|_1
\left( 1- \frac{C}{\sqrt{\gamma}} \right).
\]
Combining this inequality with \eqref{conveq521}, we obtain \eqref{massconc521}.
\end{proof}

The inequality \eqref{massconc521} gives us a way to ensure that much
of the mass of $\rho_1$ has been transported into the support of $\rho_2,$
provided we choose $\sigma$ sufficiently small and $\gamma$ is sufficiently large. However,
as we mentioned above, the weighted $L^2$ decay estimates we have for $Z(t)$ %$Y(t)$
lead to bounds on the transport time that are far from the heuristic ones.
We now discuss this issue in more detail and use duality to rectify the situation.

\subsection{Duality and the $L^1$ control}

%So we now have a good understanding of the $L^2$ norm decay.
Theorem \ref{prop:Y} and Proposition \ref{prop2} give us an explicit upper bound for
the time it takes for a large portion of mass to enter $B_{r_1}$, but this is not sufficient
for our application. Let us recap the reason: consider a special case where $\rho_0(x)$
is a bump of mass $M_0$, located at distance $L$ from the origin. In this case,
we have
\[
\|\rho_0 e^{-H}\|_{L^\infty} \sim %M_0L^{\gamma/2}
{M_0L^{\gamma/4}}.
\]
%and
%\[
%\|\rho_0 - \rho_s\|_{L^2(e^{-H})} \sim M_0L^{\gamma/8}.
%\]
Then \eqref{eq10} requires the time
\[
t \sim 1 + \log\gamma + \frac{L^4}{\gamma},
\]
to assure transport of a significant portion of $\rho$ to $B_{r_1},$ which is at odds with
the heuristic bound of the order $L^2/\gamma$.
To get control at a time scale close to heuristic, we employ a duality procedure which is somewhat delicate in our case since we may
have different regimes in differential inequalities. A direct computation leads to the
the following auxiliary duality lemma.
%The problem is that in our $L^\infty(e^{-H})\to L^2(e^{-H})$ decay, we have different regimes. This makes the duality argument complicated.
\begin{lem}\label{lemma:duality}
Let $\rho$ and $f$ be solutions to \eqref{eq:fp_const} and \eqref{eq:f} respectively,
with initial conditions~$\rho_0$ and $f_0$, where $\rho_0 \in L^\infty(e^{-H}) \cap L^1(\mathbb{R}^2)$, and $f_0 \in L^\infty(\mathbb{R}^2) \cap L^1(e^H)$, and set
\[
\rho_s(x) := \frac{e^H(x) \int \rho_0 dx}{\int e^H dx},~~
\bar f := \frac{\int f_0 e^H dx}{\int e^H dx}.
\]
Then, for any $t> 0$ and $s\in[0,t]$, the integral
\begin{equation}\label{dual526}
\int_{\mathbb{R}^2} (\rho(x,s)-\rho_s(x)) (f(x,t-s)- \bar f) dx
\end{equation}
does not depend on $s$ for all $s\in[0,t]$.
\end{lem}
{Note that the term $\bar f$ in the right side can always be dropped
since
\[
\int \rho(x,s)dx=\int\rho_s(x)dx,
\]
for all $s\ge 0$.}
\begin{proof}
By standard approximation arguments, it suffices to show the result for smooth, sufficiently quickly decaying $\rho,$ $f.$
Denote the integral in \eqref{dual526} by $U(s)$. Taking the derivative in $s$ gives
\[
\begin{split}
&\frac{d}{ds}U(s)= \int_{\mathbb{R}^2} \partial_t \rho (x,s) (f(x,t-s)-\bar f) dx -  \int_{\mathbb{R}^2} (\rho(x,s)-\rho_s(x)) \partial_t f(x,t-s) dx =: T_1 - T_2,
%&= \int_{\mathbb{R}^2}  (\Delta \rho(x,s) - \nabla\cdot (\rho(x,s) \nabla H)) F(x,t-s) dx - \int_{\mathbb{R}^2} \rho(x,s) (\Delta f(x,t-s) + %\nabla f(x,t-s) \cdot \nabla H)dx\\
%&= 0,
\end{split}
\]
where
\[
T_1 =  \int_{\mathbb{R}^2}  (\Delta \rho(x,s) - \nabla\cdot (\rho(x,s) \nabla H)) (f(x,t-s)-\bar f) dx,
\]
and
\[
T_2 = \int_{\mathbb{R}^2} (\rho(x,s)-\rho_s(x)) (\Delta f(x,t-s) + \nabla f(x,t-s) \cdot \nabla H)dx.
\]
Now one can check that $T_1=T_2$ by the divergence theorem (using, in particular, that $\rho_s$ and $\bar f$
are eigenfunctions of $\Delta-\nabla (\cdot \nabla H)$ and $\Delta + \nabla H \nabla,$ respectively, with zero eigenvalue.
\end{proof}
%
%For the rest of this subsection,
We can now prove the following theorem.
\begin{thm}\label{thm:l1}
Fix any $0<r \leq r_1.$  For all $\sigma \in (0,1)$, let $t_1$ be
as in Theorem \ref{thm1}. % and~\ref{prop:Y}, respectively.
Define
\begin{equation}\label{t3}
t_3:=C\left( t_1+ \frac{1}{\gamma}\left(\frac{\|\rho_0 e^{-H}\|_\infty}{\sigma \|\rho_0\|_1} \right)^{\frac{8}{\gamma-8}} \right),
\end{equation}
with some sufficiently large constant $C$ that will be fixed in the proof.
Then, for all $t\geq t_3$, we have
\begin{equation}\label{eq:L1bound}
\int_{B_{r}} |\rho(x,t) - \rho_s(x)| dx \leq (4\sqrt{\sigma}r+4\sigma) \|\rho_0\|_1.
\end{equation}
In particular, if $\sigma$ is chosen to be sufficiently small, $\gamma$ is sufficiently large, and $\rho_0 \geq 0,$ then we have
\begin{equation}\label{massin521}
\int_{B_{r}} \rho(x,t) dx \geq (2 r^2- 0.1) \|\rho_0\|_1,~~
\hbox{ for all $t\geq t_3$.}
\end{equation}
\end{thm}
\begin{proof}
%Let us apply~(\ref{eq10}) at
%Advance time by
%$t=t_2/3$. % and apply \eqref{eq10}.
Consider first what happens at the time $t_3/3$.
If $Z(t_3/3)$ %$Y(t_2)$
drops below $ Z^\sigma = \sigma e^{-H(0)} \|\rho_0\|_1^2$, we are done due to
Proposition \ref{prop2}. Otherwise, we have the second bound in \eqref{eq10} for $Z(t_3/3)$
with~$t=t_3/3$.
To obtain a better $L^1$ control of $\rho(\cdot,t_3)-\rho_s$ in the latter case,
we use the following duality argument.
For any $f_0 \in L^\infty(\mathbb{R}^2) \cap L^1(e^H)$, let $f(x,t)$ be the solution to
the dual equation~\eqref{eq:f} with initial condition $f_0$.  Applying
Lemma \ref{lemma:duality} with $t=2t_3/3$,  $s={2t_3}/{3}$
and then $s={t_3}/{3}$, we obtain
\begin{equation}\label{eq:rho_f}
 \int_{\mathbb{R}^2} \left(\rho\left(x,\frac{2t_3}{3}\right)-\rho_s(x)\right) f_0(x)dx =\int_{\mathbb{R}^2} \left(\rho\left(x,\frac{t_3}{3}\right)- \rho_s\right) \left(f\left(x,\frac{t_3}{3}\right)-\bar f\right) dx.
\end{equation}
We dropped the term involving $\bar f$ in the left side using the remark after
Lemma~\ref{lemma:duality}.
% which is allowed since $\int (\rho(x,t) - \rho_s) dx = 0$ for all $t$.
We can then bound the left side in (\ref{eq:rho_f})  as
\begin{equation}\label{ineq_rho_temp}
\begin{split}
&\left| \int_{\mathbb{R}^2} \left(\rho\left(x,\frac{2t_3}{3}\right)-\rho_s(x)\right) f_0(x)dx\right| \leq \left\|\rho\left(\frac{t_3}{3}\right) - \rho_s\right\|_{L^2(e^{-H})} \left\|f\left(\frac{t_3}{3}\right) - \bar f\right\|_{L^2(e^H)} \\
&\leq  \left\|\rho\left(\frac{t_3}{3}\right) - \rho_s\right\|_{L^2(e^{-H})} \max\left\{ (\sigma e^{-H(0)}\|f_0 e^H\|_1^2)^{1/2}, \left(c\gamma \left(\frac{t_3}{3}-t_1\right)\right)^{-\frac{\gamma-8}{16}}\|f_0\|_\infty \right\} \quad\text{(by Theorem \ref{thm1})}\\
&\leq  \left(c\gamma \left(\frac{t_3}{3}-t_1\right)\right)^{-\frac{\gamma-8}{16}}\|\rho_0 e^{-H}\|_\infty  \left(\sigma^{\frac12} e^{-\frac{H(0)}{2}} \|f_0 e^H\|_1 +  \left(c\gamma \left(\frac{t_3}{3}-t_1\right)\right)^{-\frac{\gamma-8}{16}}\|f_0\|_\infty \right) \quad\text{(by Theorem \ref{prop:Y})}\\
&\leq \alpha \|f_0 e^H\|_1 + \beta \|f_0\|_\infty,
\end{split}
\end{equation}
where
\[
\alpha := \left(c\gamma \left(\frac{t_3}{3}-t_1\right)\right)^{-\frac{\gamma-8}{16}}  \sigma^{\frac{1}{2} }e^{-\frac{H(0)}{2}}\|\rho_0 e^{-H}\|_\infty,
~~~\beta := \left(c\gamma \left(\frac{t_3}{3}-t_1\right)\right)^{-\frac{\gamma-8}{8}}\|\rho_0 e^{-H}\|_\infty.
\]

Now let us apply the following lemma, the proof of which is postponed till the end of this subsection.
\begin{lem}\label{lemma:duality_2}
Suppose that for some function $G \in L^\infty(e^{-H}) \cap L^1(\mathbb{R}^2)$, there
exist~$\alpha, \beta>0$ such that
\begin{equation}\label{eq:FG}
\left |\int_{\mathbb{R}^2} G(x)f(x) dx\right | \leq \alpha \|fe^H\|_1 + \beta \|f\|_\infty \quad\text{ for all } f \in L^\infty(\mathbb{R}^2) \cap L^1(e^H).
\end{equation}
Then $G$ can be decomposed as $G = G_1 + G_2$, where $G_1, G_2 \in L^\infty(e^{-H}) \cap L^1(\mathbb{R}^2)$ satisfy the estimates $\|G_1 e^{-H}\|_\infty \leq 2\alpha$, $\|G_2\|_1 \leq 2\beta$.
\end{lem}

Applying this lemma to \eqref{ineq_rho_temp}, we can decompose
\[
\rho\left(x, \frac{2t_3}{3}\right) - \rho_s(x) = G_1(x) + G_2(x),
\]
where
\begin{equation}\label{estimate_G1}
\|G_1 e^{-H}\|_\infty \leq 2\alpha = 2 \left(c\gamma \left(\frac{t_3}{3}-t_1\right)\right)^{-\frac{\gamma-8}{16}}  \sigma^{\frac{1}{2} }e^{\frac{-H(0)}{2}}\|\rho_0 e^{-H}\|_\infty ,
\end{equation}
and
\[
\|G_2\|_1 \leq 2\beta = 2   \left(c\gamma \left(\frac{t_3}{3}-t_1\right)\right)^{-\frac{\gamma-8}{8}}\|\rho_0 e^{-H}\|_\infty \leq 2\sigma \|\rho_0\|_1,
\]
where the last inequality comes from choosing a sufficiently large universal constant $C$ in the definition \eqref{t3} of $t_3$.

Let $\zeta_1(x,t)$ and $\zeta_2(x,t)$ denote the solutions to \eqref{eq:fp_const}
starting at $t=2t_3/3$ with initial conditions
%data with %(where $\frac{2t_3}{3}$ is the initial time)
$\zeta_1(\cdot,{2t_3}/{3}) = G_1$, $\zeta_2(\cdot, {2t_3}/{3}) = G_2$, respectively.
Since \eqref{eq:fp_const} is linear, we have
\[
\rho(\cdot,t_3) - \rho_s = \zeta_1(\cdot,t_3) + \zeta_2(\cdot,t_3).
\]
Note that $\|\zeta_2(\cdot,t)\|_1$ is non-increasing in time, hence
% (it may decrease since positive and negative parts may cancel), hence we have that
\begin{equation}\label{eq:rho2_estimate}
\|\zeta_2(\cdot, t_3)\|_1 \leq \|G_2\|_1 \leq 2\sigma \|\rho_0\|_1.
\end{equation}

%Finally, it remains
To control $\zeta_1(\cdot,t_3)$, set
\[
\zeta_1^s := e^H \frac{\int G_1 dx}{\int e^H dx}.
\]
%Note that even though $\zeta_1$ may change sign, its integral
%remain unchanged during evolution.
By Theorem~\ref{prop:Y}, we have
\begin{equation}\label{wl2521}
\|\zeta_1(\cdot,t_3) - \zeta_1^s\|_{L^2(e^{-H})} \leq  \max\left\{ \sigma^{1/2} e^{-\frac{H(0)}{2}} \|G_1\|_1, \left(c\gamma \left(\frac{t_3}{3}-t_1\right)\right)^{-\frac{\gamma-8}{16}}\|G_1 e^{-H}\|_\infty \right\}
\end{equation}
If the first term in the max function is larger, using the fact that
\[
\|G_1\|_1 \leq \|\rho\|_1 + \|\rho_s\|_1 + \|G_2\|_1\leq 3\|\rho_0\|_1,
\]
we obtain
\[
\|\zeta_1(\cdot,t_3) - \zeta_1^s\|_{L^2(e^{-H})} \leq 3\sigma^{1/2} e^{-\frac{H(0)}{2}} \|\rho_0\|_1 \leq 3\sqrt{Z^\sigma}.
\]
And if the second term is larger, combining \eqref{wl2521} with \eqref{estimate_G1}, we get
\[
\|\zeta_1(\cdot,t_3) - \zeta_1^s\|_{L^2(e^{-H})} \leq 2   \left(c\gamma \left(\frac{t_3}{3}-t_1\right)\right)^{-\frac{\gamma-8}{8}}\sigma^{\frac{1}{2} }e^{-\frac{H(0)}{2}}\|\rho_0 e^{-H}\|_\infty \leq 2\sigma^{3/2}e^{-\frac{H(0)}{2}} \|\rho_0\|_1 = 2\sigma\sqrt{Z^\sigma}.
\]
In both cases, applying Proposition \ref{prop2} yields
\[
\int_{B_{r}} |\zeta_1(x,t_3) - \zeta_1^s(x)|dx \leq 4 \sqrt{\sigma} r \|\rho_0\|_1.
\]
Finally, combining the above estimate with \eqref{eq:rho2_estimate}, we have
\[
\begin{split}
\|\rho(t_3) - \rho_s\|_{L^1(B_{r})} &\leq \|\zeta_1(t_3) - \zeta_1^s\|_{L^1(B_{r})} +  \|\zeta_2(t_3) + \zeta_1^s\|_{L^1(B_{r})}\\
& \leq \|\zeta_1(t_3) - \zeta_1^s\|_{L^1(B_{r})} +  \|\zeta_2(t_3)\|_1 + \|G_2\|_1\\
&\leq (4\sqrt{\sigma}r+4\sigma)\|\rho_0\|_1,
\end{split}
\]
where in the second inequality we used the mean zero property
\[
\int (G_1+G_2)\,dx=0,
\]
which gives
\[
\|\zeta_1^s\|_1 = \Big|\int G_1 dx\Big| = \Big|\int G_2 dx\Big|.
\]
The above argument shows that \eqref{eq:L1bound} holds at $t=t_3$. For $t>t_3$, the same argument works by replacing $t_3$ with $t$.

The estimate \eqref{massin521} follows from a simple computation similar to that
in the proof of \eqref{massconc521}.
\end{proof}

\begin{proof}[Proof of Lemma \ref{lemma:duality_2}]
Let $S = \{x : |G(x)| \geq 2\alpha e^H\}$, and define $G_1 := G \chi_{S^c}(x)$, so
that
\[
\|G_1 e^{-H} \|_\infty \leq 2\alpha.
\]
To show that $G_2 := G - G_1 = G\chi_S(x)$, satisfies $\|G_2\|_1 \leq 2\beta$,
we use (\ref{eq:FG}) with $f= (\text{sgn}\,G) \chi_S$:
%Since $f \in  L^\infty(\mathbb{R}^2) \cap L^1(e^H)$,
%Applying \eqref{eq:FG} to such $f$ gives
\[
\begin{split}
\|G_2\|_{L^1(\mathbb{R}^2)} = \int_S G_2 f dx \leq \alpha \int_S e^H dx + \beta \leq \frac{1}{2}\|G_2\|_{1} + \beta,
\end{split}
\]
%In the last inequality we used the definition of $G_2$ and  the set $S$.
%We can then absorb the $\frac{1}{2}\|G_2\|_1$ summand into the left hand side,
and the proof is complete. %finished.
\end{proof}

\section{Transport estimates based on comparison principles}\label{trancompsec}

In this section we take a quick detour to provide a simple alternative
proof that a significant portion of the initial mass of $\rho_0$ gets transported
inside a certain ball of radius less than $1$ under the action of the potential $H$ in time $\tau \sim L^2/\gamma.$
As we mentioned in the introduction,
this result can be used to obtain a simpler proof of a result similar to Theorem~\ref{main1} if one is willing to compromise and settle
for an estimate that provides little information on the closeness to ground state.

%on the time at which some fixed portion less than $1/2$ of the initial mass of $\rho_2$ gets consumed.
%The advantage of this proof is its simplicity, but it provides less control over the mass of $\rho_2$ and no information on
%closeness to the ground state for intermediate time scales.

The main step is the analysis of the dual equation \eqref{eq:f}. Recall that the dual operator $L^*$ is given by
\[
L^*f = -\Delta f - \nabla H \cdot \nabla f,
\]
and the dual evolution by
\begin{equation} \label{dual216}
\partial_t f = \Delta f + \nabla H \cdot \nabla f = -L^*f.
\end{equation}
We will prove the following theorem.
\begin{thm}\label{trantheor}
Let $f(x,t)$ solve \eqref{dual216} with $H$ given by \eqref{H:explicit}.
Suppose that the radial initial data $f_0 \in C_0^\infty$ satisfies $1 \geq f \geq 0,$ $f$ non-increasing in the radial direction,
and $f_0(x) \geq \chi_{B_{d_1}}(x)$
where $1 \geq d_1 > r_1 = 1/\sqrt{2}.$ Then there exist a constant $c>0$ such that for all sufficiently large $\gamma$ we have
\begin{equation}\label{flb216}
f(x,t) \geq c \chi_{B_{c \sqrt{1+\gamma t}}}(x)
\end{equation}
for all $t \geq 0.$
\end{thm}
\begin{proof}
Fix some $d_0$ such that $1 \geq d_1 > d_0 > r_1.$ For simplicity, in the argument that follows, we can think for instance of
$d_0=5/7$ and $d_1=6/7,$
but any other choice satisfying the above relationship works as well (the constant $c$ will depend on this choice).
Due to parabolic comparison principles and since $H$ is radial, we have that the solution $f(x,t)$ remains radial, non-increasing in the
radial direction, and satisfies $1 \geq f(x,t) \geq 0$ for all times.

Observe that %since $\int f(x,t) e^{H(x)}\,dx$ is conserved and $f_0(x) \equiv 1$ in $B_{d_0},$ we must have
\begin{eqnarray}
\left. f(x,t) \right|_{\S_{d_0}} \int_{\R^2 \setminus B_{d_0}} e^H\,dx \geq \int_{\R^2 \setminus B_{d_0}} f e^H \,dx
\geq \int_{\R^2} f_0 e^H \,dx - \int_{B_{d_0}} fe^H \,dx \geq \nonumber \\ \label{aux216a}
\int_{\R^2 \setminus B_{d_0}} f_0 e^H\,dx \geq \int_{\R^2 \setminus B_{d_0}} \chi_{B_{d_1}} e^H = \int_{B_{d_1} \setminus B_{d_0}} e^H \,dx.
\end{eqnarray}
Here $\S_{d_0}$ is the circle of radius $d_0;$
in the first step we used monotonicity of $f$ in radial variable, in the second step conservation of
$\int f(x,t) e^{H(x)}\,dx$ and in the third step \[ \int_{B_{d_0}} f_0 e^H \,dx  \geq \int_{B_{d_0}} f e^H\,dx \]
due to $1 = f_0(x) \geq f(x,t)$ in $B_{d_0}.$
However, since $|\nabla H(x)| = -\partial_r H \geq c_0 \gamma |x|^{-1}$ if $|x| \geq d_0$,  we have
\[  e^{H(tx)} \leq t^{-c_0 \gamma}e^{H(x)}.      \]
Set $q = d_1/d_0.$ Then
\[ \int_{B_{d_1 q^k} \setminus B_{d_0 q^k}} e^{H(x)} \,dx = q^{kd}\int_{B_{d_1} \setminus B_{d_0}}e^{H(q^k x)} \,dx \leq
q^{k(d-c_0\gamma)} \int_{B_{d_1} \setminus B_{d_0}} e^{H(x)} \,dx. \]
 Therefore,
 \[ \int_{\R^2 \setminus B_{d_0}} e^{H(x)}\,dx \leq 2  \int_{B_{d_1} \setminus B_{d_0}} e^{H(x)} \,dx \]
 for $\gamma$ large enough. In this case from \eqref{aux216a} we conclude that $\left. f(x,t) \right|_{\S_{d_0}} \geq 1/2$ for all times.

Now fix any convex $C^2$ function $\omega$ on $[d_0,\infty)$ such that $\omega(d_0)=1/2$, $\omega(r)>0$ for $r \in [d_0,d_1),$
and $\omega(r)=0$ if $r \geq d_1.$
For $\varphi \in [0,1]$ define $\omega_\varphi(r) = \omega (d_0 + \varphi(r-d_0));$ we will abuse notation by also writing $\omega_\varphi(x)= \omega_\varphi(|x|).$
Note that
\[ L^{*} \omega_\varphi(x) = \omega''_\varphi (r) + \frac1r  \omega'_\varphi(r) + \partial_r H(r) \omega'_\varphi(r)
\geq \frac{c_0\gamma -1}{r}| \omega'_\varphi(r)| \geq \frac{c_0 \gamma}{2r} | \omega'(d_0+\varphi(r-d_0))|\varphi, \]
where we used $\omega''_\varphi(r) \geq 0,$ $\omega'(r)<0,$ and the last step holds if $\gamma$ is sufficiently large.
Choose a decreasing $\varphi(t)$ defined for $t \geq 0$ such that $\varphi(0)=1.$ Consider $F(x,t) = \omega_{\varphi(t)}(x).$ Since we always have
$\left. f(x,t) \right|_{\S_{d_0}} \geq 1/2 = \left. F \right|_{\S_{d_0}}$ and $f_0(x) \geq \chi_{B_{d_1}(x)} \geq \omega(|x|),$
we can be sure that $f(x,t) \geq F(x,t)$ in $\R^2 \setminus B_{d_0}$
for all times if $\partial_t F \leq L^* F.$
However $\partial_t F= (r-d_0) \omega'(d_0+\varphi(r-d_0))\varphi'(t)$ and $\partial_t F = L^*f=0$ if $d_0 +\varphi(r-d_0) \geq 1.$
Hence we just need to check the inequality
\[ -(r-d_0)\varphi'(t) \leq \frac{c_0 \gamma}{2r} \varphi \]
when $r  \leq d_0 + \frac{1-d_0}{\varphi} \leq \frac{1}{\varphi}.$ Thus, it suffices to ensure that
\[ -\frac{1-d_0}{\varphi} \varphi'(t) \leq \frac{c_0\gamma}{2} \varphi^2 \] which would follow from
$\partial_t (1/\varphi^2(t))  \leq c_0 \gamma.$ Therefore
\[ \varphi(t) = \frac{1}{\sqrt{1+c_0\gamma t}} \]
is acceptable.  Now fix a constant $a<d_1-d_0$, then we can make
\[ d_0 + \frac{1}{\sqrt{1+c_0\gamma t}}(r-d_0) \leq d_0+a \]
for $r \leq d_0 +c\sqrt{1+\gamma t}$ by choosing small enough $c.$ In this case, if $d_0 \leq r \leq d_0+c\sqrt{1+\gamma t},$ we have
\[ f(x,t) \geq \omega\left(d_0 + \frac{1}{\sqrt{1+c_0\gamma t}}(r-d_0)\right) \geq \omega(d_0+a) \geq c>0, \]
where we may have to adjust our constant $c$ to make it smaller if necessary.
\end{proof}

Here is the corollary for the behavior of the density $\rho(x,t)$ satisfying \eqref{fokker11}.
\begin{cor}\label{cortran}
Let $\rho(x,t)$ solve \eqref{fokker11} with a potential $H$ given by \eqref{H:explicit}.
Suppose that the initial data $\rho_0$ satisfies $\rho_0(x) \geq 0$ and $\int_{1\leq |x| \leq L} \rho_0(x)\,dx = M_0.$
%Fix $d_2 > 1/\sqrt{2}.$
Then for all sufficiently large $\gamma,$ there exists a constant $C_1$ such that if $t \geq C_1 L^2/\gamma$, we have
\begin{equation}\label{aux217a}
\int_{B_{6/7}} \rho(x,t) \, dx \geq c M_0.
\end{equation}
\end{cor}
\it Remark. \rm For simplicity, we picked a fixed constant as a radius of the ball in \eqref{aux217a}. It is not hard to run
the argument for an arbitrary radius greater than $1/\sqrt{2},$ but then all constants and the range of validity in $\gamma$
will depend on the choice of radius.
\begin{proof}
Choose $d_0=5/7$ and $d_1$ so that $d_0 <d_1 < 6/7.$
Take $f_0 \in C_0^\infty(B_{6/7})$ as in Theorem~\ref{trantheor}.
Due to duality, we have
\[ \int_{\R^2} f_0(x) \rho(x,t)\,dx = \int_{\R^2} f(x,t) \rho_0(x)\,dx. \]
Therefore, applying Theorem~\ref{trantheor} we find that if $C_1$ is sufficiently large then
\[  \int_{B_{6/7}} \rho(x,t) \, dx \geq  \int_{\R^2} f_0(x) \rho(x,t)\,dx  = \int_{\R^2} f(x,t) \rho_0(x)\,dx
\geq c M_0. \]
\end{proof}

Corollary~\ref{cortran} and \eqref{aux217a} can take place of Theorem~\ref{thm:l1} and \eqref{massin521} in the
nonlinear argument of the next section. We state here the theorem alternative to Theorem~\ref{main1} that this
would yield.

\begin{thm}\label{mainalt1}
Under assumptions of Theorem~\ref{main1}, with chemotaxis present, a quarter of the initial mass of $\rho_2$ will react by time
$\tau_C \leq C_1 L^2/\gamma.$
\end{thm}

\it Remark. \rm It is not difficult to design an additional argument that will show, under assumptions of Theorem~\ref{main1},
that larger than half of the initial mass of $\rho_2$ will react if we wait an additional time $\sim 1$. Basically, once mass $\sim M_0$
has entered $B_1,$ arguments similar to the ones we used above and employing mass comparison with the simple heat equation lead to the
conclusion that after an additional unit time, mass $\sim M_0$ can be found inside $B_{1/\sqrt{2}}$ (or in fact in a ball of smaller
radius, with a constant of proportionality depending on the radius). Then the pass-through argument of the following section
would yield consumption of the larger fraction of $\rho_2.$

\section{Decay for $\rho_2$ based on a ``pass-through'' argument}\label{system}

Let us now consider the nonlinear system (\ref{eq:original_system}):
\begin{equation}
\begin{split}
&\partial_t \rho_1 - \Delta \rho_1 +
\chi\nabla\cdot(\rho_1 \nabla(-\Delta)^{-1} \rho_2) = -\epsilon \rho_1 \rho_2\\
&\partial_t \rho_2 = -\epsilon \rho_1\rho_2.
\end{split}
\label{eq:actual}
\end{equation}
%Throughout this section,

We focus on the case when the initial conditions $\rho_{1}(\cdot,0)$ and $\rho_{2}(\cdot,0)$
are radially symmetric, so that radial symmetry is preserved for all time.
Assume that $\rho_1(\cdot,0)$ is initially concentrated near $r=L$ with the
total mass $M_0$, while $\rho_2(\cdot,0) = \gamma \eta(x)$, with
$\eta \in C_0^\infty$. We think of~$\eta$ as very close to $\chi_{B_1}(x)$ in the
$L^1$ norm. As in the introduction,
we assume that $\epsilon M_0 \gg \gamma \gg 1,$ and $M_0 \gg \theta.$ As we will see,
the constant $B$ involved in $\gg$ will depend on the value of the ratio $\chi \gamma/\epsilon$
and would have to be larger if it is small (but can be taken uniformly for all larger values
of this parameter ratio).
% with $M\gg \gamma>4$.
Combining Proposition \ref{compmassprop521} and Theorem \ref{thm:l1} together, we obtain that
if  $t_3 \leq \tau_C$, with $t_3$ given in Theorem \ref{thm:l1}, and $\tau_C$
the half-time of $\rho_2$, then at least than ${1}/{4}$ of the mass of $\rho_1$ must have entered $B_{1/2}$ by the time $t_3$.

In this section, we will use this result to obtain decay estimates on
the mass of $\rho_2$ which will show that, in fact, $\tau_C \leq t_3.$
Let us start with an heuristic argument to see how much of~$\rho_2$ should react by
the time $t_3$.
Since the drift velocity $\partial_r(-\Delta)^{-1} \rho_2 \sim -\gamma$
for all $r\in ({1}/{2}, 1)$, a generic particle of $\rho_1$
should take about $\sim\gamma^{-1}$ time to pass through the region $({1}/{2}, 1)$.
It will react with $\rho_2$ during this time with the coupling coefficient $\epsilon$, so that
approximately the~${\epsilon M_0}/{\gamma}$ portion of the mass of $\rho_2$ originally situated
in $B_1 \setminus B_{1/2}$ should be gone by the time $t_3$. In other words,
if ${\epsilon M_0}/{\gamma} \gg 1$, then we should have $\tau_C \leq t_3$.

We will discuss below why we have to resort to this ``pass-through" argument to
get an estimate on the reaction time. The reason has to do with the form
of the Keller-Segel chemotaxis term that leads to the possibility of
an excessive concentration of $\rho_1.$

%\textcolor{orange}{(I added the above paragraph of heuristic argument. This part of heuristics is not done in Section 2 yet; in that section we only talked about how long it %takes for $M_0/5$ mass to enter $B_{1/2}$. I think we should include this part of heuristics somewhere in the paper -- either here, or move it forward to Sec 2.)}

The goal of this section is to rigorously justify the above heuristics.
The key step is the following proposition.
\begin{prop}\label{prop:mass_eaten}
Let $\rho_1, \rho_2$ be a solution to \eqref{eq:original_system} with radially symmetric
initial conditions. Assume that $\rho_1(\cdot,0)$ is concentrated near $r=L$
with the total mass $M_0$, and $\rho_2(\cdot,0) = \theta \eta(x)$ as described above.
Assume that $\epsilon M_0 \gg \gamma \gg 1$.
Then the following holds with some universal constant $c>0$, where $t_3>0$ is as given by \eqref{t3}:
%\red{do you want $\epsilon M_0\gg \gamma$ and $1/2\to1/\sqrt{2}$?}
\begin{equation}\label{goal_prop51}
\int_0^{t_3+1} \rho_1(r,t) dt  \geq \frac{cM_0}{\gamma}, \text{ for all }r \in (1/2, 1).
\end{equation}
\end{prop}
Before we prove the proposition, let us point out that it implies Theorem~\ref{main1}.
\begin{proof}[Proof of Theorem~\ref{main1}]
%Since $\partial_t \rho_2 = -\epsilon \rho_1 \rho_2$, we have
The second equation in (\ref{eq:actual}) implies that
\[
\rho_2(r,t) = \rho_2(r,0) \exp\Big\{-\epsilon \int_0^t \rho_1(r,s)ds\Big\},
\]
so that  if (\ref{goal_prop51}) holds, then
\[
\frac{\rho_2(r,t_3+1)}{\rho_2(r,0)} = \exp\Big\{-\epsilon\int_0^{t_3+1} \rho_1(r,t) dt\Big\}
\leq e^{-{\epsilon c M_0}/{\gamma}},~~
\hbox{ for all $r \in (1/2, 1)$.}
\]
Thus, if ${\epsilon M_0}/{\gamma} \gg 1$, then most of the mass of $\rho_2$ originally supported in $B_1 \setminus B_{1/2}$ will react away by time $t_3+1$ and
the half-time $\tau_C$ satisfies $\tau_C \leq t_3+1$.
\end{proof}

%\begin{remark}
Recall that in the pure diffusion case, we have
\[
\tau_D \gtrsim \frac{L^2}{\log(M_0\epsilon)}.
\]
Comparing this with $t_3+1,$ and assuming that ${L^2}/{\gamma} \gtrsim \log \gamma$, we see
that chemotaxis would significantly reduce the half-time of reaction in the regime
\[
1 \ll \gamma \ll M_0\epsilon \ll e^\gamma.
\]
As we mentioned in the introduction, such relationship between parameters is natural in some applications.
%\end{remark}

The rest of this section contains the proof of Proposition \ref{prop:mass_eaten}.
As before, we set
\[
H(\cdot,t) := \chi (-\Delta)^{-1} \rho_2(\cdot,t).
\]
Since $H(\cdot,t)$ is radial, we denote it by $H(r,t)$.

Recall from (\ref{eq:M}) that %Section~\ref{sec:comparison_principle} that
\[
M(r,t) =  \int_{B_r} \rho_1(x,t) dx,
\]
satisfies
\begin{equation}\label{eq:Mt}
\partial_t M - \partial^2_{rr}M + \frac{1}{r} \partial_r M +  (\partial_r M) (\partial_r  H)
+ \epsilon \int_{B_r} \rho_1 \rho_2 dx = 0.
\end{equation}
Since $\rho_1(r,t) = (2\pi r)^{-1}{\partial_r M(r,t)}$, to prove \eqref{goal_prop51}, it suffices to show that
\begin{equation}\label{goal_prop51_new}
\int_0^{t_3+1} \partial_r M(r,t) dt  \geq \frac{cM_0}{\gamma}, \text{ for all }r \in (1/2, 1).
\end{equation}
%(Note that $r\in (\frac{1}{2},1)$ in \eqref{goal_prop51}, which is comparable to 1.)

Let $I = (a,b) \subset ({1}/{2}, 1)$ be an arbitrary interval.
For any $s \in (a,b)$, integrating \eqref{eq:Mt} over~$(a,s)$ in $r$ gives
\[
\begin{split}
\int_a^s \partial_t M(r,t) dr &= \partial_r M(s,t) - \partial_r M(a,t) - \int_a^s \left(\frac{1}{r} \partial_r M(r,t) + \partial_r M(r,t) \partial_r H(r)\right) dr  \\
& -\int_a^s \epsilon \int_{B_r} \rho_1\rho_2 dx dr
\leq \partial_r M(s,t) + C\gamma \int_a^s \partial_r M(r,t) dr,
\end{split}
\]
%where in the inequality we used that
since $\partial_r M \geq 0$, $\partial_r H_r \geq -C\gamma$, and $\rho_1,\rho_2 \geq 0$.
%\red{here you need $r>1/\sqrt{2}$?}.
As $M(r,0)=0$ for all $r<1$, since $\rho_1$ is initially concentrated near $r=L$,
integrating this inequality in time from $t=0$ to~$t=t_3$ gives
\[
\int_a^s %\left(
M(r,t_3) %- M(r,0) \right)
\,dr
\leq \int_0^{t_3} \partial_r M(s,t) dt + C\gamma \int_0^{t_3} \int_a^s \partial_r M(r,t) dr dt.
\]
%The left side can be estimated as follows.
%Note that For any $r\in (\frac{1}{2},1)$, we have $M(r,0)=0$ since $\rho_1$ is initially
%concentrated near $r=L$.
Combining Proposition \ref{compmassprop521} and Theorem \ref{thm:l1},
we have $M(r,t_3)\geq {M_0}/{4}$ for all $r \in ({1}/{2},1)$,
%Thus $M(r,t_3) - M(r,0) \geq \frac{M_0}{4}$ for all $r\in (a,s)\subset(1/2,1)$,
and the above inequality becomes
%Since the integrand in the left hand side is bounded below by $\frac{M_0}{5}$ for all $r\in (a,s)\subset(\frac{1}{2},1)$, the above inequality becomes
\[
\int_0^{t_3} \partial_r M(s,t) dt + C\gamma \int_0^{t_3} \int_a^s \partial_r M(r,t) dr dt \geq \frac{(s-a) M_0}{4}.
\]
%Note that if we don't have the first term on the left hand side, we would be done. (We can
%divide by $b-a$ and use the fact that they are arbitrary).
%We then
Integrating this inequality over $s\in I$ gives %, and obtain that
\[
\int_0^{t_3} \int_I \partial_r M(s,t) dsdt + C\gamma\int_0^{t_3} (b-a) \int_I \partial_r M(r,t) dr dt \geq \frac{(b-a)^2 M_0}{8},
\]
so that
\[
\int_0^{t_3} \frac{1}{|I|} \int_I \partial_r M(s,t) dsdt \geq \frac{ M_0}{8(|I|^{-1}+C\gamma) }.
\]
Therefore, for any interval $I \subset ({1}/{2},1)$ with $|I| = \gamma^{-1}$, we have that
\begin{equation}\label{eq:int_average}
\frac{1}{|I|} \int_I \int_0^{t_3}  \partial_r M(s,t) dt ds \geq \frac{cM_0}{\gamma}.
\end{equation}
This inequality shows that \eqref{goal_prop51_new} holds in each such
interval $I$ in an average sense. To fisnih the proof of Proposition \ref{prop:mass_eaten},
we need to rule out the possibility that
\[
\int_0^{t_3} \partial_r M(s,t) dt
\]
is distributed very non-uniformly among $s\in I$.
We are going to show that such situation cannot happen
since $\rho_1$ satisfies a parabolic PDE.

%The only danger that we need to eliminate now to complete the proof of proposition is that $\int_0^{t_3}  M_r(s,t) dt$ can be distributed very non-uniformly among $s\in I$, which would have led to a large part of $\eta$ remaining.
%Since $\rho$ satisfies a parabolic PDE, we will show that such situation will not happen.

Taking a derivative of \eqref{eq:Mt}, we deduce a
parabolic equation
\begin{equation}\label{eq:Yt}
\partial_t  u - \partial^2_{rr} u + \left(\frac{1}{r} + \partial_r H\right) \partial_r u + \left(-\frac{1}{r^2} + \partial^2_{rr} H +
\epsilon \rho_2(r,t)\right) u = 0
\end{equation}
for $u(r,t) := \partial_r M(r,t) = 2\pi r \rho_1(r,t) $.
%And \eqref{eq:int_average} can be stated in terms of $u$ as
%\begin{equation}\label{eq:int_average_u}
%\int_0^{t_3} \frac{1}{|I|} \int_I u(s,t) dsdt \geq \frac{cM_0}{\gamma}.
%\end{equation}
%The following proposition gives a pointwise lower bound for $u$.
\begin{lem}\label{prop_harnack}
There exists a universal constant $c>0$, such that any
non-negative solution to~\eqref{eq:Yt} satisfies
\[
u(r,t) \geq c\gamma^3 \int_{I} \int_{t_0}^{t_0+\gamma^{-2}} u(r,t) dt dr \quad\text{ for all }r\in I, t\in [t_0+\gamma^{-2},t_0 + 2\gamma^{-2}],
\]
for all intervals $I \subset ({1}/{2}, 1)$
with $|I| = 2\gamma^{-1}$, and $t_0 \geq \gamma^{-2}$.
%Suppose that at time $t_0$ we have $u(r,t_0) \geq 0$, and $\int_I u(r,t_0) dr \geq A$, where $I \subset (\frac{1}{2},1)$ is an interval of length $\frac{1}{\gamma}$. Then for every $t \in (t_0 + \frac{1}{\gamma^2}, t_0 + \frac{3}{\gamma^2})$, we have $u(r,t) \geq \frac{cA}{|I|}$ for all $r \in I$.
\end{lem}

\begin{proof}
%Note that
%$
%|H_{rr}| = |\frac{1}{r^2} \int_0^r \eta(s,t) s ds - f(r,t)| \leq \gamma
%$, and $\epsilon \eta \leq \epsilon \gamma$.
Let us rescale \eqref{eq:Yt}, setting $y = \gamma r, \tau =\gamma^2  (t-t_0)$. In the new coordinates, $u$ satisfies
\[
u_\tau - u_{yy} + b(y) u_y + c(y,\tau) u = 0,
\]
where $|b(y)| \leq C$ and $|c(y, \tau)| \leq C$ for
all $y\in ({\gamma}/{2},\gamma), \tau\geq 0$.
The bounds on $b$ and $c$ follow from the facts
that $r\in ({1}/{2},1)$, $|H_r| \leq C\gamma, |H_{rr}| \leq C\gamma$,
$\rho_2 \leq \|\rho_2(\cdot,0)\|_\infty\leq \theta$,
and
\begin{equation}\label{112619} \frac{\epsilon \theta}{\gamma^2} = \frac{\epsilon}{\chi \gamma} \leq c^{-1} \end{equation}
(where $c$ is from Theorem~\ref{main1}).

By the parabolic Harnack inequality (e.g. \cite[Theorem 6.27 or Corollary 7.42]{lieberman}),
for any interval $I' \subset ({\gamma}/{2},\gamma)$ with length $2$, we have
\[
u(y,\tau) \geq C \int_{I'} \int_0^1 u(y,t) dt dy \quad \text{ for all }y \in I', \tau\in[1,2];
\]
here the constant $C$ depends on $c$ in \eqref{112619}.
Translating this back into the original coordinates finishes the proof.
\end{proof}

Consider the time intervals $J_k := [2k\gamma^{-2}, 2(k+1)\gamma^{-2}]$,
$k\in \mathbb{N}$, and let $n$ be the smallest integer such that $2(n+1)\gamma^{-2} \geq t_3$.
%so that $[0,t_3]\subset \cup_{0\leq k\leq n} J_k$.
Then for any interval $I \subset ({1}/{2},1)$ with $|I| = 2\gamma^{-1}$, we can rewrite \eqref{eq:int_average} as
\begin{equation}\label{int_average2}
\sum_{k=0}^n \int_{I\times J_k} \partial_r M(r,t) dr dt \geq \frac{cM_0}{\gamma^2},
\end{equation}
while Lemma \ref{prop_harnack} gives, for each $k\ge 0$:
%and the fact that $u=\partial_r M$ yield that
\begin{equation}\label{m_jk}
%\inf_{r\in I, t\in J_{k+1}}
\partial_r M(r,t) \geq c\gamma^3
\int_{I\times J_k} \partial_r M(s,t) ds dt \quad\text{ for all $r\in I$ and
$t\in J_{k+1}$,} %k \geq 0.
\end{equation}
so that
\[
\int_{J_{k+1}}\partial_r M(r,t)dt\geq c\gamma \int_{I\times J_k} \partial_r M(s,t) ds dt.
\]	
%We then sum up this inequality for $0\leq k\leq n$,
%and combine it with \eqref{int_average2} to obtain
It follows that for each $r\in I$ we have
\begin{equation}\label{oct2802}
\begin{aligned}
%\inf_{r\in I}
&\int_0^{t_3+1} \partial_r M(r,t) dt \geq\int_0^{(n+2)\gamma^{-2}} \partial_r M(r,t)\,dt
=\sum_{k=0}^{n} \int_{J_{k+1}}\partial_r M(r,t)\,dt\\
&\geq %\inf_{r\in I}
%\sum_{k=0}^{n} \gamma^{-2} \inf_{t\in J_{k+1}} \partial_r M dt
%\geq
c \gamma \sum_{k=0}^{n} \int_{I \times J_k} \partial_r M(s,t) \,dsdt \geq \frac{c M_0}{\gamma}.
\end{aligned}
\end{equation}
%Note that $(n+2)\gamma^{-2} \leq t_3+1$, which follows from the definition of $n$ and the fact that $\gamma \gg 1$.
%Hence
%\[
%\int_0^{t_3+1} \partial_r M(r,t) dt \geq \frac{cM_0 }{\gamma}  \text{ for all }r\in I,\]
% and
%since $I \subset (\frac{1}{2},1)$ is an arbitrary interval with length $2\gamma^{-1}$, the
%above inequality holds for all $r\in (\frac{1}{2},1)$, which
This finishes the proof of Proposition \ref{prop:mass_eaten}.

%\section{Caught up in doubt: an example of chemotaxis slowdown}\label{zex}

\section{Discussion}\label{Disc}

In this section, we briefly discuss the nature of the constraints in our main nonlinear application.
The arguments here are purely heuristic, though some of the statements can be made rigorous.
Observe that for $H(x)= \gamma (-\Delta)^{-1}\chi_{B_1}(x),$ the ground state is
\begin{equation}\label{biggs526}
e^H = \left\{ \begin{array}{ll} e^{{\gamma}(1-r^2)/4} & r <1 \\
r^{-{\gamma}/{2}}  & r \geq 1. \end{array} \right.
\end{equation}
A simple calculation shows that {for $r<1$ we have}
\[
\int_{B_r} e^H \,dx = \frac{4\pi}{\gamma}e^{\gamma/4} \left( 1 - e^{- {\gamma} r^2/4} \right),
\]
while
\[ \int_{(B_r)^c} e^H \,dx = \frac{4\pi}{\gamma}e^{\gamma/4}
\left(e^{-\gamma r^2/4}- e^{-\gamma/4} \right) + \frac{4\pi}{\gamma-4}. \]
Therefore, most of the mass of $e^H$ is concentrated in
a ball of radius $\sim \gamma^{-1/2}$
centered at the origin.
% if we take a ball with radius scaling with
%a larger power of $\gamma,$ the mass in its complement
%is smaller by an exponential in some power of $\gamma$ factor.

This explains why the radial constraint on the initial conditions
is needed to make touch with the heuristics. Indeed, consider $\rho_1$ that
is concentrated
initially at a distance $L$ away from the support of $\rho_2,$ in
a region of size $\sim 1$ (as opposed to radial).
If $\gamma$ is large, as this mass gets transported towards the origin,
it will enter the support of $\rho_2$ -- the unit ball centered at the origin --
through a narrow sector and then concentrate overwhelmingly in a tiny region near the origin.
After a time $\sim L^2/\gamma,$ the density $\rho_1$ will approximate $e^H$
given by \eqref{biggs526} since not much reaction has happened during
the passage through a narrow sector. Thus, even after the transport phase has taken place, the reaction rate is going to be penalized since $\rho_1$ is smaller than $M_0$
by a factor that is exponential in $\gamma$ on most of the support of $\rho_2.$ As $\rho_2$ gets depleted near the origin the potential and so
the configuration of $\rho_1$ will adjust, but this process is not straightforward to control. It seems clear that some essential extra time will be lost.

A similar issue applies in the "risky" regime $\epsilon M_0\ll 1$,
even in the radial case. Then, little reaction happens on
the pass through, while the reaction after the transport stage
incurs the same penalty due to the aforementioned excessive concentration.

Both of these constraints are due to an artifact of the specific form of the Keller-Segel chemotaxis term. The extreme concentration of $e^H$
can be seen as a consequence of the scaling $\chi \nabla (-\Delta)^{-1}\rho_2 \sim -\gamma$ near and on the support of $\rho_2,$ which is very large
when $\gamma$ is large. But in reality, there is always a speed limit on how fast biological agents can move. A
variation of the classical Keller-Segel model is the so-called flux limited chemotaxis system given by
\begin{equation}\label{flks}
\partial_t \rho_1 + \chi\nabla \cdot \left(\rho_1 \frac{\nabla c}{|\nabla c|}\psi(|\nabla c|)\right)-\Delta \rho_1 = -\epsilon \rho_1 \rho_2, \,\,\,c= (-\Delta)^{-1}\rho_2,\,\,\,
\,\,\,\,\partial_t \rho_2 = -\epsilon \rho_1 \rho_2.
\end{equation}
The function $\psi$ appearing in \eqref{flks} satisfies $\psi(0)=0,$ is monotone increasing, and saturates at some level that we can take equal to one (given that we have
an explicit coupling constant $\chi$). The system \eqref{flks} is more complex to analyze
due to the strongly nonlinear flux, but is more realistic.
A variety of flux limited Keller-Segel systems have been considered
recently in many works (see e.g. \cite{BW,HP} for more references);
in particular, papers \cite{DS,JV,PVW} provided derivation of the flux limited Keller-Segel system from kinetic models built on biologically
reasonable assumptions about the behavior of the modeled organisms.

In future work, we plan to adapt the techniques developed in this paper to analyze  \eqref{flks}. The adaptation is not straightforward,
but preliminary computations show that in this case the radial assumption is not necessary, and the case of the "risky" reaction can be handled.
%On the other hand, there are indications that even though radial assumptions is too harsh, there
%Another interesting question we plan to address in a future work is the dependence of the chemotactic effect on the reaction rates for various %configurations of the
%initial data.

\appendix
\section{Appendix}
\subsection{Proof of Theorem \ref{thm:L_infty_bound}}
%\textcolor{blue}{
\begin{proof}
Let us first assume that $\rho_0$ is non-negative.
The proof is almost identical to \cite[Theorem 5]{carlenloss}, which we include here for the sake of completeness. Let $r(t)$ be a $C^1$ increasing function to be specified later. We compute the time evolution of $\|\rho(t)\|_{r(t)}$ as follows, where we omit the $t, x$ dependence on the right hand side for notational simplicity:
%\begin{equation}\label{evo_t1}
%\begin{split}
%\frac{d}{dt} \int_{\mathbb{R}^d} \rho(x,t)^{r(t)} dx &= \dot r \int \rho^r \ln \rho dx + r \int \rho^{r-1} (\Delta \rho - \nabla\cdot (\rho\nabla H)) dx\\
%&\leq \frac{\dot r}{r} \int \rho^r \ln (\rho^r) dx - \frac{4(r-1)}{r} \int  |\nabla \rho^{\frac{r}{2}}|^2 dx + (r-1)\gamma  \int \rho^r dx
%\end{split}
%\end{equation}
\begin{equation}\label{evo_t1}
\begin{split}
\frac{d}{dt} \|\rho(\cdot,t)\|_{r(t)} &= -\frac{r'}{r^2} \|\rho\|_r \ln(\|\rho\|_r^r) + \frac{r'}{r} \|\rho\|_r^{1-r}\int \rho^r \ln \rho dx + \|\rho\|_r^{1-r} \int \rho^{r-1}\partial_t \rho dx\\
&= \frac{r'}{r^2}  \|\rho\|_r^{1-r} \int \rho^r \ln\left(\frac{\rho^r}{\|\rho\|_r^r}\right) dx +  \|\rho\|_r^{1-r} \int \rho^{r-1}(\Delta \rho - \nabla\cdot (\rho\nabla \Phi) - h\rho) dx\\
&\leq \frac{r'}{r^2}\|\rho\|_r^{1-r} \left(  \int \rho^r \ln\left(\frac{\rho^r}{\|\rho\|_r^r}\right) dx - \frac{4(r-1)}{r'} \int  |\nabla \rho^{\frac{r}{2}}|^2 dx\right) + \frac{(r-1)}{r}  \gamma \|\rho\|_r,
\end{split}
\end{equation}
where in the last inequality we use the assumptions $h(\cdot,t)\geq 0$ and $\Delta \Phi(\cdot,t) \geq -\gamma$ for all $t$, as well as the fact that $\rho$ remains non-negative for all $t\geq 0$.
Next we use a sharp form of the logarithm Sobolev inequality in $\mathbb{R}^n$. It is equation (7.17) in \cite{carlenloss}, and it is equivalent to Gross's logarithmic Sobolev inequality in \cite{gross} after a scale transformation. For all $f \in H^1(\mathbb{R}^d)$, the following holds for all $a>0$:
\begin{equation}\label{logsobolev}
\int_{\mathbb{R}^d} f^2 \ln\left(\frac{f^2}{\|f\|_2^2} \right) dx + \left(d + \frac{d}{2}\ln a \right) \int_{\mathbb{R}^d} f^2 dx \leq \frac{a}{\pi} \int_{\mathbb{R}^d} |\nabla f|^2 dx.
\end{equation}
Choosing $f = \rho^{r/2}$ and $a = \frac{4\pi(r-1)}{r'}$, \eqref{logsobolev} becomes
\[
 \int_{\mathbb{R}^d} \rho^r \ln\left(\frac{\rho^r}{\|\rho\|_r^r}\right) dx +  \left(d + \frac{d}{2} \ln\left(\frac{4\pi(r-1)}{r'}\right) \right) \|\rho\|_r^r \leq  \frac{4(r-1)}{r'} \int_{\mathbb{R}^d} |\nabla \rho^{\frac{r}{2}}|^2 dx.
\]
Applying this to \eqref{evo_t1} gives us
\[
\frac{d}{dt} \|\rho(t)\|_{r(t)} \leq  \frac{r'}{r^2}\|\rho\|_r \left( - d - \frac{d}{2} \ln\left(\frac{4\pi(r-1)}{r'}\right)\right) +  \frac{(r-1)}{r}  \gamma \|\rho\|_r .
\]
Let $G(t) := \ln \|\rho(t)\|_{r(t)}$. Then the above differential inequality becomes
\begin{equation}\label{diff_ineq_G}
\frac{dG}{dt} \leq \frac{r'}{r^2} \left( - d - \frac{d}{2} \ln\left(\frac{4\pi(r-1)}{r'}\right)\right) +  \frac{(r-1)}{r}  \gamma.
\end{equation}
Since our goal is to estimate $\|\rho(T)\|_\infty$ using $\|\rho(0)\|_1$ (where $T>0$ is an arbitrary time at which we want to obtain our estimate), let us set $r(0) = 1$ and $r(T) = p$, where $p>1$ will be sent to infinity at the end. Integrating \eqref{diff_ineq_G} in $[0,T]$ yields
\[
\begin{split}
\ln\left(\frac{\|\rho(T)\|_p}{\|\rho(0)\|_1}\right) &= G(T)-G(0) \leq \int_0^T \left( \frac{r'}{r^2} \left( - d - \frac{d}{2} \ln\left(\frac{4\pi(r-1)}{r'}\right)\right) +  \frac{(r-1)}{r}  \gamma \right)\, dt\\
&\leq -\int_0^T s' \left( - d - \frac{d}{2} \ln(4\pi(s-s^2)) + \frac{d}{2} \ln (-s') \right) dt + \gamma T  \quad (\text{let }s(t) := \frac{1}{r(t)})\\
&\leq \int_1^{\frac{1}{p}}  \left(d + \frac{d}{2} \ln(4\pi(s-s^2)) \right) ds  + \frac{d}{2} \int_0^T (- s') \ln (-s') dt + \gamma T.
\end{split}
\]
The first integral on the right hand side can be explicitly computed, and it is uniformly bounded by some constant $C(d)$ as $p\to \infty$. For the second integral, since $\int_0^T (-s') dt$ is fixed as $s(0)-s(T) = 1-\frac{1}{p}$, Jensen's inequality gives that the integral is minimized when $-s'$ is a constant. We thus set $-s' = \frac{1-\frac{1}{p}}{T}$, which yields
\[
\ln\left(\frac{\|\rho(T)\|_p}{\|\rho(0)\|_1}\right) \leq C(d) + \frac{d}{2} \left(1-\frac{1}{p}\right) \ln\left( \frac{ 1-\frac{1}{p} }{T}\right) + \gamma T,
\]
hence in the limit $p\to\infty$ we obtain
\begin{equation}\label{infty_norm_temp}
\|\rho(T)\|_\infty \leq C(d) T^{-\frac{d}{2}} e^{\gamma T} \|\rho(0)\|_1 \quad\text{ for all }T> 0.
\end{equation}
Note that $t^{-\frac{d}{2}} e^{\gamma t} $ reaches its minimum value $(\frac{2\gamma}{d})^{\frac{d}{2}} e^{\frac{d}{2}}$ at $t = \frac{d}{2\gamma}$. For $t\geq  \frac{d}{2\gamma}$, by applying the estimate \eqref{infty_norm_temp} with $t-\frac{d}{2\gamma}$ as the initial time (and using the fact that $\|\rho(t-\frac{d}{2\gamma})\|_1 = \|\rho(0)\|_1$), we obtain $\|\rho(t)\|_\infty \leq C(d) \gamma^{d/2}\|\rho_0\|_1$ for all $t\geq \frac{d}{2\gamma}$.  Combining this with \eqref{infty_norm_temp} gives
\[
\|\rho(t)\|_\infty \leq C(d) \max\{ t^{-d/2}, \gamma^{d/2}\} \|\rho(0)\|_1 \quad\text{ for all }t> 0.
\]

To establish the theorem for the case of sign changing $\rho_0,$ notice that the equation \eqref{eq:fp} is linear. Thus we can run the evolution separately for
the positive and negative parts of the initial data, and both solutions will satisfy \eqref{oct802}. By linearity, the actual solution of \eqref{eq:fp} is just the difference
of these two solutions and \eqref{oct802} clearly holds for it as well.

\end{proof}
%}

{\bf Acknowledgement}. \rm
The authors acknowledge partial support of the NSF-DMS grants 1715418, 1846745, 1848790, 1900008 and 1910023.
We are grateful to Andrej Zlatos for insightful discussions.

\end{document}